\documentclass{cpamC}     
\received{January 2020}       \revised{July 2020}
\pubyear{\null} \copyrightyear{2021}
\volume{\null}
\startingpage{1}                      
\usepackage{cite,upref,breakurl,mleftright}
\usepackage[]{hyperref}
\newcommand{\<}{\langle}
\renewcommand{\>}{\rangle}
\renewcommand\tilde\wtilde
\renewcommand\hat\what
\input cmacsA

\authorheadline{B.\ Engquist and Y.\ Yang}
\titleheadline{Optimal Transport for FWI}

\usepackage{mathrsfs}
\usepackage{subfig}
\usepackage{enumerate}
\usepackage{graphicx}
\usepackage{multirow}
\usepackage{color}
\usepackage{xspace}
\usepackage{pdfsync}
\usepackage{mathtools}
\usepackage{algorithmicx}
\usepackage{algpseudocode}
\usepackage{algorithm}

\graphicspath{{Fig/}}
\newtheorem{theorem}{Theorem}[section]

\newtheorem{corollary}[theorem]{Corollary}

\theoremstyle{definition}
\newtheorem{definition}[theorem]{Definition}
\newtheorem{assumption}[theorem]{Assumption}
\theoremstyle{remark}
\newtheorem{remark}[theorem]{Remark}

\DeclareMathOperator*{\argmin}{argmin}
\DeclareSubrefFormat{}{}
\newcommand{\bq}{\begin{equation}}
\newcommand{\eq}{\end{equation}}
\newcommand{\R}{\mathbb{R}}

\newcommand{\bO}{\mathcal{O}}

\newcommand{\tr}{\text{tr}}

\newcommand{\MA}{Monge-Amp\`ere\xspace}
\newcommand{\df}{\delta f}

\newcommand{\Lf}{\mathcal{L}}
\renewcommand{\diag}{\text{diag}}

\newcommand*\Laplace{\mathop{}\!\mathbin\bigtriangleup}

\algnewcommand{\LineComment}[1]{\State \(\triangleright\) #1}

\begin{document} 
\mleftright


\title{Optimal Transport Based Seismic Inversion:\\Beyond Cycle
  Skipping}

\author{Bj\"orn Engquist}{The University of Texas at Austin}
\author{Yunan Yang}{New York University}





\begin{abstract}
  Full-waveform inversion (FWI) is today a standard process for the
  inverse problem of seismic imaging. PDE-constrained optimization is
  used to determine unknown parameters in a wave equation that
  represent geophysical properties. The objective function measures
  the misfit between the observed data and the calculated synthetic
  data, and it has traditionally been the least-squares norm. In a
  sequence of papers, we introduced the Wasserstein metric from
  optimal transport as an alternative misfit function for mitigating
  the so-called cycle skipping, which is the trapping of the
  optimization process in local minima. In this paper, we first give a
  sharper theorem regarding the convexity of the Wasserstein metric as
  the objective function.  We then focus on two new issues. One is the
  necessary normalization of turning seismic signals into probability
  measures such that the theory of optimal transport applies. The
  other, which is beyond cycle skipping, is the inversion for
  parameters below reflecting interfaces.  For the first, we propose a
  class of normalizations and prove several favorable properties for
  this class. For the latter, we demonstrate that FWI using optimal
  transport can recover geophysical properties from domains where no
  seismic waves travel through. We finally illustrate these properties
  by the realistic application of imaging salt inclusions, which has
  been a significant challenge in exploration geophysics.\\
\end{abstract}

\maketitle






\section{Introduction}
The goal in seismic exploration is to estimate essential geophysical
properties, most commonly the wave velocity, based on the observed
data. The development of human-made seismic sources and advanced
recording devices facilitate seismic inversion using entire wavefields
in time and space rather than merely travel time information as in
classical seismology. The computational technique full-waveform
inversion (FWI)~\cite{lailly1983seismic,TarantolaInversion} was a
breakthrough in seismic imaging, and it follows the established
strategy of a partial differential equation (PDE) constrained
optimization. FWI can achieve results with stunning clarity and
resolution~\cite{Virieux2017}. Unknown parameters in a wave equation
representing geophysical properties are determined by minimizing the
misfit between the observed and PDE-simulated data, i.e.,
\bq \label{eq:fwi} m^* = \argmin\limits_m \{J(f(m), g)+ \mathscr R(m)
\}, \eq where $m$ is the model parameter, which can be seen as a
function or discrete as a vector. The misfit $J$ is the objective
function, $f(m)$ is the PDE-simulated data given model parameter $m$,
$g$ is the observed data, and $\mathscr R(m)$ represents the added
regularization. In both time~\cite{TarantolaInversion} and frequency
domain~\cite{pratt1990inverse1}, the least-squares norm
$J(f,g)=\|f-g\|^2_2$ has been the most widely used misfit function,
which we will hereafter denote as $L^2$. In this paper, we focus on
the effects of a different choice of the objective function $J$ and
avoid adding any regularization $\mathscr R(m)$. This is to focus on
the properties of the misfit function.

\begin{figure}
  \subfloat[Observed data
  $g$]{\includegraphics[width=0.33\textwidth]{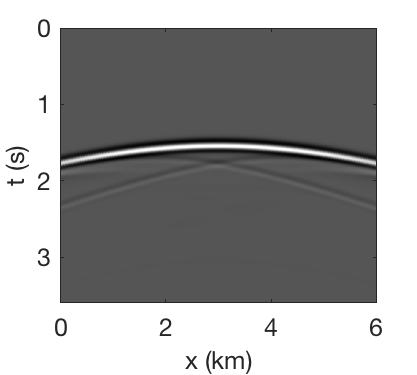}\label{fig:ckdata1}}\hfill
  \subfloat[Synthetic data $f$]
  {\includegraphics[width=0.33\textwidth]{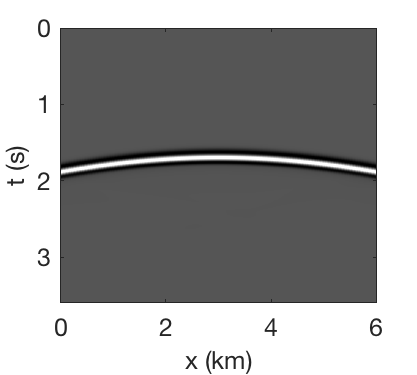}\label{fig:ckdata2}}\hfill
  \subfloat[The difference $f-g$]
  {\includegraphics[width=0.33\textwidth]{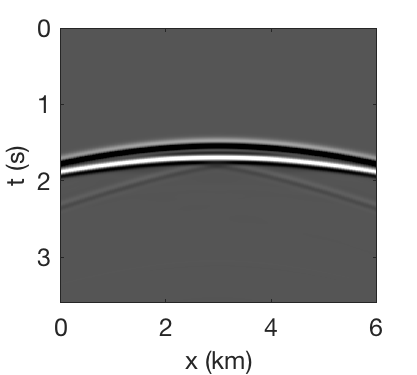}\label{fig:ckdiff}}\\
  \subfloat[Comparison of the traces at $x=3$ km]
  {\includegraphics[width=0.95\textwidth]{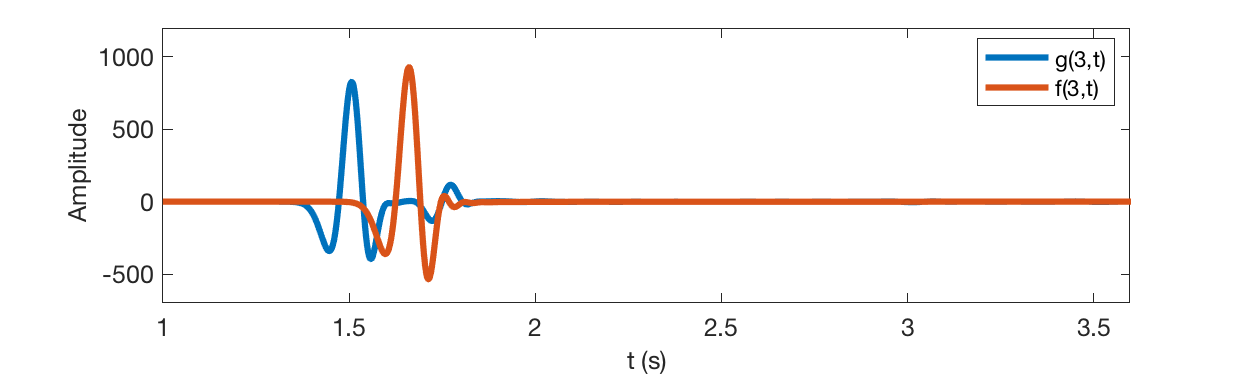}\label{fig:ckcomp}}
  \caption{(a)~The observed data $g$; (b)~the synthetic data $f$;
    (c)~the difference $f-g$; (d)~the comparison between two traces at
    $x = 3\,$km, i.e., $f(3,t)$ and $g(3,t)$. Inversion using the
    $L^2$ norm suffers from cycle-skipping issues in this scenario,
    which is further discussed in
    Section~\ref{sec:cheese}.}\label{fig:ck}
\vspace{-1\baselineskip}
\end{figure}

It is, however, well-known that FWI based on the $L^2$~norm is
sensitive to the initial model, the data spectrum, and the noise in
the measurement~\cite{virieux2009overview}.  Cycle skipping can occur
when the phase mismatch between the two wavelike signals is greater
than half of the wavelength. The fastest way to decrease the
$L^2$~norm is to match the next cycle instead of the current one,
which can lead to an incorrectly updated model parameter. It is a
dominant type of local-minima trapping in seismic inversion due to the
oscillatory nature of the seismic waves. For example,
Figure~\ref{fig:ckdata1} and Figure~\ref{fig:ckdata2} show two
datasets, which are 2D wavefields measured at the upper
boundary. Their phase mismatch is about one wavelength;
see~Figure~\ref{fig:ckdiff} and~Figure~\ref{fig:ckcomp}. $L^2$-based
inversion for this case suffers from cycle-skipping issues. We will
discuss inversions based on this example in
Section~\ref{sec:cheese}. When the velocity in the model is too slow
compared to the true value, the simulated signal will arrive later
than the observed true signal. It is, therefore, natural to study the
effects on the mismatch from shifts between the signals.

\begin{figure}
  \centering \subfloat[2nd order derivative of
  Gaussian]{\includegraphics[width=0.45\textwidth]{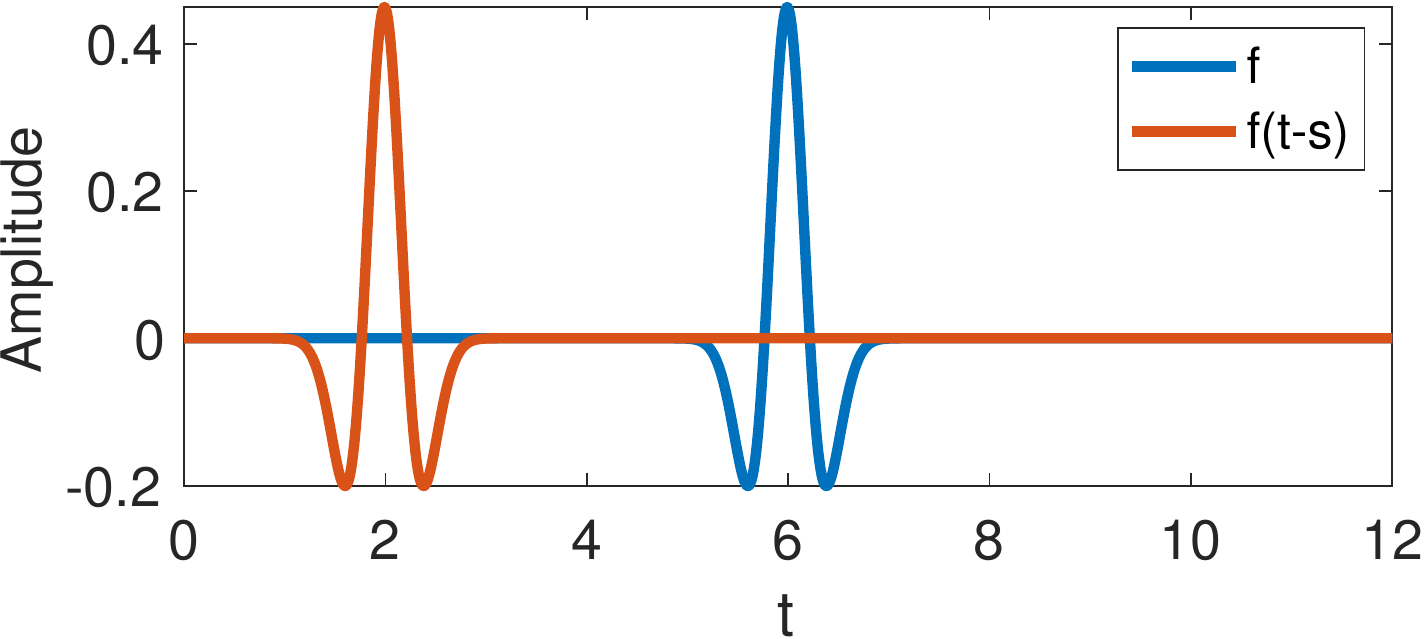}\label{fig:2_ricker_signal}}\hspace{1cm}
  \subfloat[4th order derivative of Gaussian]{\includegraphics[width=0.45\textwidth]{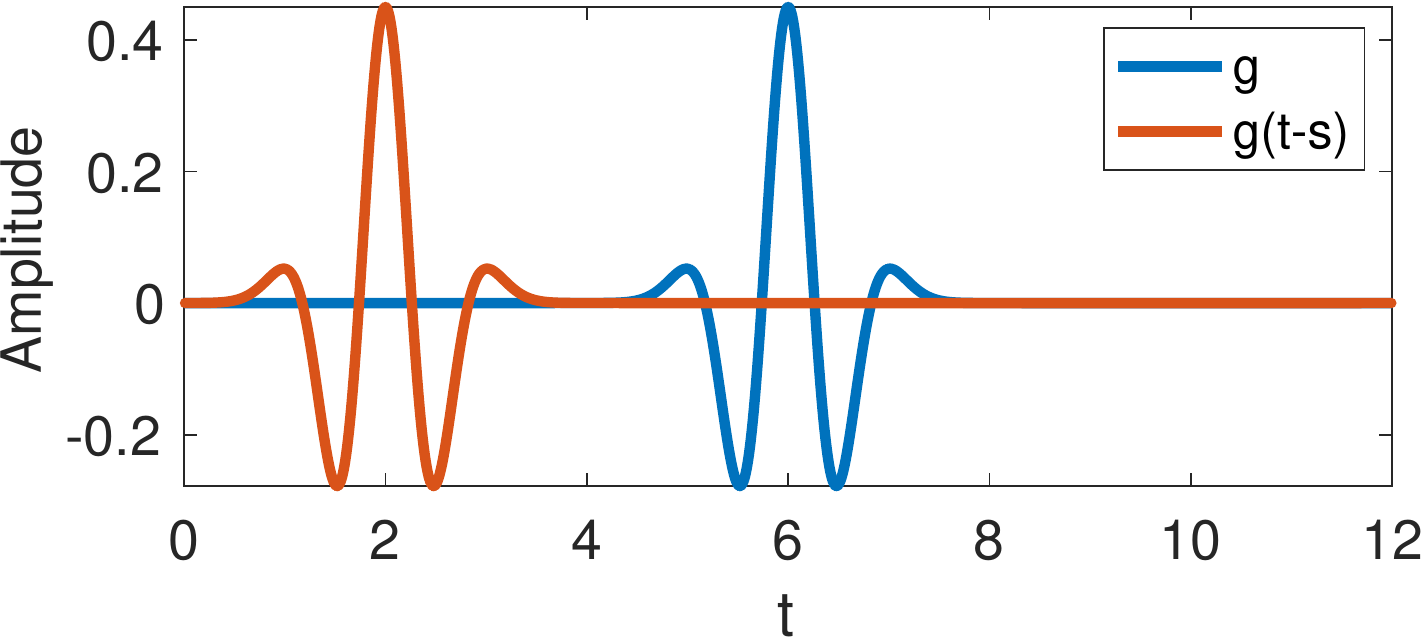}\label{fig:2_ricker_signal2}}\\
  \subfloat[$L^2$ misfit]
  {\includegraphics[width=0.45\textwidth]{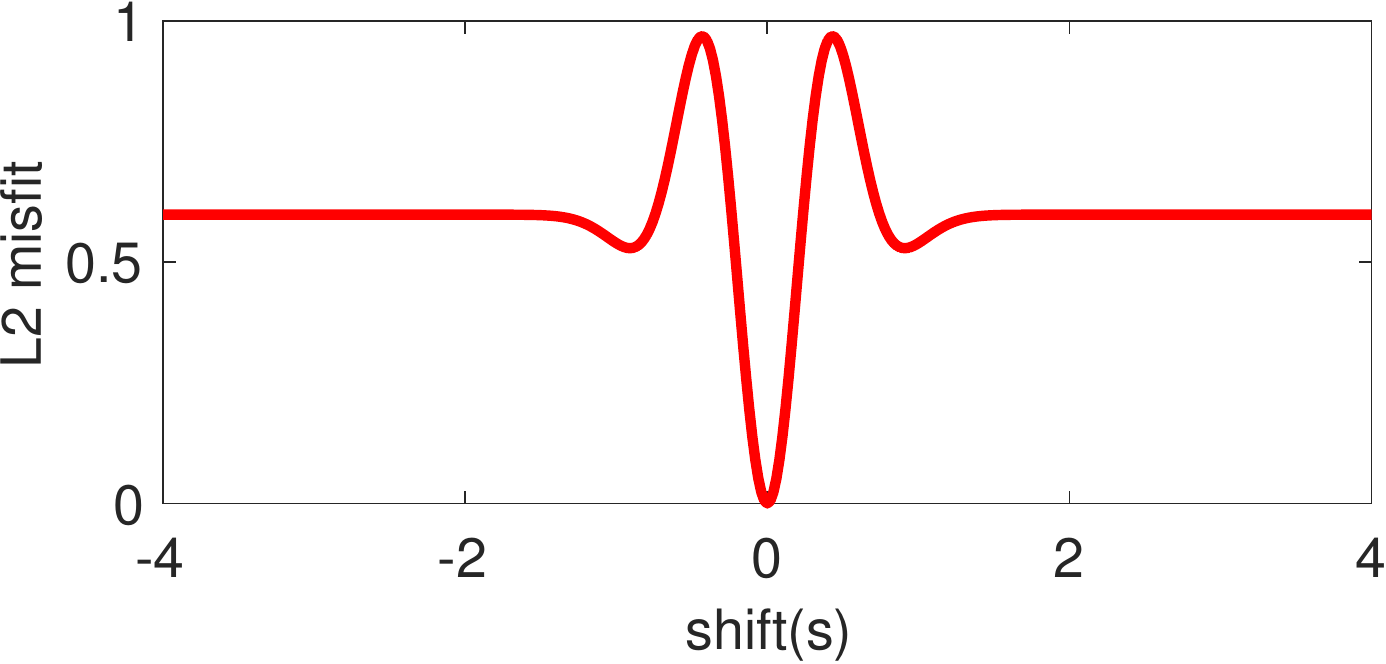}\label{fig:2_ricker_L2}}\hspace{1cm}
  \subfloat[$W_2$ misfit]
  {\includegraphics[width=0.45\textwidth]{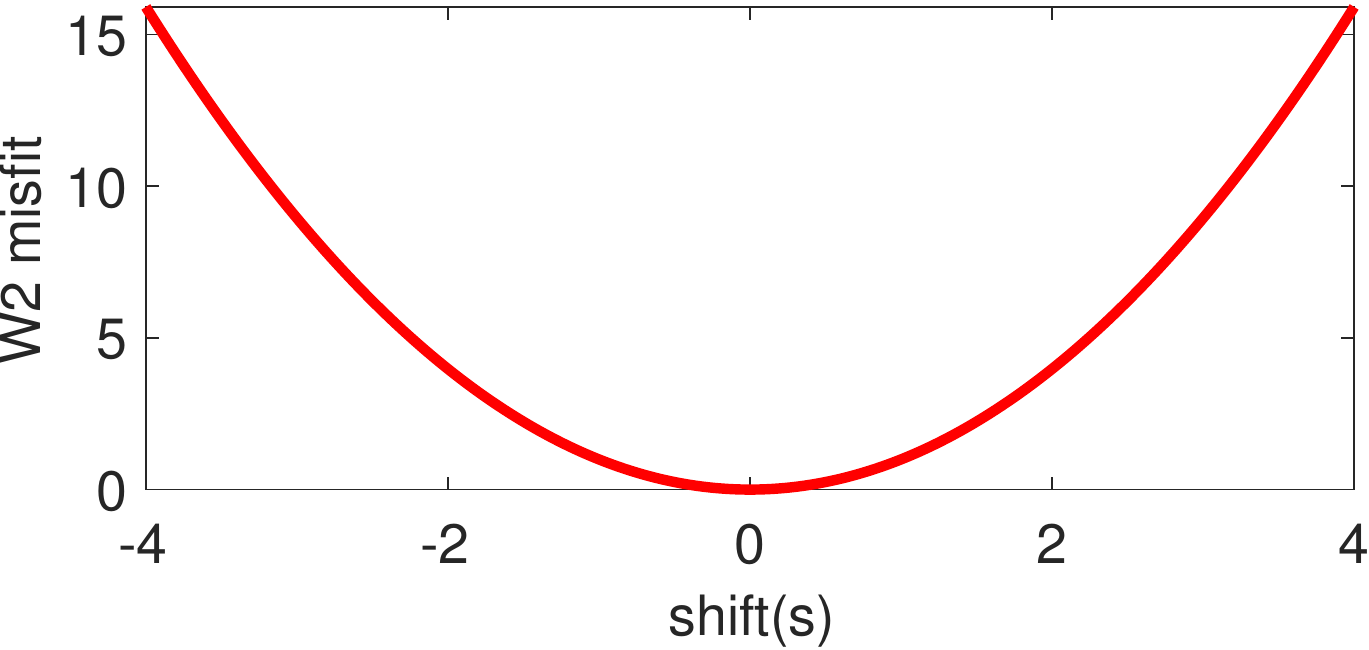}\label{fig:2_ricker_W2}}\\
  \subfloat[$L^2$ misfit]
  {\includegraphics[width=0.45\textwidth]{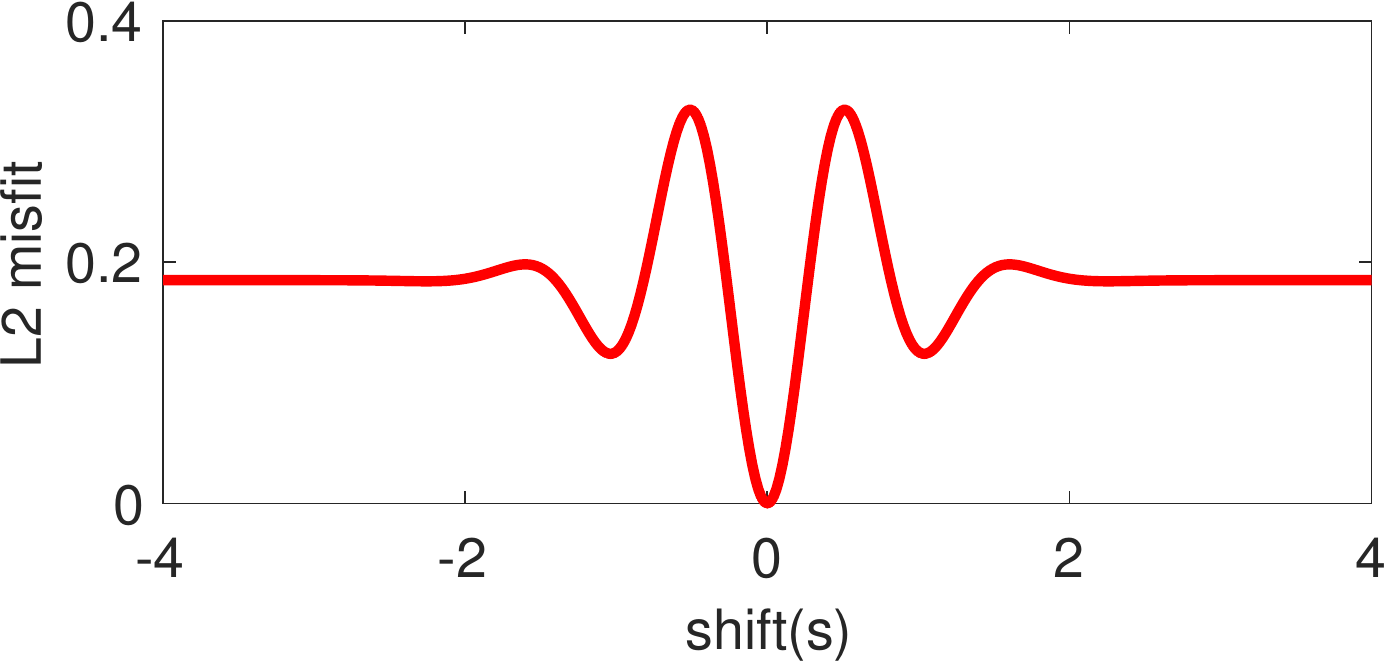}\label{fig:2_ricker_L22}}\hspace{1cm}
  \subfloat[$W_2$ misfit]
  {\includegraphics[width=0.45\textwidth]{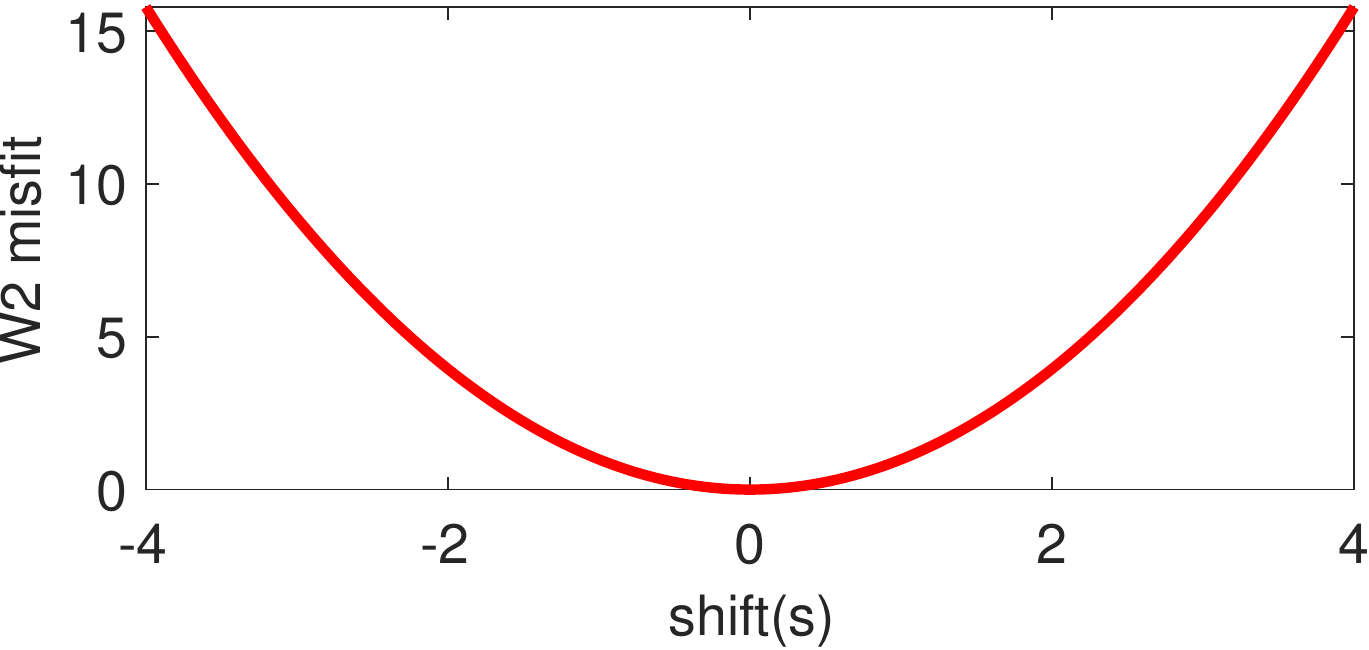}\label{fig:2_ricker_W22}}
  \caption{(a)~Ricker wavelet $f(t)$ and its shift $f(t-s)$;
    (b)~signal $g(t)$ and its shift $g(t-s)$;~(c)~the $L^2$ norm
    and~(d)~the $W_2$ distance between $f(t)$ and $f(t-s)$ in terms of
    shift $s$;~(e)~the $L^2$ norm and~(f)~the $W_2$ distance between
    $g(t)$ and $g(t-s)$ as a function of $s$. The exponential
    scaling~\eqref{eq:exp} is applied to normalize the signals for the
    $W_2$ computation; see
    Section~\ref{sec:Data_Normalization}.}\label{fig:2_ricker_all}
\end{figure}

The quadratic Wasserstein distance ($W_2$) is ideal for dealing with
this type of problem as it has perfect convexity with respect to
translation and dilation (Theorem~\ref{thm: biconvex}). We will prove
the theorem in Section~\ref{sec:cvx} in a new and more general
form. The convexity was the original motivation for us to introduce
the Wasserstein distance as the misfit function in seismic
inversion~\cite{EFWass, engquist2016optimal}, which has now been
picked up and developed in parallel by many other research groups and
generalized to other Wasserstein metric beyond
$W_2$~\cite{W1_2D,W1_3D,chen2017quadratic,sun2019stereo}. Figure~\ref{fig:2_ricker_all}
is an illustration of the convexity of the $W_2$ metric regarding data
translation. The signals used in this example are the Ricker
wavelet~\cite{ricker1943further} and the more oscillatory fourth-order
derivative of Gaussian. The $W_2$ metric precisely captures the data
shift $s$ as the misfit is $s^2$ independent of the profile of the
signal. It is not the case for the $L^2$ norm whose basin of
attraction depends on the spectral property of the signal.

Although the original goal of introducing optimal transport was to
mitigate cycle skipping as discussed above, we will see in this paper
that there are other good properties beyond reducing cycle
skipping. These properties will be divided into two parts. The first
is the effect of data normalization~\cite{Survey2}, which is an
essential preprocessing step of transforming seismic signals to
probability densities. In optimal transport theory, there are two main
requirements for functions $f$ and $g$ that are compactly supported on
a domain $\Pi$:
\[
  f\geq 0,\quad g\geq 0,\quad \<f\>  = \int_{\Pi} f =\int_{\Pi} g = \<g\>  .
\]
Since these constraints are not expected for seismic signals,
different approaches have been promoted to tackle this issue. There
exists a significant amount of work from the applied mathematics
community to generalize optimal transport to the so-called unbalanced
optimal transport
\cite{zhou2018wasserstein,gangbo2019unnormalized,chizat2015unbalanced}. Regarding
the nonpositivity, the Kantorovich-Rubinstein norm was proposed to
approximate the 1-Wasserstein metric~\cite{W1_2D}, and mapping seismic
signals into a graph space by increasing the dimensionality of the
optimal transport problem by $1$ is also demonstrated to be a feasible
solution~\cite{metivier2018optimal,metivier2019graph}.

Another way to achieve data positivity and mass conservation is to
directly transform the seismic data into probability densities by
linear or nonlinear scaling
functions~\cite{qiu2017full,Survey2,Survey1,yang2017application}. In
this paper, we focus on the fundamental properties of such data
normalizations on $W_2$-based
inversion. In~\cite{yang2017application}, we normalized the signals by
adding a constant and then scaling: \bq\label{eq:linear} \wtilde{f} =
\frac{f + b}{\<f+b\>},\quad \tilde{g} = \frac{g + b}{\<g+b\>},\quad b>0.
\eq An exponential based normalization (equation~\eqref{eq:exp}) was
proposed in~\cite{qiu2017full}. Later, the softplus function defined
in equation~\eqref{eq:softplus}~\cite{glorot2011deep} as a more stable
version of the exponential scaling soon became popular in practice:
\begin{equation}\label{eq:exp} \tilde{f} = \frac{\exp(bf) }{\<\exp(bf)\>},\quad
\tilde{g} = \frac{\exp(bg)}{\<\exp(bg)\>},\quad b>0.  \end{equation}
\begin{equation}\label{eq:softplus} \tilde{f} =
\frac{\log(\exp(bf)+1)}{\<\log(\exp(bf)+1)\>},\quad \tilde{g} =
\frac{\log(\exp(bg)+1)}{\<\log(\exp(bg)+1)\>},\quad b>0.  \end{equation} In
Figure~\ref{fig:2_ricker_all}, the exponential
normalization~\eqref{eq:exp} is applied to transform the signed
functions into probability distributions before the computation of the
$W_2$ metric.

We remark that~\eqref{eq:exp} and~\eqref{eq:softplus} suppress the
negative parts of $f$, which works well in most experiments, but it is
also possible to add the objective function with the
normalization~\eqref{eq:exp} and~\eqref{eq:softplus} applied to $-f$
to avoid biasing towards either side. Although the linear
normalization~\eqref{eq:linear} does not give a convex misfit function
with respect to simple shifts~\cite{yangletter}, it works remarkably
well in realistic large-scale
examples~\cite{yang2017application}. Earlier, adding a constant to the
signal was to guarantee a positive function, but empirical
observations motivate us to continue studying the positive influences
of certain data scaling methods. In Assumption~\ref{ass:DN}, we
summarize several essential features for which we can later prove
several desirable properties. This class of normalization methods
allows us to apply the Wasserstein distance to signed signals, which
can thus be seen as a type of \textit{unbalanced optimal transport}.

We prove that the Wasserstein distance is still a metric
$d(f,g) = W_2(\wtilde{f}, \tilde{g})$ in
Theorem~\ref{thm:metric}. Also, by adding a positive constant to the
signals, one turns $W_2$ into a ``Huber-type'' norm
(Theorem~\ref{thm:Huber}). Researchers have studied the robustness of
the Huber norm~\cite{huber1973robust}, which combines the best
properties of the $\ell^2$ norm and the $\ell^1$ norm by being strongly convex
when close to the target or the minimum and less steep for extreme
values or outliers. For seismic inversion, the ``Huber-type'' property
means irrelevant seismic signals that are far apart in time will not
excessively influence the optimal transport map as well as the misfit
function. By adding a positive constant to the normalized data, we
could guarantee the regularity of the optimal map between the
synthetic and the observed data~(Theorem~\ref{thm:map_smooth}) and
consequently enhance the low-frequency contents in the gradient.

The second topic beyond cycle skipping is the remarkable property of
$W_2$ in producing useful information from below the lowest reflecting
interface. This is one part of the Earth from which no seismic waves
return to the surface to be recorded. The most common type of recorded
data in this scenario is seismic reflection. Reflections carry
essential information of the deep region in the subsurface, especially
when there are no transmission waves or other refracted waves
traveling through due to a narrow range of the recording. Conventional
$L^2$-based FWI using reflection data has been problematic in the
absence of a highly accurate initial model. In simple cases, some
recovery is still possible, but it usually takes thousands of
iterations. The entire scheme is often stalled because the
high-wavenumber model features updated by reflections slow down the
reconstruction of the missing low-wavenumber components in the
velocity. This issue can be mitigated by using the $W_2$-based
inversion because of its sensitivity to small amplitudes and the
low-frequency bias.  We will show several tests, including the salt
body inversion, that partial inversion for velocity below the deepest
reflecting layer is still possible by using the $W_2$ metric. It is
another significant advantage of applying optimal transport to seismic
inversion beyond reducing local minima.

The focus in this paper is on the properties of the $W_2$-based
objective function in full-waveform inversion, and the mathematical
analysis here plays an important role. Therefore, the numerical
examples are straightforward using the same wave propagator to
generate both the synthetic and the observed data and without
regularization or postprocessing. In a realistic setting, the wave
source for the synthetic data is only an estimation. There are also
modeling and numerical errors in the wave simulation as well as noisy
data. What makes us confident of the practical value of the techniques
discussed here is the emerging popularity in the industry and
successful application to real field data, which have been
reported~\cite{poncet2018fwi,Ramos2018,pladys2019assessment}. We
include one numerical example to show the robustness of $W_2$-based
inversion in Section~\ref{sec:6} by using a perturbed synthetic
source and adding correlated data noise.

\section{Background}
The primary purpose of this section is to present relevant background
on two important topics involved in this paper, full-waveform
inversion and optimal transport. We will also briefly review the
adjoint-state method, which is a computationally efficient technique
in solving large-scale inverse problems through local optimization
algorithms.

\begin{figure}
  \centering \includegraphics[width=0.6\textwidth]{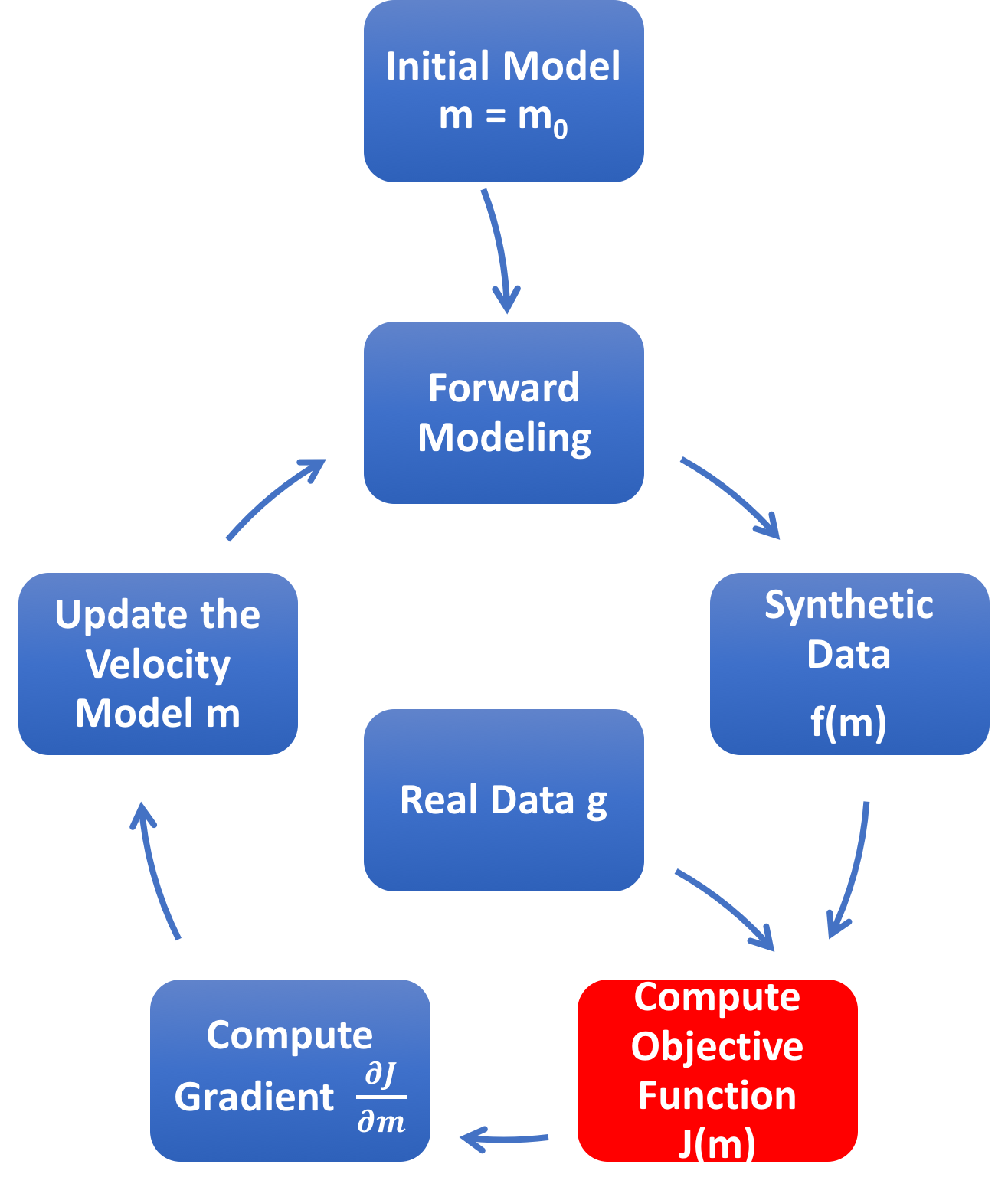}
  \caption{A diagram of FWI as an optimization problem. The step in
    red is where we propose to use the Wasserstein
    metric.}\label{fig:FWI_flow}
\vspace{-3.5ex}
\end{figure}

\subsection{Full-waveform inversion}\label{sec:FWI}
Full-waveform inversion (FWI) is a nonlinear inverse technique that
utilizes the entire wavefield information to estimate subsurface
properties. It is a data-driven method to obtain high-resolution
images by minimizing the difference or misfit between observed and
synthetic seismic waveforms~\cite{virieux2009overview}. The goal of
full-waveform inversion is to estimate the true velocity field through
the solution of the optimization problem in~\eqref{eq:fwi}.  In this
paper, we consider the inverse problem of finding the velocity
parameter of a 2D acoustic wave equation~\eqref{eq:FWD} on the domain
$\Omega=\{(x,z)\in \R^2: z\geq 0\}$ from knowing the wave solution on
the part of the boundary $\partial \Omega=\{(x,z)\in \R^2: z=0\}$.

In Figure~\ref{fig:FWI_flow}, we highlight the major components of the
FWI workflow in each iteration. Starting with an initial velocity
parameter, one forward solve gives the synthetic data $f = Ru$. Here,
$R$ is an extraction operator that outputs the wavefield
$u(\mathbf{x},t)$ on the boundary $\partial \Omega$. $u(\mathbf{x},t)$
is the solution to the following wave equation on spatial domain
$\Omega$ from time $0$ to $T_{\text{max}}$:
\begin{equation}\label{eq:FWD}
  \begin{cases}
       m(\mathbf{x})\frac{\partial^2 u(\mathbf{x},t)}{\partial t^2}- \Laplace u(\mathbf{x},t) = s(\mathbf{x},t),\\
       u(\mathbf{x}, 0 ) = 0, \ \frac{\partial u}{\partial t}(\mathbf{x}, 0 ) = 0& \text{on}\ \Omega,               \\
        \nabla u \cdot \mathbf{n} = 0& \text{on}\ \partial\Omega.    \\
    \end{cases}
  \end{equation}
The model parameter is $m(\mathbf{x}) = \frac{1}{v(\mathbf{x})^2}$
where $v(\mathbf{x})$ is the wave velocity and $s(\mathbf{x},t)$ is
the known source term. It is a linear PDE but a nonlinear map from the
$m(\mathbf{x})$ to $u(\mathbf{x},t)$.
Wave simulation is the most resource-intensive part of the workflow
due to the size of the discretized problem in geophysical
applications. For realistic applications with hundreds of different
wave sources, wave simulations are also done as an embarrassingly
parallel task on supercomputers.  The difference between the synthetic
data $f(m)$ and the real data $g$ is calculated by an objective
function $J(m)$. The conventional choice is the least-squares norm
($L^2$), as we discussed that suffers from cycle skipping and
sensitivity to noise~\cite{Virieux2017}:
\begin{equation*}\label{eq:L2_misfit}
  J_{L^2}(m)=\frac{1}{2} \sum_{i_s} \int\limits_{\partial \Omega} \!\!\!\int\limits_0^{T_{\text{max}}} \abs{f_{i_s}(\mathbf{x},t;m)-g_{i_s}(\mathbf{x},t)}^2dt \,d\mathbf{x}.
\end{equation*}
In practice, the objective function is a summation over multiple
sources where $i_s$ is the index representing the wavefield or data
generated by a different source term $s(\mathbf x,t)$
in~\eqref{eq:FWD}. The summation over the source helps broaden the
bandwidth of the observed data, capture waves from various
illumination angles, and reduce the effects of noise.

The adjoint-state method is used here to compute the FWI gradient of
the objective function; see detailed derivations
in~\cite{Demanet2016,chavent2010nonlinear,Plessix}. For large-scale
PDE-constrained optimizations, the adjoint-state method is a common
practice to efficiently compute the gradient of a function or an
operator numerically~\cite{cao2003adjoint}. It has broad applications
in general inverse problems beyond seismic
imaging~\cite{chavent1975history,vogel2002computational,mcgillivray1990methods}.

Based on the adjoint-state method~\cite{Plessix}, one only needs to
solve two wave equations numerically to compute the Fr\'{e}chet
derivative or gradient with respect to model parameters in FWI. The
first one is the forward propagation~\eqref{eq:FWD}. The second one is
the following adjoint equation on the domain $\Omega$ from time
$T_{\text{max}}$ to $0$:
\begin{equation} \label{eq:FWI_adj} \begin{cases}
       m\frac{\partial^2 w(\mathbf{x},t)}{\partial t^2}- \Laplace w(\mathbf{x},t)  = -R^*\frac{\delta J}{\delta f},\\
       w(\mathbf{x}, T_{\text{max}} ) = 0, \ \frac{\partial w}{\partial t}(\mathbf{x}, T_{\text{max}} ) = 0
       & \text{on}\ \Omega,               \\
        \nabla w \cdot \mathbf{n} = 0& \text{on}\ \partial\Omega.   
    \end{cases}
  \end{equation}
The adjoint equation above requires zero final condition at time
$T_{\text{max}}$. Thus,~\eqref{eq:FWI_adj} is often referred to as
backpropagation in geophysics. Solving the adjoint equation is also
practically done in parallel. Once we have the forward wavefield $u$
and adjoint wavefield $w$, the gradient is calculated as follows:
\begin{equation}~\label{eq:adj_grad3} \frac{\delta J}{\delta m} =- \sum_{i_s}
\int_0^{T_{\text{max}}} \frac{\partial^2
  u_{i_s}(\mathbf{x},t)}{\partial t^2} w_{i_s}(\mathbf{x},t)dt.  \end{equation}
We remark that a modification of the misfit function only impacts the
source term of the adjoint wave
equation~\eqref{eq:FWI_adj}~\cite{W1_2D}.

Using the gradient formula~\eqref{eq:adj_grad3}, the velocity
parameter is updated by an optimization method as the last step in
Figure~\ref{fig:FWI_flow} before entering the next iteration. In this
paper, we use L-BFGS with the backtracking line search following the
Wolfe conditions~\cite{liu1989limited}. The step size is required to
both give a sufficient decrease in the objective function and satisfy
the curvature condition~\cite{wolfe1969convergence}.

In~\cite{yang2017application}, we proposed a trace-by-trace objective
function based on $W_2$. A trace is the time history measured at one
receiver. The entire dataset consists of the time history of all the
receivers. For example, with $\mathbf{x}$ fixed, $f(\mathbf{x},t)$ is
a trace. The corresponding misfit function is \begin{equation} \label{eqn:Wp1D}
J_{W_2}(m) = \frac{1}{2} \sum_{i_s} \int\limits_{\partial \Omega}
W_2^2(f_{i_s}(\mathbf{x},t;m),g_{i_s}(\mathbf{x},t)) d\mathbf{x}, \end{equation}
Mathematically it is $W_2$ metric in the time domain and $L^2$ norm in
the spatial domain.  In Section~\ref{sec:OT}, we will define the $W_2$
metric formally.

\subsection{Optimal transport}\label{sec:OT}
The optimal mass transport problem seeks the most efficient way of
transforming one distribution of mass to the other, relative to a
given cost function. It was first brought up by Monge in
1781~\cite{Monge} and later expanded by
Kantorovich~\cite{kantorovich1960mathematical}.  Optimal
transport-related techniques are nonlinear as they explore both the
signal amplitude and the phases. The significant contributions of the
mathematical analysis of the optimal transport problem since the 1990s
~\cite{Villani} together with current advancements in numerical
methods~\cite{peyre2019computational} have driven the recent
development of numerous applications based on optimal
transport~\cite{kolouri2016transport}.

Given two probability densities $f = d\mu$ and $g=d\nu$, we are
interested in the measure-preserving map $T$ such that
$f = g \circ T$.
\begin{definition}[Measure-preserving map]\label{def:mass_preserve}
  A transport map $T: X \rightarrow Y$ is measure-preserving if for
  any measurable set $B \in Y$,
  \[~\label{eq:mass_preserve1} \mu (T^{-1}(B)) = \nu(B).
  \]
  If this condition is satisfied, $\nu$ is said to be the push-forward
  of $\mu$ by $T$, and we write $\nu = T_\# \mu $.
\end{definition}
If the optimal transport map $T(x)$ is sufficiently smooth and
$\det(\nabla T(x) ) \neq 0$, Definition~\ref{def:mass_preserve}
naturally leads to the requirement
\begin{equation}\label{eq:mass_preserve2} f(x) =
  g(T(x))\det(\nabla T(x)).  \end{equation}
The transport cost function $c(x,y)$
maps pairs $(x,y) \in X\times Y$ to $\mathbb{R}\cup \{+\infty\}$,
which denotes the cost of transporting one unit mass from location $x$
to $y$. The most common choices of $c(x,y)$ include $|x-y|$ and
$|x-y|^2$, which denote the Euclidean norms for vectors $x$ and $y$
hereafter. While there are many maps $T$ that can perform the
relocation, we are interested in finding the optimal map that
minimizes the total cost. If $c(x,y) = |x-y|^p$ for $p \geq 1$, the
optimal transport cost becomes the class of the Wasserstein metric:
\begin{definition}[The Wasserstein metric]
  We denote by $\mathscr{P}_p(X)$ the set of probability measures with
  finite moments of order $p$. For all $p \in [1, \infty)$,
  \begin{equation}\label{eq:static}
    W_p(\mu,\nu)=\left( \inf _{T_{\mu,\nu}\in \mathcal{M}}\int_{\mathbb{R}^n}\left|x-T_{\mu,\nu}(x)\right|^p \, d\mu(x)\right) ^{\frac{1}{p}},\quad \mu, \nu \in \mathscr{P}_p(X).
  \end{equation}
  $T_{\mu,\nu}$ is the measure-preserving map between $\mu$ and $\nu$,
  or equivalently, $(T_{\mu,\nu} )_\# \mu =\nu$. $\mathcal{M}$ is the
  set of all such maps that rearrange the distribution $\mu$ into $\nu$.
\end{definition}

Equation~\eqref{eq:static} is based on Monge's problem for which
the optimal map does not always exist since ``mass splitting'' is not
allowed. For example, consider $\mu = \delta_{1}$ and
$\nu = \frac{1}{2}\delta_{0}+\frac{1}{2}\delta_{2}$, where $\delta_x$
is the Dirac measure. The only rearrangement from $\mu$ to $\nu$ is
not technically a map (function). Kantorovich relaxed the
constraints~\cite{kantorovich1960mathematical} and proposed the
following alternative formulation~\eqref{eq:static2}. Instead of
searching for a function $T$, the transference plan $\pi$ is
considered. The plan is a measure supported on the product space
$X\times Y$ where $\pi(x,y_1)$ and $\pi(x,y_2)$ are well-defined for
$y_1\neq y_2$. The optimal transport problem under the Kantorovich
formulation becomes a linear problem in terms of the plan $\pi$:
\begin{equation}~\label{eq:static2} \inf_{\pi} I[\pi] = \bigg\{ \int_{X \times Y}
c(x,y) d\pi\ |\ \pi \geq 0\ \text{and}\ \pi \in \Gamma(\mu, \nu)
\bigg\} , \end{equation}
where
$\Gamma (\mu, \nu) =\{ \pi \in \mathcal{P}(X\times Y)\ |\ (P_X)_\# \pi
= \mu, (P_Y)_\# \pi = \nu \}$. Here $(P_X)$ and $(P_Y)$ denote the two
projections, and $(P_X)_\# \pi$ and $(P_Y)_\# \pi $ are two measures
obtained by pushing forward $\pi$ with these two projections.

One special property of optimal transport is the so-called c-cyclical
monotonicity. It offers a powerful tool to characterize the general
geometrical structure of the optimal transport plan from the
Kantorovich formulation. It has been proved that optimal plans for any
cost function have c-cyclically monotone
support~\cite{villani2008optimal}. In addition, the concept can be
used to formulate a \textit{sufficient} condition of the optimal
transport plan~\cite{villani2008optimal,brenier1991polar,KnottSmith}
under certain mild assumptions~\cite{ambrosio2013user}. Later, we will
use this property to prove Theorem~\ref{thm: biconvex}, one of the key
results in the paper.\medskip

\begin{definition}[Cyclical monotonicity]~\label{def: c-cyc} We say
  that a set $\xi \subseteq X\times Y$ is \linebreak[3]\textit{c-cyclic\-ally monotone} if for
  any $m\in\mathbb{N}^+$, $(x_i,y_i) \in \xi$, $1\leq i \leq m$,
  implies
  \[ \label{eq:cyclical} \sum_{i=1}^{m} c(x_i, y_i) \leq
    \sum_{i=1}^{m} c(x_i, y_{i-1}),\ (x_0,y_0) \equiv (x_m,y_m).
  \]
\end{definition}

We focus on the quadratic cost ($p=2$) and the quadratic Wasserstein
distance ($W_2$). We also assume that $\mu$ does not give mass to
small sets~\cite{brenier1991polar}, which is a necessary condition to
guarantee the existence and uniqueness of the optimal map under the
quadratic cost for Monge's problem. This requirement is natural for
seismic signals.  Brenier~\cite{brenier1991polar} also proved that the
optimal map coincides with the optimal plan in the sense that
$\pi =(\textnormal{Id} \times T)_\# \mu$ if Monge's problem has a solution. Thus,
one can extend the notion of cyclical monotonicity to optimal maps.

The $W_2$ distance is related to the Sobolev norm
$\dot H^{-1}$~\cite{peyre2018comparison} and has been proved to be
insensitive to mean-zero noise~\cite{engquist2016optimal,ERY2019}.  If
both $f=d\mu$ and $g=d\nu$ are bounded from below and above by
constants $c_1$ and $c_2$, the following nonasymptotic equivalence
holds:
\begin{equation}\label{eq:smoothing1}
  \frac{1}{c_2} \|f-g\|_{\dot H^{-1}} \leq W_2(\mu, \nu) \leq \frac{1}{c_1} \|f-g\|_{\dot H^{-1}},
\end{equation} 
where
$\|f\|_{\dot H^{-1}} = \big\| |\xi|^{-1} \hat{f}(\xi) \big\|_{L^2}$,
$\what f$ is the Fourier transform of $f$, and $\xi$ represents the
frequency.  If $d\zeta$ is an infinitesimal perturbation that has zero
total mass~\cite{Villani},
\begin{equation}~\label{eq:smoothing2} W_2(\mu,
  \mu+d\zeta)=\|d\zeta\|_{\dot H_{(d\mu)}^{-1}}+o(d\zeta),
\end{equation}
which shows the asymptotic connections. Here, $\dot H_{(d\mu)}^{-1}$
is the $\dot H^{-1}$ norm weighted by measure $\mu$.

The $1/|\xi|$ weighting suppresses higher frequencies, as seen from
the definition of the $\dot H^{-1}$ norm. It is also referred to as
the smoothing property of the negative Sobolev norms. The asymptotic
and nonasymptotic connections in~\eqref{eq:smoothing1}
and~\eqref{eq:smoothing2} partially explain the smoothing properties
of the $W_2$ metric, which applies to any data
dimension~\cite{ERY2019}. It is also a natural result of the optimal
transport problem formulation.  On the other hand, the $L^2$ norm is
known to be sensitive to noise~\cite{virieux2009overview}. The noise
insensitivity of the Wasserstein metric has been used in various
applications~\cite{lellmann2014imaging,Puthawala2018}.
\begin{theorem}[$W_2$ insensitivity to
  noise~\cite{engquist2016optimal}]\label{thm:noise}
  Consider probability density function $f = g + \delta$, where
  $\delta$ is mean-zero noise \tu(random variable with zero mean\tu) with
  variance $\eta$, piecewise constant on $N$ intervals (numerical
  discretization). Then
  \[ \|f-g\|_{2}^2 = \bO(\eta), \quad W_2^2(f,g) =
    \bO\left(\frac{\eta}{N}\right).
  \]
\end{theorem}
\begin{remark}
  Theorem~\ref{thm:noise} holds for general (signed) signals of zero
  mean if the data is normalized by the linear
  scaling~\eqref{eq:linear}.
\end{remark}

\section{Full-Waveform Inversion with the Wasserstein
  Metric} \label{sec:cvx} A significant source of difficulty in
seismic inversion is the high degree of nonlinearity and
nonconvexity. FWI is typically performed with the $L^2$ norm as the
objective function using local optimization algorithms in which the
subsurface model is described by using a large number of unknowns, and
the number of model parameters is determined a
priori~\cite{tarantola2005inverse}. It is relatively inexpensive to
update the model through local optimization algorithms, but the
convergence highly depends on the choice of a starting
model. Mathematically it is related to the highly nonconvex nature of
the PDE-constrained optimization problem and results in finding only
local minima.

The current challenges of FWI motivate us to modify the objective
function in the general framework in Figure~\ref{fig:FWI_flow} by
replacing the traditional $L^2$~norm with a new metric of better
convexity and stability for seismic inverse problems.  Engquist and
Froese~\cite{EFWass} first proposed to use the Wasserstein distance as
an alternative misfit function measuring the difference between
synthetic data $f$ and observed data $g$. This new objective function
has several properties advantageous for seismic inversion. In
particular, the convexity of the objective function with respect to
the data translation is a crucial property. Large-scale perturbations
of the velocity parameter mainly change the phases of the time-domain
signals~\cite{jannane1989wavelengths,virieux2009overview,engquist2016optimal}. The
convexity regarding the data shift is the key to avoid the so-called
cycle-skipping issues, which is one of the main challenges of
FWI. Results regarding the convexity are given in Theorem~\ref{thm:
  biconvex} below.

Seismic signals are in both the time and the spatial domain. One can
solve a 2D or 3D optimal transport problem to compute the Wasserstein
distance~\cite{W1_2D,W1_3D} or use the trace-by-trace
approach~\eqref{eqn:Wp1D}, which utilizes the explicit solution to the
1D optimal transport problem~\cite{Villani}. It is fast and accurate
to compute the Wasserstein distance between 1D signals, so the
trace-by-trace approach is cost effective for
implementation. Nevertheless, benefits have been observed regarding
the lateral coherency of the data by solving a 2D or 3D optimal
transport problem to compute the $W_2$
metric~\cite{poncet2018fwi,messud2019multidimensional}. Both
approaches have been appreciated by the
industry~\cite{wang2019adaptive,Ramos2018}. One can refer
to~\cite{yang2017application,W1_3D} for more discussions.

The translation and dilation in the wavefields are direct effects of
variations in the velocity $v$, as can be seen from D'Alembert's
formula that solves the 1D wave
equation~\cite{engquist2016optimal}. In particular, we will
reformulate the theorems in~\cite{engquist2016optimal} as a joint
convexity of $W_2$ with respect to both signal translation and
dilation and prove it in a more general setting. In practice, the
perturbation of model parameters will cause both signal translation
and dilation simultaneously, and the convexity with respect to both
changes is an ideal property for gradient-based optimization.

Since seismic signals are partial measurements of the boundary value
in~\eqref{eq:FWD}, they are compactly supported in $\mathbb{R}^d$ and
bounded from above. It is therefore natural to assume

\begin{equation} \label{eq:cvx_assum} \int_{\mathbb{R}^d}\!\!\int_{\mathbb{R}^d}
  |x-y|^2 f(x)g(y)dx \, dy < +\infty.  \end{equation}
Normalization is a very important
step in seismic inversion that turns oscillatory seismic signals into
probability measures. It is one prerequisite of optimal
transport~\cite{qiu2017full,metivier2019graph,Survey2,
  yang2017application}. In this section, we regard the normalized
synthetic data and observed data as probability densities $f = d\mu$
and $g =d\nu$ compactly supported on convex domains
$X, Y\subseteq \R^d$, respectively.

Next, we will improve our result in~\cite{engquist2016optimal} with a
stronger convexity proof in the following theorem, Theorem~\ref{thm: biconvex},
which states a joint convexity in multiple variables with respect to
both translation and dilation changes in the data.
Assume that $s_k\in \mathbb{R}$, $k = 1,\dots,d$, is a set of
translation parameters and $\{e_k\}_{k=1}^d$ is the standard basis of
the Euclidean space $\mathbb{R}^d$.
$A=\diag (1/\lambda_1,\dots,1/\lambda_d)$ is a dilation matrix where
$\lambda_k\in \mathbb{R}^+$, $k= 1,\dots,d$.  We define $f_{\Theta}$ as
jointly the translation and dilation transformation of function $g$
such that 
\begin{equation} \label{eq:f_theta} f_{\Theta}(x)
=\det(A)g\Biggl(A\Biggl(x-\sum_{k=1}^d s_k e_k\Biggr)\Biggr).  \end{equation}We will prove the convexity
in terms of the multivariable
$$\Theta = \{s_1, \dots, s_d, \lambda_1, \dots, \lambda_d\}\in \R^{2d}.$$  

\begin{theorem}[Convexity of $W_2$ in translation and dilation]
  \label{thm: biconvex}
  Let $g = d\nu$ be a probability density function with finite second
  moment and $f_{\Theta}$ be defined by~\eqref{eq:f_theta}. If, in
  addition, $g$ is compactly supported on convex domain
  $Y\subseteq \R^d$, the optimal transport map between $f_{\Theta}(x)$
  and $g(y)$ is $y = T_{\Theta}(x)$ where
  $\langle T_{\Theta}(x), e_k \rangle = \frac{1}{\lambda_k} (\langle x
  , e_k \rangle - s_k)$, $k= 1,\dots,d$. Moreover,
  $I(\Theta) = W_2^2(f_{\Theta}(x),g)$ is a strictly convex function
  of the multivariable $\Theta$.
\end{theorem}

\begin{proof}[Proof of Theorem~\ref{thm: biconvex}]
  First, we will justify that $y =T_{\Theta}(x)$ is a
  measure-preserving map according to
  Definition~\ref{def:mass_preserve}. It is sufficient to check that
  $T_{\Theta}$ satisfies Equation~\eqref{eq:mass_preserve2}:
  \[
    f_{\Theta}(x) =\det(A)g(T_{\Theta}(x)) =\det(\nabla T_{\Theta}(x))
    g(T_{\Theta}(x)).
  \]

  Since $f_\Theta$ and $g$ have finite second moment by
  assumption, \eqref{eq:cvx_assum} holds.  Next, we will show that the
  new joint measure
  $\pi_{\Theta}=(\tu{Id} \times T_{\Theta})\# \mu_{\Theta}$ is cyclically
  monotone. This is based on two lemmas from \cite[p.~80]{Villani} and
  the fundamental theorem of optimal transport
  in~\cite[p.~10]{ambrosio2013user} on the equivalence of optimality and
  cyclical monotonicity under the assumption of ~\eqref{eq:cvx_assum}.

  For $c(x,y) = |x-y|^2$, the cyclical monotonicity in
  Definition~\ref{def: c-cyc} is equivalent to
  \[
    \sum_{i=1}^{m}x_i \cdot (T(x_i)- T(x_{i-1})) \geq 0,
  \]
  for any given set of $\{x_i\}_{i=1}^m \subset X$. For
  $T_{\Theta}(x)$, we have
  \begin{align}
    \lteqn \sum_{i=1}^{m}x_i \cdot (T_{\Theta}(x_i)-T_{\Theta}(x_{i-1}))\notag\\ & =  
                                                                    \sum_{i=1}^{m}  \sum_{k=1}^d  \langle x_i,e_k \rangle \cdot  (\langle T_{\Theta}(x_{i}) ,e_k \rangle - \langle T_{\Theta}(x_{i-1}) ,e_k \rangle)   \nonumber \\ 
                                                                  &=   \sum_{i=1}^{m}  \sum_{k=1}^d   \frac{1}{\lambda_k} \langle x_i,e_k \rangle \cdot  (\langle x_i,e_k \rangle - \langle x_{i-1},e_k \rangle ) \nonumber \\
                                                                  &=  \frac{1}{2} \sum_{k=1}^d  \frac{1}{\lambda_k} \sum_{i=1}^{m}   |\langle x_i,e_k \rangle - \langle x_{i-1},e_k \rangle |^2 \geq 0,
  \end{align}
  which indicates that the support of the transport plan
  $\pi_{\Theta}=(\tu{Id} \times T_{\Theta})\# \mu_{\Theta}$ is cyclically
  monotone. By the uniqueness of monotone measure-preserving optimal
  maps between two distributions~\cite{mccann1995existence}, we assert
  that $T_{\Theta}(x)$ is the optimal map between $f_{\Theta}$ and
  $g$. The squared $W_2$ distance between $f_{\Theta}$ and $g$ is
  \begin{align}
    I(\Theta)   =  W_2^2(f_{\Theta},g) & = \int_{X} |x-T_{\Theta} (x)|^2 f_{\Theta}(x) dx \nonumber\\
                                       & = \int_{Y} \sum_{k=1}^d |(\lambda_k-1) \langle y,e_k\rangle + s_k |^2 d\nu \nonumber\\
                                       & =   \sum_{k=1}^d a_k (\lambda_k-1)^2    + 2 \sum_{k=1}^d b_k s_k (\lambda_k-1)  + \sum_{k=1}^d s_k^2,  
  \end{align}
  where $a_k =\int_{Y} |\langle y,e_k\rangle|^2\, d\nu $ and
  $b_k =\int_Y \langle y,e_k\rangle d\nu$.

  $I(\Theta)$ is a quadratic function whose Hessian matrix $H(\Theta)$
  is
  \[
    \resizebox{.97\textwidth}{!}{$\begin{pmatrix}
      I_{s_1s_1} &\dots & I_{s_1s_d} & I_{s_1\lambda_1}&\dots & I_{s_1\lambda_d}  \\
      \vdots   &\ddots & \vdots       &\vdots                    &\ddots  &  \vdots  \\
      I_{s_d s_1} &\dots & I_{s_ds_d} & I_{s_d\lambda_1}&\dots & I_{s_d\lambda_d}  \\
      I_{\lambda_1 s_1} &\dots & I_{\lambda_1 s_d} & I_{\lambda_1 \lambda_1}&\dots & I_{\lambda_1 \lambda_d}  \\
      \vdots   &\ddots & \vdots       &\vdots                    &\ddots  &  \vdots  \\
      I_{\lambda_d s_1} &\dots & I_{\lambda_d s_d} & I_{\lambda_d \lambda_1}&\dots & I_{\lambda_d \lambda_d}  \\
    \end{pmatrix} = 2
    \begin{pmatrix}
      1            &\dots    & 0              & b_1                       &\dots    & 0   \\
      \vdots   &\ddots & \vdots       &\vdots                    &\ddots  &  \vdots  \\
      0            &\dots   &  1                &0                          &\dots    & b_d \\
      b_1          &\dots  & 0             & a_1                         &\dots    & 0   \\
      \vdots   &\ddots & \vdots       &\vdots                    &\ddots  &  \vdots  \\
      0             &\dots  &b_d           & 0                           &\dots & a_d  \\
    \end{pmatrix}.$}
  \]
  $H(\Theta)$ is a symmetric matrix with eigenvalues
  \[
    a_k +1\pm \sqrt{a_k^2 - 2a_k + 4b_k^2 + 1},\q k=1,\dots, d.
  \]
  Since $a_k =\int_{Y} |\langle y,e_k\rangle|^2 \, d\nu \geq 0$ by
  definition, and
  \begin{align}
    \lteqn   \left(a_k +1\right)^2 - \Bigl(\sqrt{a_k^2 - 2a_k + 4b_k^2 + 1} \Bigr)^2 \nonumber\\ 
    &= 4\biggl( \int_{Y} |\langle y,e_k\rangle|^2 d\nu \int_{Y} 1^2 d\nu - \biggl(\int_Y \langle y,e_k\rangle d\nu \biggr)^2\biggr) \nonumber \\
    & \geq 0 \quad \text{by Cauchy-Schwarz inequality},~\label{eq:CSineq}
  \end{align}
  all the eigenvalues of $H(\Theta)$ are nonnegative.  Given any
  $k=1,\dots,d$, the equality in~\eqref{eq:CSineq} holds if and only
  if $l(y) = |\langle y,e_k\rangle|^2$ is a constant function
  $\forall y\in Y$, which contradicts the fact that $Y$ is a convex
  domain. Therefore, the Hessian matrix of $I(\Theta)$ is symmetric
  positive definite, which completes our proof that
  $W_2^2(f_{\Theta},g)$ is a strictly convex function with respect to
  $\Theta = \{s_1, \dots, s_d, \lambda_1, \dots, \lambda_d\}$, the
  combination of translation and dilation variables.
\end{proof}

\begin{figure}
  \centering
  \includegraphics[width=0.7\textwidth]{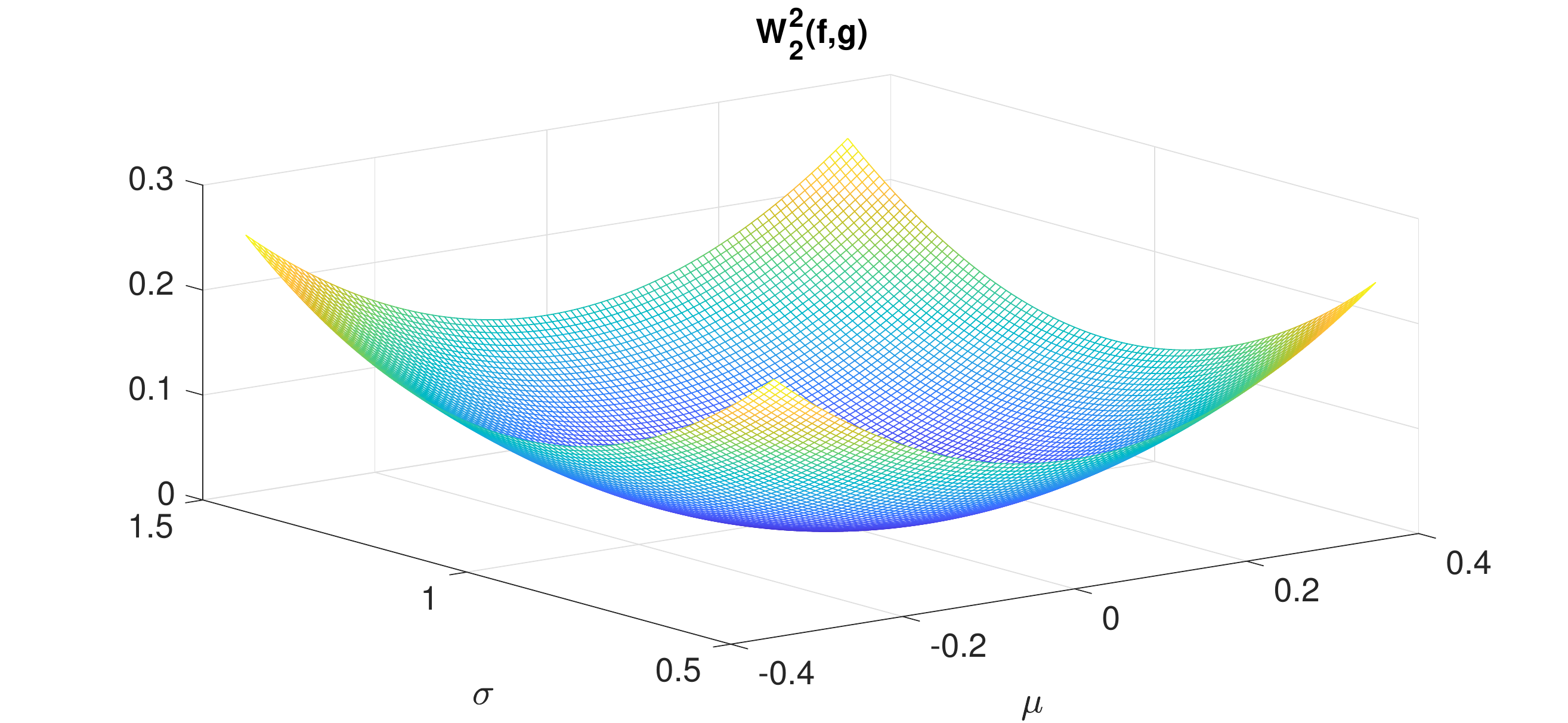}
  \caption{$W^2_2(f_{[\mu,\sigma]},g)$ where $f_{[\mu,\sigma]}$ and
    $g$ are probability density functions of normal distribution
    $\mathcal{N}(\mu,\sigma^2)$ and
    $\mathcal{N}(0,1)$.}~\label{fig:Gaussian_W2}
\vspace{-4ex}
\end{figure}

In Figure~\ref{fig:Gaussian_W2}, we illustrate the joint convexity of
the squared $W_2$ distance with respect to both translation and
dilation by comparing density functions of normal distributions. We
set $f_{[\mu,\sigma]}$ as the density function of the 1D normal
distribution $\mathcal{N}(\mu,\sigma^2)$. The reference function $g$
is the density of $\mathcal{N}(0,1)$. Figure~\ref{fig:Gaussian_W2} is
the optimization landscape of $W^2_2(f_{[\mu,\sigma]},g)$ as a
multivariable function. It is globally convex with respect to both
translation $\mu$ and dilation $\sigma$.

Note that Theorem~\ref{thm: biconvex} applies when $f$ and $g$ are
probability density functions, which is \textit{not} the case for
seismic data. We will discuss more on this topic in
Section~\ref{sec:Data_Normalization}. Let us consider that we can
transform the data into density functions and compute the $W_2$
distance. The convexity of the ``normalized'' $W_2$ with respect to
translation and dilation depends on the choice of normalization
function. We will prove that the softplus scaling~\eqref{eq:softplus}
still keeps the good convexity when the hyperparameter $b$
in~\eqref{eq:softplus} is large.

\section{Data Normalization}\label{sec:Data_Normalization}
We discuss the intrinsic convexity of the Wasserstein distance
in~Section~\ref{sec:cvx}, but the fact that functions should be
restricted to probability distributions is the primary constraint of
applying optimal transport to general signals, in particular,
oscillatory seismic waves.  So far, mathematical extensions of the
optimal transport theory to signed measures are quite
limited~\cite{mainini2012description,ambrosio2011gradient}. Different
strategies have been proposed in the literature to tackle this issue
numerically.

The dual formulation of the 1-Wasserstein distance ($W_1$) coincides
in expression with the so-called Kantorovich-Rubinstein (KR) norm
while the latter is well-defined for all functions. Using the KR norm
as an alternative to $W_1$ is a feasible approach in seismic
inversion~\cite{W1_2D,W1_3D}. Another interesting strategy is to map
the discrete signed signals into a graph space and then compare the
Wasserstein distance between the obtained point
clouds~\cite{scarinci2019robust,thorpe2017transportation,metivier2019graph,metivier2018optimal}. One
can also choose to penalize the matching filter between datasets to be
an identity based on the Wasserstein metric instead of working with
the data itself~\cite{sun2019application,sun2019stereo}. Here we
focus on a different approach: to transform the data into probability
density functions using a nonlinear scaling function before the
comparison.

Since properties of the Wasserstein distance are deeply rooted in the
theory of optimal transport, which studies probability measures, all
the strategies above have advantages as well as limitations. It is
efficient to compute the KR norm for 2D and 3D data, but the norm does
not preserve the convexity regarding data shifts. The graph-based idea
may preserve the convexity if the displacements along with the
amplitude and the phase are well-balanced, but inevitably increases
the dimensionality of the optimal transport problem. Compared to these
strategies, the approach we present here is remarkably efficient as it
does not increase the dimensionality of the data such that it uses the
closed-form solution to the 1D optimal transport problem. The benefits
in terms of computational costs are significant for practical
large-scale applications.

We have discussed normalization before in~\cite{qiu2017full,Survey2},
but not until here are any rigorous results given. In~\cite{EFWass},
the signals were separated into positive and negative parts
$f^+ = \max\{f,0\}$ and $f^- = \max\{-f,0\}$, and scaled by the total
sum. However, the approach cannot be combined with the adjoint-state
method~\cite{Plessix}, while the latter is essential to solve
large-scale problems. This separation scaling introduces
discontinuities in derivatives from $f^+$ or $f^-$, and the
discontinuous Fr\'{e}chet derivative of the objective function with
respect to $f$ cannot be obtained. The squaring $f^2$ or the
absolute-value scaling $\abs{f}$ are not ideal either since they are
not one-to-one maps, and consequently lead to nonuniqueness and
potentially more local minima for the optimization
problem~\eqref{eq:fwi}. One can refer to \cite{Survey2} for more
discussions.

Later, the linear scaling~\eqref{eq:linear}~\cite{yang2017application}
and the exponential-based methods~\cite{qiu2017full}, \eqref{eq:exp},
and~\eqref{eq:softplus} are observed to be effective in practice. The
main issue of data normalization is how to properly transform the
data, as one can see from the literature or practice that some scaling
methods seem to work better than others. This section focuses on
presenting several useful scaling methods and explaining their
corresponding impacts on the $W_2$ misfit function. We aim to offer
better understandings of the role that the normalization function
plays in inversion.  First, we generalize the class of effective
normalization functions that satisfy Assumption~\ref{ass:DN}.
\begin{assumption}
  \label{ass:DN}
  Given a scaling function $\sigma: \R \rightarrow \R^{+}$, we define
  the normalization operator $P_{\sigma}$ on function
  $f: \Pi \rightarrow\R$ as follows: for $ x \in \Pi_1,\ t \in \Pi_2$
  where $\Pi = \Pi_1 \times \Pi_2$,
\begin{equation} \label{eq:DN} \begin{aligned}(P_{\sigma}f)
  (x,t) &= \dfrac{\sigma(y(x,t))+c}{S_\sigma(x)},\\ S_\sigma(x) &=
  \int_{\Pi_2} \left(\sigma(f(x,\tau)) + c \right)d\tau, 
\end{aligned}
\quad c\geq
  0.  \end{equation}The scaling function $\sigma$ satisfies the following
  assumption
  \begin{enumerate}[\quad (i)]
  \item $\sigma$ is one-to-one;
  \item $\sigma: \R \rightarrow \R^{+}$ is a $C^{\infty}$ function.
  \end{enumerate}
\end{assumption}
Functions that satisfy Assumption~\ref{ass:DN} include the linear
scaling $\sigma_l(x;b)$, the exponential scaling $\sigma_e(x;b)$, and
the softplus function $\sigma_s(x;b)$, where $b$ is a
hyperparameter. We use the definition ``hyperparameter'' to
distinguish it from the velocity parameter, which is 
determined in the inversion process:
\[
  \sigma_l(x;b) = x+b, \quad \sigma_e(x;b) = \exp(bx),\quad
  \sigma_s(x;b) = \log\left(\exp(bx)+1\right).
\]
Equivalent definitions are in
equations~\eqref{eq:linear},~\eqref{eq:exp}, and~\eqref{eq:softplus}.

In the rest of~Section~\ref{sec:Data_Normalization} we prove several
properties for the class of normalization operators that satisfy
Assumption~\ref{ass:DN} and discuss their roles in improving
optimal-transport-based FWI. Since the normalization~\eqref{eq:DN} is
performed on a domain $\Pi_2$, and the properties apply for any
$x\in \Pi_1$, we will assume $f$ and $g$ are functions that are
defined on the domain $\Pi_2$ (instead of $\Pi$) for the rest of the
section. As mentioned in Section~\ref{sec:FWI}, we consider the
trace-by-trace approach (1D optimal transport) for seismic
inversion. Hence, $\Pi_2 \subseteq \R$, but all the properties in this
section hold for $\Pi_2\subseteq \R^d$, $d\geq 2$, as well.

\subsection{A Metric for Signed Measures}
Let $\mathcal{P}_s({\Pi_2})$ be the set of finite signed measures that
are compactly supported on domain ${\Pi_2} \subseteq
\mathbb{R}^d$. Consider $f = d\mu$ and $g =d\nu$ where
$\mu, \nu \in \mathcal{P}_s({\Pi_2})$. We denote
$\tilde{f} = P_\sigma(f)$ and $\tilde{g}=P_\sigma(g)$ as normalized
probability densities, where $P_\sigma$ is any scaling operator that
satisfies Assumption~\ref{ass:DN}.  We shall use
$W_2(\tilde{f},\tilde{g})$ as the objective function measuring the
misfit between original seismic signals $f$ and $g$. It can also be
viewed as a new loss function
$W_\sigma(f,g) = W_2(\tilde{f},\tilde{g})$ that defines a metric
between $f$ and $g$.
\begin{theorem}[Metric for signed measures]\label{thm:metric}
  Given $P_\sigma$ that satisfies Assumption~\ref{ass:DN}, $W_\sigma$
  defines a metric on $\mathcal{P}_s({\Pi_2})$.
\end{theorem}

\begin{proof}
  Since $W_2$ is a metric on probability measures with finite second
  moment, $W_\sigma$ is symmetric, nonnegative, and finite on
  $\mathcal{P}_s({\Pi_2})$. Also, we have that
  $$W_\sigma(f,f) = W_2(P_\sigma(f),P_\sigma(f)) = 0.$$

  If $W_\sigma(f,g) = 0$, then the following holds:
  \begin{equation*} \label{eq:metric_1} P_\sigma(f) =
    \frac{\sigma(f)+c}{ \int_{\Pi_2} \sigma(f(\tau)) d\tau +
      c|{\Pi_2}|} = \frac{\sigma(g)+c}{ \int_{\Pi_2}
      \sigma(g(\tau))d\tau + c|{\Pi_2}|} = P_\sigma(g).
  \end{equation*}
  Since $f$ and $g$ are both compactly supported on ${\Pi_2}$,
  $\exists x^*\in {\Pi_2}$ such that $f(x^*) = g(x^*) = 0$. Together
  with $\tilde{f}(x^*) = \tilde{g}(x^*)$, we have
  \begin{equation*} \label{eq:metric_2} 
\begin{split}  \int_{\Pi_2}  \sigma(f(\tau))d\tau= \frac{\sigma(f(x^*))+c}{(P_\sigma f)(x^*) }
    - c|{\Pi_2}|&= \frac{\sigma(g(x^*))+c}{(P_\sigma g)(x^*) }
    -c|{\Pi_2}|\\ &= \int_{\Pi_2} \sigma(g(\tau))d\tau.
  \end{split}
\end{equation*}
  Together with~\eqref{eq:metric_1} and the fact that $\sigma$ is
  one-to-one, we have $f=g$ on ${\Pi_2}$.

  All that remains to check is the triangle inequality. Consider
  $h = d\rho$ where $\rho \in \mathcal{P}_s({\Pi_2})$.
  \begin{align*} \label{eq:metric_tri}
    W_\sigma(f,g) + W_\sigma(g,h)  &= W_2(\tilde{f},\tilde{g}) + W_2(\tilde{g},\tilde{h}) 
                                   \leq W_2(\tilde{f},\tilde{h})
                                     =W_\sigma(f,h).
 \qh \end{align*}
\end{proof}

\begin{remark}[Variance and invariance under mass subtraction]
  One can extend the optimal mass transportation problem between
  nonnegative measures whose mass is not normalized to
  unity~\cite{Villani}. Unlike the 1-Wasserstein distance ($W_1$),
  which corresponds to the case of $p=1$ in~\eqref{eq:static},
  $W_2(f,g)$ is \textit{not} invariant under mass subtraction. This
  property can easily be extended to a set of functions
  $\{\bar h\in L^2({\Pi_2}): f+\bar h \geq 0, g+\bar h\geq 0\}$ since
  generally
  \[
    W_2(f+ \bar h,g+\bar h)\neq W_2(f,g),
  \]
  while the $L^2$ norm and the $W_1$ distance remain unchanged:
  \[
    \|(f+\bar h) - (g+\bar h)\|_{L^2} = \|f-g\|_{L^2},\quad 
      W_1(f+\bar h,g+\bar h) =W_1(f,g).
  \]
  $W_2$ has the unique feature of \textit{variance} under mass
  subtraction/addition. Later, we will see that this feature gives us
  the Huber-type property (Theorem~\ref{thm:Huber}) and regularization
  effects (Theorem~\ref{thm:map_smooth}) by adding a positive constant
  $c$ to the signals.
\end{remark}

\subsection{Hyperparameter $b$: Effects on Convexity}\label{sec:b}
We list three specific scaling functions that satisfy
Assumption~\ref{ass:DN}, i.e., \eqref{eq:linear}, \eqref{eq:exp},
and~\eqref{eq:softplus}. In particular, the exponential scaling and
the softplus function are defined by the hyperparameter $b$, which
controls the convexity of the normalized Wasserstein distance. Without
loss of generality, we set $c=0$ in~\eqref{eq:DN}.

\begin{figure}
  \centering \subfloat[Linear
  scaling]{\includegraphics[height=0.23\textwidth]{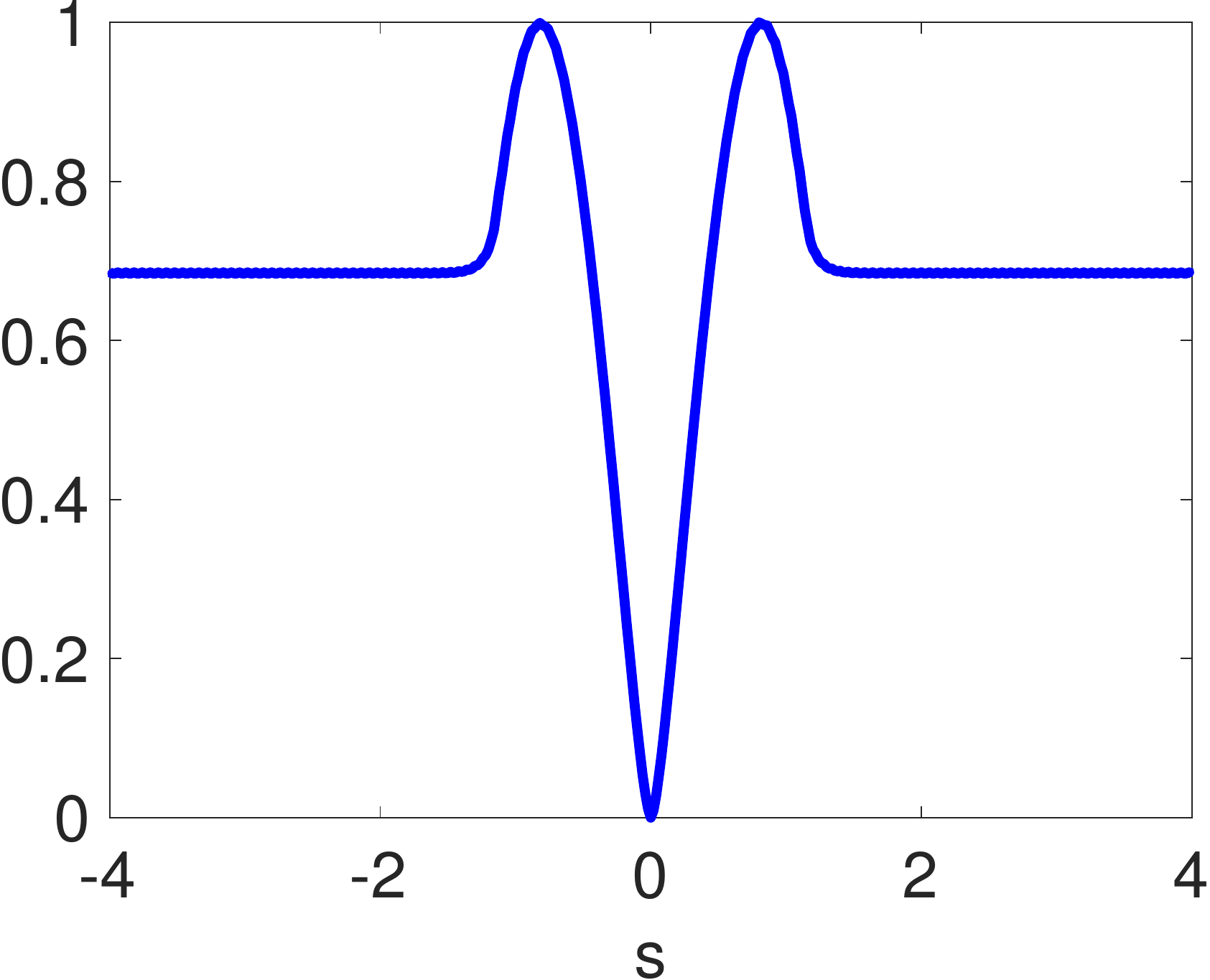}\label{fig:W2-linear}}\quad 
  \subfloat[Exponential (small
  $b$)]{\includegraphics[height=0.23\textwidth]{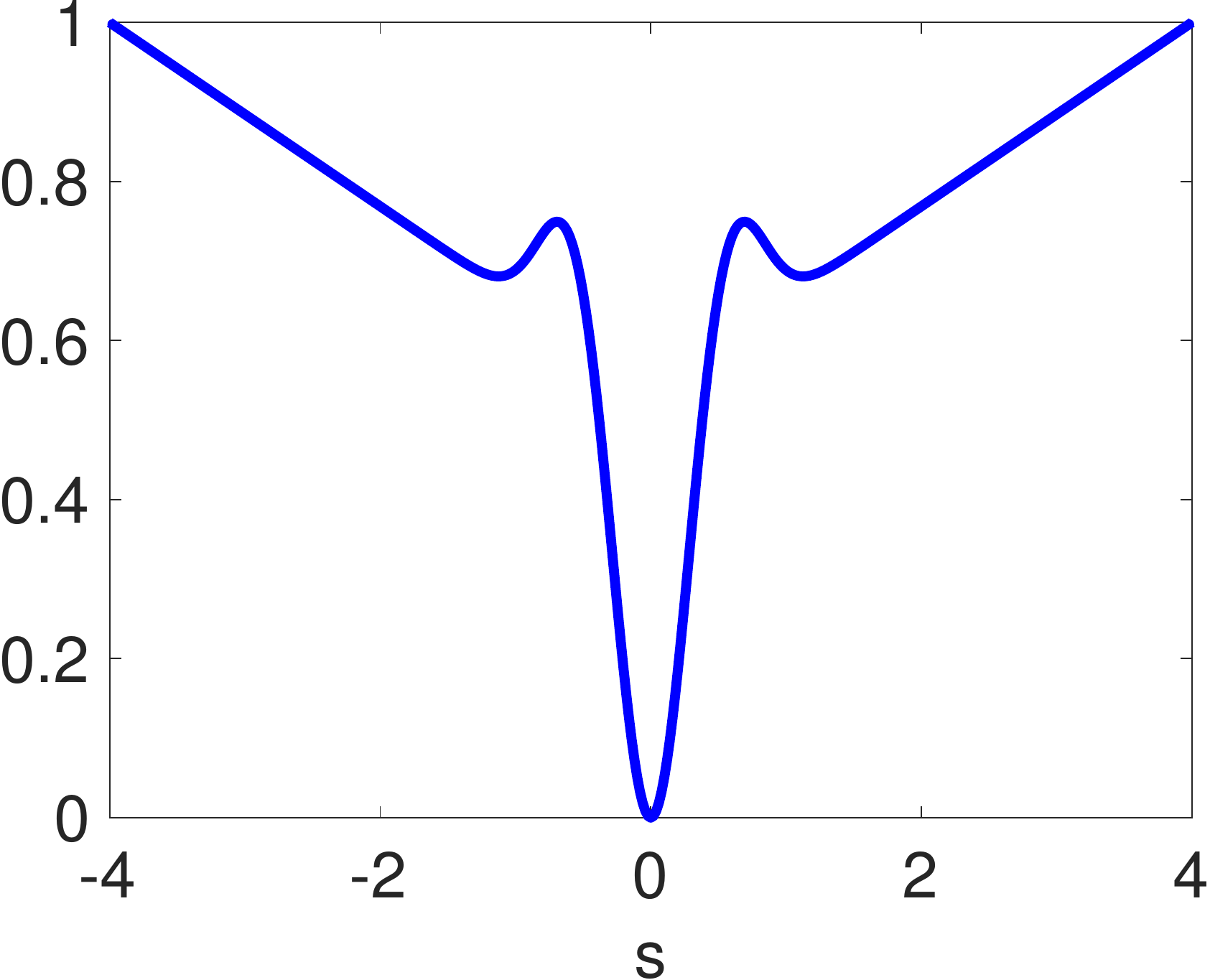}\label{fig:W2-bad-exp}}\quad 
  \subfloat[Exponential (large $b$)]{\includegraphics[height=0.23\textwidth]{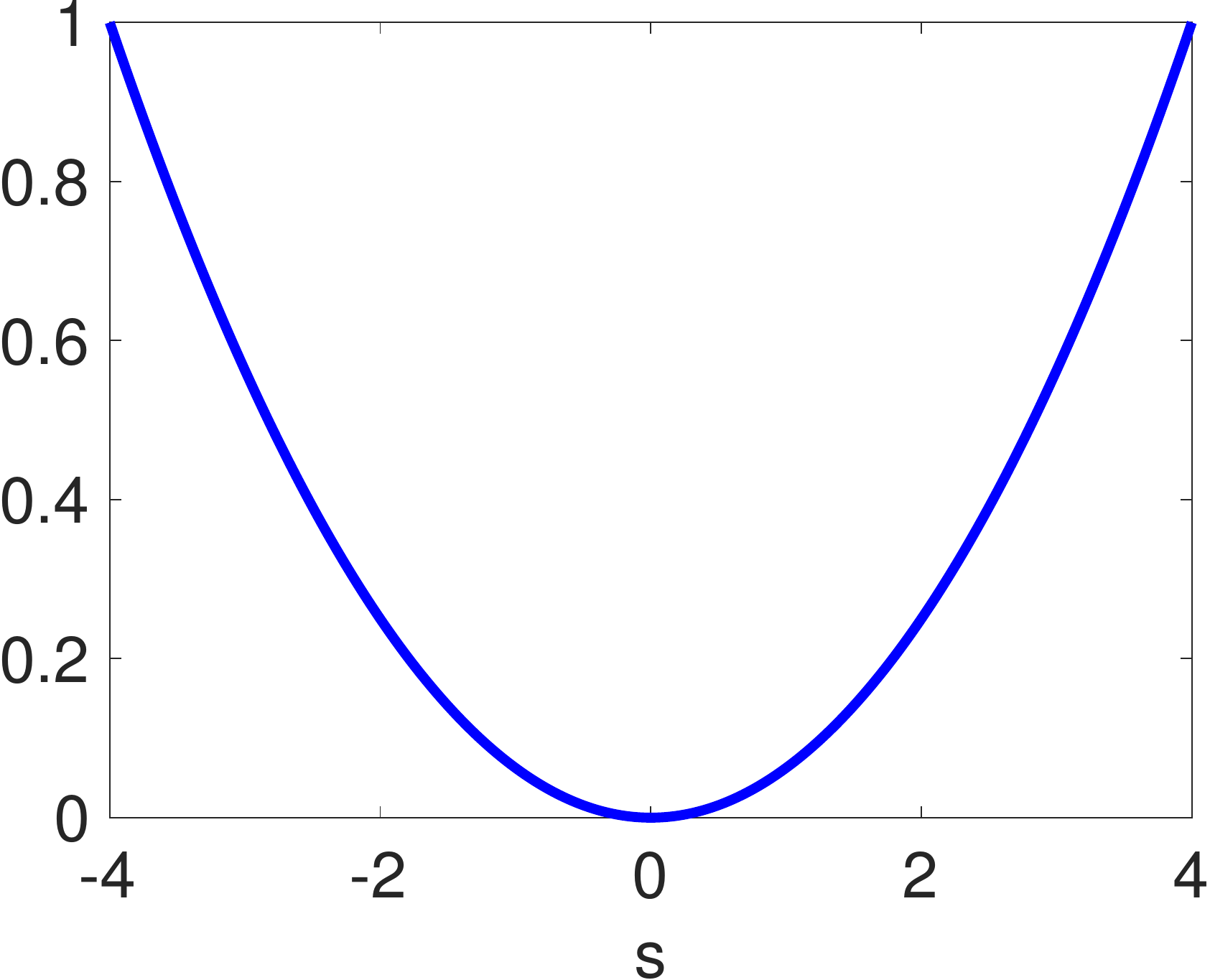}\label{fig:W2-good-exp}}\\
  \subfloat[Softplus (small
  $b$)]{\includegraphics[height=0.23\textwidth]{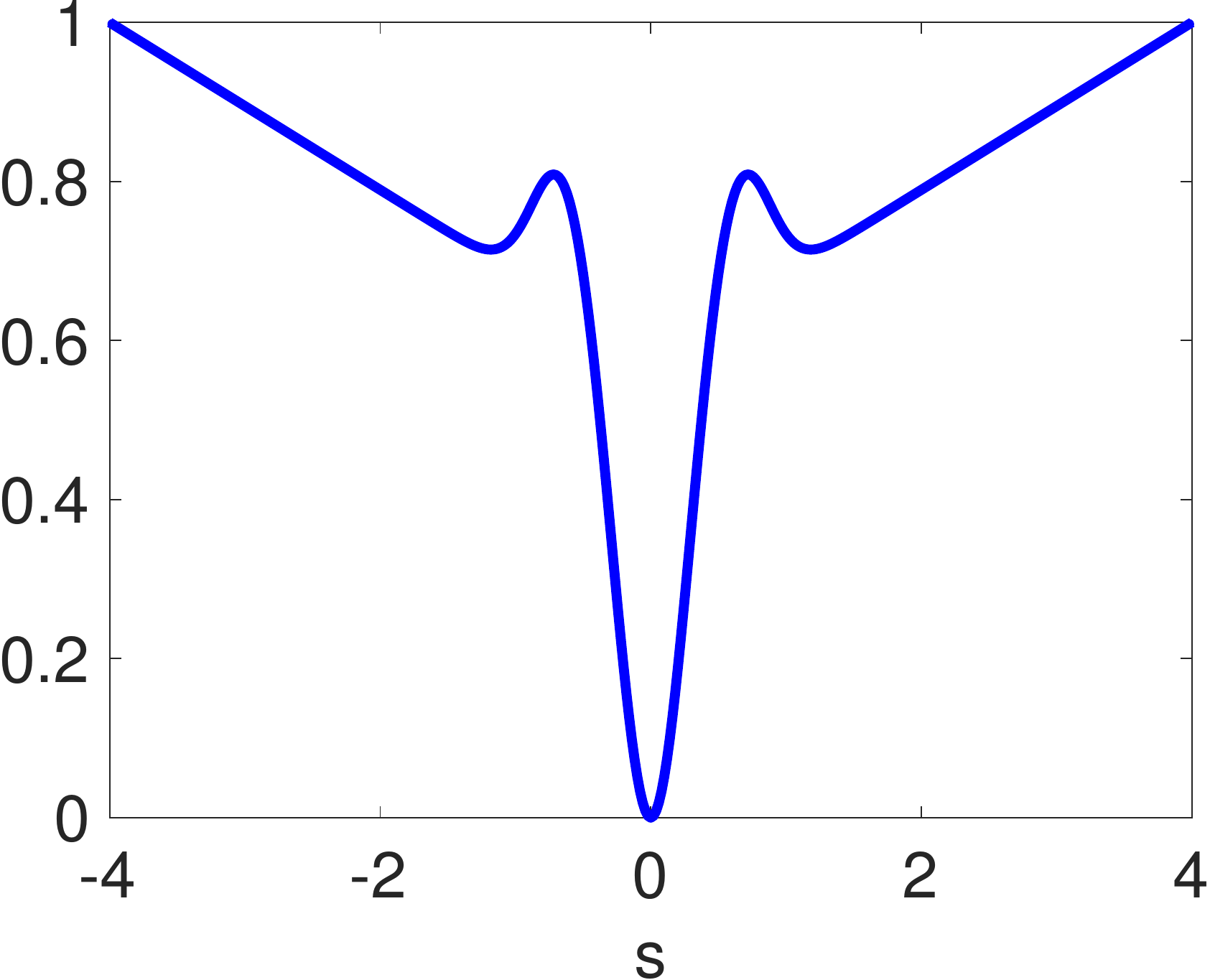}\label{fig:W2-bad-logi}}\quad 
  \subfloat[Softplus (large
  $b$)]{\includegraphics[height=0.23\textwidth]{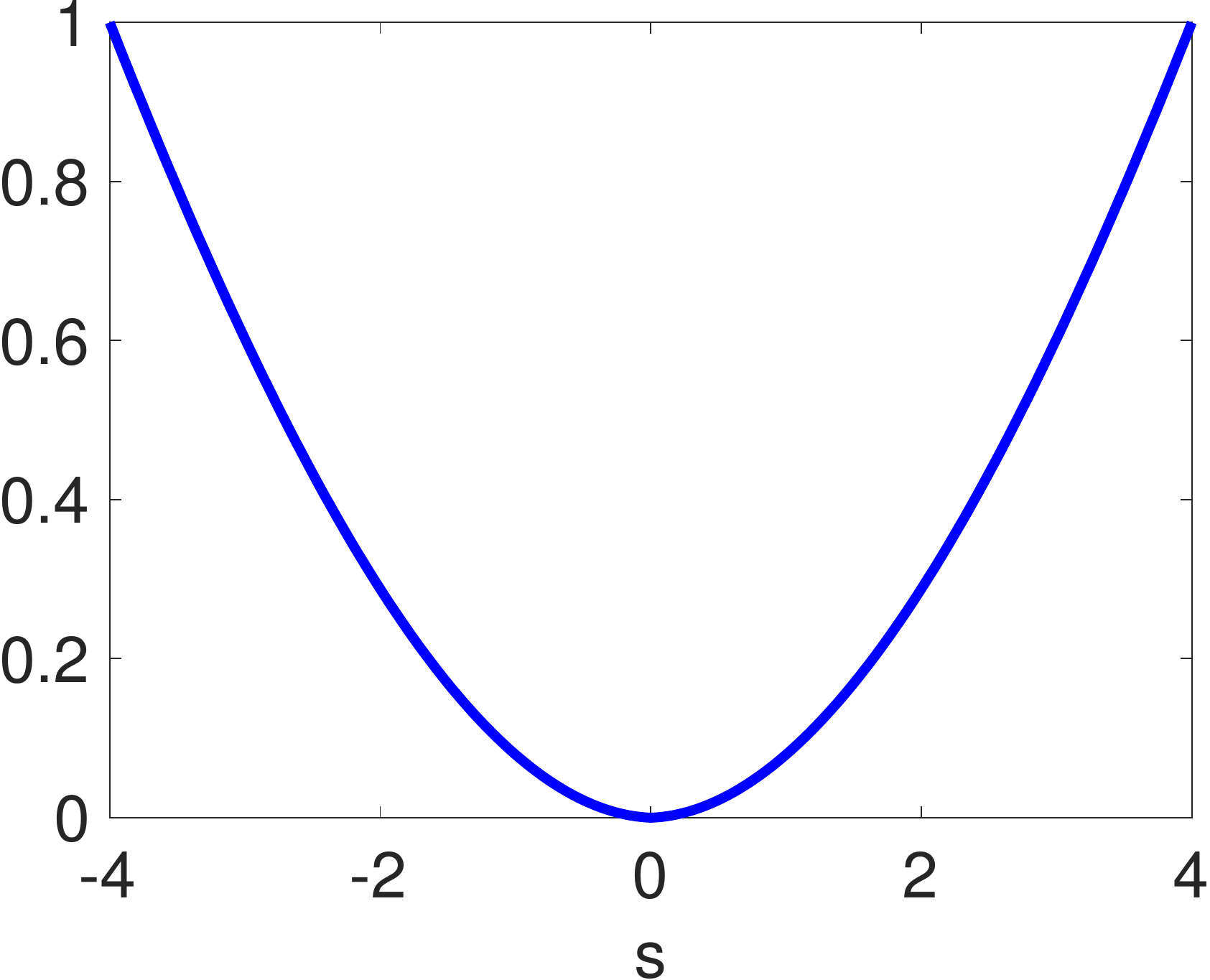}\label{fig:W2-good-logi}}
  \caption{Optimization landscape of the $W_2$ metric between Ricker
    wavelets $f(x)$ and $f(x-s)$, generated by (a) the linear, (b) and (c)
    the exponential with small and large $b$, and (d) and (e) the softplus
    scaling with small and large $b$.}\label{fig:W2-normalization}
\end{figure}

We compare the optimization landscape of the normalized $W_2$ metric
between function $f(x)$ and its shift $f(x-s)$.  The linear scaling
affects the convexity of~$W_2$ with respect to
translation~\cite{Survey2}; see
Figure~\ref{fig:W2-linear}. Figures~\ref{fig:W2-bad-exp}
and~\ref{fig:W2-good-exp} are obtained by the exponential scaling. We
choose scalar $b$ such that $\|bf\|_{\ell_\infty} \approx 0.5$ in
Figure~\ref{fig:W2-bad-exp} and $\|bf\|_{\ell_\infty} \approx 4$ in
Figure~\ref{fig:W2-good-exp}. The objective function plot in
Figure~\ref{fig:W2-good-exp} is more quadratic than the one in
Figure~\ref{fig:W2-bad-exp}, while the latter is nonconvex. The larger
the $b$, the more suppressed the negative counterparts of the
waveform. Similar patterns are also observed in
Figure~\ref{fig:W2-bad-logi} and~\ref{fig:W2-good-logi} for the
softplus scaling function. Nevertheless, one should note that an
extremely large $b$ is not preferred for the exponential scaling due
to the risk of noise amplification and potentially machine
overflow. Empirically, $b$ should be chosen properly in the range
$0.2 \leq \|bf\|_{\ell_\infty} \leq 6$. In this sense, the softplus
scaling $\sigma_s$ is more stable than the exponential scaling.

We present an extended result based on~Theorem~\ref{thm:
  biconvex}. Here, we consider \textit{signed} functions $g$ and
$f_\Theta$ defined by~\eqref{eq:f_theta}, compactly supported on
$\Pi_{\mathscr{K}}, \forall \Theta \in \mathscr{K}$. Here,
$\mathscr{K}$ is a compact subset of $\R^{2d}$ that represents the
set of transformation parameters. Under these assumptions, the
convexity of the $W_2$ metric is preserved when comparing signed
functions.

\begin{corollary}[Convexity of $W_2$ with the softplus scaling]
  \label{thm: softplus}
  Let $\widetilde{f_{\Theta,b}}$ and $\widetilde{g_b}$ be normalized
  functions of $f_{\Theta}$ and $g$ based on the softplus
  scaling~\eqref{eq:softplus} with hyperparameter $b$. Then, there is
  $b^*\in \R^+$ such that
  $I(\Theta,b) = W^2_2(\widetilde{f_{\Theta,b}},\widetilde{g_b})$ is
  strictly convex with respect to $\Theta$ if $b > b^*$.
\end{corollary}
\begin{proof}
  One key observation is that $I(\Theta,b)$ is a smooth function of
  multivariable $\Theta$ and the scalar variable $b$.  As
  $b\rightarrow +\infty$,
  \[
    \lim \limits_{b\rightarrow +\infty}\widetilde{ f_{\Theta,b}} =
    \frac{f_\Theta^+}{\int_{\Pi_{\mathscr{K}}} f_\Theta^+ } \coloneqq
    \widetilde{f_\Theta^+}, \quad \lim \limits_{b\rightarrow
      +\infty}\widetilde{ g_{b} }= \frac{g^+}{\int_{\Pi_{\mathscr{K}}}
      g^+} \coloneqq \widetilde{g^+}.
  \]
  Since $\widetilde{f_\Theta^+}$ and $\widetilde{g^+}$ are nonnegative
  functions with equal total sum,
  \[
\lim \limits_{b\rightarrow +\infty} I(\Theta, b) =
W_2^2(\widetilde{f_\Theta^+} ,\widetilde{g^+})
\]
is strictly convex
  in $\Theta$ by Theorem~\ref{thm: biconvex}. The Hessian of
  $W_2^2(\widetilde{f_\Theta^+} ,\widetilde{g^+})$ in $\Theta$,
  $H^+(\Theta)$, is symmetric positive definite for all
  $\Theta \in \mathscr{K}$. If we denote the Hessian of $I(\Theta,b)$
  with respect to $\Theta$ as $H(\Theta,b)$, which is also a
  matrix-valued smooth function in $\Theta$ and $b$, then
  $\lim_{b\rightarrow +\infty} H(\Theta, b) =
  H_\Theta^+$. Therefore, there is a $b^*$ such that when $b>b^*$,
  $H(\Theta, b)$ is symmetric positive definite for all
  $\Theta \in \mathscr{K}$, which leads to the conclusion of the
  corollary.
\end{proof}

\subsection{Hyperparameter $c$: Huber-Type Property}
While the choice of $b$ is essential for preserving the ideal
convexity of the $W_2$ metric with respect to shifts
(Theorem~\ref{thm: biconvex}), the other hyperparameter $c\geq 0$
in~\eqref{eq:DN} regularizes the quadratic Wasserstein metric as a
``Huber-type'' norm, which can be generalized to the entire class of
normalization functions that satisfies Assumption~\ref{ass:DN}. In
statistics, the Huber norm~\cite{huber1973robust} is a loss function
used in robust regression that is less sensitive to outliers in data
than the squared error loss. For a vector $\mathbold{\eta}$, the Huber
norm of $s\mathbold{\eta}$, $s\geq 0$, is $\bO(s^2)$ for small $s$ and
$\bO(s)$ once $s$ is larger than a threshold. For optimal
transport-based FWI, the Huber property is good for not
overemphasizing the mass transport between seismic events that are far
apart and physically unrelated as $s^2 \gg s$ for large $s$. The big-O
notation is defined as follows.
\begin{definition}[Big-O notation]
  Let $f$ and $g$ be real-valued functions with domain $\R$. We say
  $f(x) = \bO(g(x))$ if there are positive constants $M$ and $k$ such
  that $|f(x)| \leq M|g(x)|$ for all $x \geq k$. The values of $M$ and
  $k$ must be fixed for the function $f$ and must not depend on $x$.
\end{definition}

Assuming ${\Pi_2}\subseteq \R$, we will next show that the positive
constant $c$ in the data normalization operator $P_\sigma$, defined
in~\eqref{eq:DN}, turns the $W_2$ metric into a ``Huber-type''
norm. The threshold for the transition between $\bO(s^2)$ and $\bO(s)$
depends on the constant $c$ and the support ${\Pi_2}$. The constant
$c$ is added \textit{after} signals become nonnegative, and the choice
of $\sigma$ in Assumption~\ref{ass:DN} is independent of the
Huber-type property. Without loss of generality, we state the theorem
in the context of probability densities to avoid unrelated discussions
on making data nonnegative.

\begin{theorem}[Huber-type property for 1D signal]\label{thm:Huber}
  Let $f$ and $g$ be probability density functions compactly supported
  on ${\Pi_2} \subseteq \mathbb{R}$ and $g(x)= f(x-s)$ on
  ${\Pi^s_2} = \{x\in \mathbb{R}: x-s\in {\Pi_2}\}$ and zero
  otherwise. Consider $\tilde{f}$ and $\tilde{g}$ as new density
  functions defined by linear normalization~\eqref{eq:linear} for a
  given $c>0$\tu; then
  \begin{equation*}
    W^2_2(\tilde{f},\tilde{g})=\begin{cases}
      \bO(|s|^2) & \text{if $|s| \leq \frac{1}{c}+|\!\supp (f)|$},\\
      \bO(|s|)      & \text{otherwise}.
    \end{cases}
  \end{equation*}
\end{theorem}
\begin{proof}Without loss of generality, we assume $f$ is compactly
  supported on an interval $[a_1,a_2]\subseteq {\Pi_2}$ and $s\geq 0$.
  Note that $\tilde{f}$ and $\tilde{g}$ are no longer compactly
  supported on the domain ${\Pi_2}$. In one dimension, one can solve the
  optimal transportation problem explicitly in terms of the cumulative
  distribution functions
  \[ F(x) = \int_0^x f(t) dt, \quad G(x) = \int_0^x g(t) dt. \] It
  is well-known~\cite[theorem~2.18]{Villani} that the optimal
  transportation cost is
\begin{equation}\label{eq:cost1D} W_2^2(f,g) =
  \int_0^1|F^{-1}(t)-G^{-1}(t)|^2\,dt.  \end{equation}If additionally the target
  density~$g$ is positive, the optimal map from $f$ to $g$ becomes
  \begin{equation}\label{eq:map1D} T(x) = G^{-1}(F(x)).  \end{equation}
  Based on the 1D
  explicit formula~\eqref{eq:cost1D},
  \begin{multline*}
       W^2_2(\tilde{f},\tilde{g}) =  \int_0^1 |\tilde{F}^{-1}(y) -  \tilde{G}^{-1}(y)|^2 dy =\\
      \begin{cases}
        2\int_{y_1}^{y_3} |{F}^{-1}(y) - \frac{y}{c_1}|^2\,dy+ \int_{y_3}^{y_2} |{F}^{-1}(y) - {F}^{-1}(y-c_1s)+s|^2\,dy, \\\quad |s| \leq \frac{1}{c}+|a_2-a_1|,\\[2\jot]
        2\int_{y_1}^{y_2}
        |\tilde{F}^{-1}(y)-\frac{y}{c_1}|^2\,dy+\int_{y_2}^{y_3}
        \frac{1}{c^2}\,dy \quad \text{otherwise}.
      \end{cases}
    \end{multline*}
  Here $F, G, \tilde{F}, \tilde{G}$ are cumulative distribution
  functions of $f, g, \tilde{f}, \tilde{g}$, respectively. $|{\Pi_2}|$
  denotes the Lebesgue measure of the bounded domain ${\Pi_2}$,
  \[
  c_1 = \frac{c}{1+c|{\Pi_2}|},\quad  y_1 = c_1a_1,\quad 
  y_2 = c_1a_2+\frac{1}{1+c|{\Pi_2}|},\q  y_3 = c_1a_1+c_1s.
\]
Since
  $y_1$ and $y_2$ are independent of $s$, it is not hard to show that
  $W^2_2(\tilde{f},\tilde{g})$ is linear in $s$ if
  $s> \frac{1}{c}+|a_2-a_1|$ while
  $W^2_2(\tilde{f},\tilde{g}) = \bO(s^2)$ if
  $0\leq s \leq \frac{1}{c}+|a_2-a_1|$.
\end{proof}

\begin{remark}
  For higher dimensions ${\Pi_2}\subseteq\R^d$, choosing the map $T$
  along the shift direction gives
  \smash{$W_2^2(\tilde{f},\tilde{g}) \leq \bO(s)$} for large $s$. The optimal
  map can reduce \smash{$W_2^2(\tilde{f},\tilde{g})$} to $\bO(\log(s))$ if
  $d=2$, which grows more slowly as $s$ increases. We focus on $d=1$
  and do not elaborate the details for higher dimensions here since
  the trace-by-trace approach~\eqref{eqn:Wp1D} is mainly used in this
  paper and also in practice.
\end{remark}

Based on the subadditivity of the $W_2$ distance under rescaled
convolution~\cite{Villani},
\[
  W_2^2(\tilde{f}, \tilde{g} ) =W_2^2\biggl( \frac{{f}+c }{1+c|{\Pi_2}|},
   \frac{{g}+c }{1+c|{\Pi_2}|}\biggr) \leq \frac{W_2^2({f}, {g})
   }{(1+c|{\Pi_2}|)^2} \leq W_2^2({f}, {g}).
 \]
 The inequality shows that adding a constant $c$ to the data decreases
 the loss computed by the original objective function and explains the
 ``Huber-type'' property.

 \begin{figure}
   \centering
  \subfloat[][]{\includegraphics[height=0.18\textwidth]{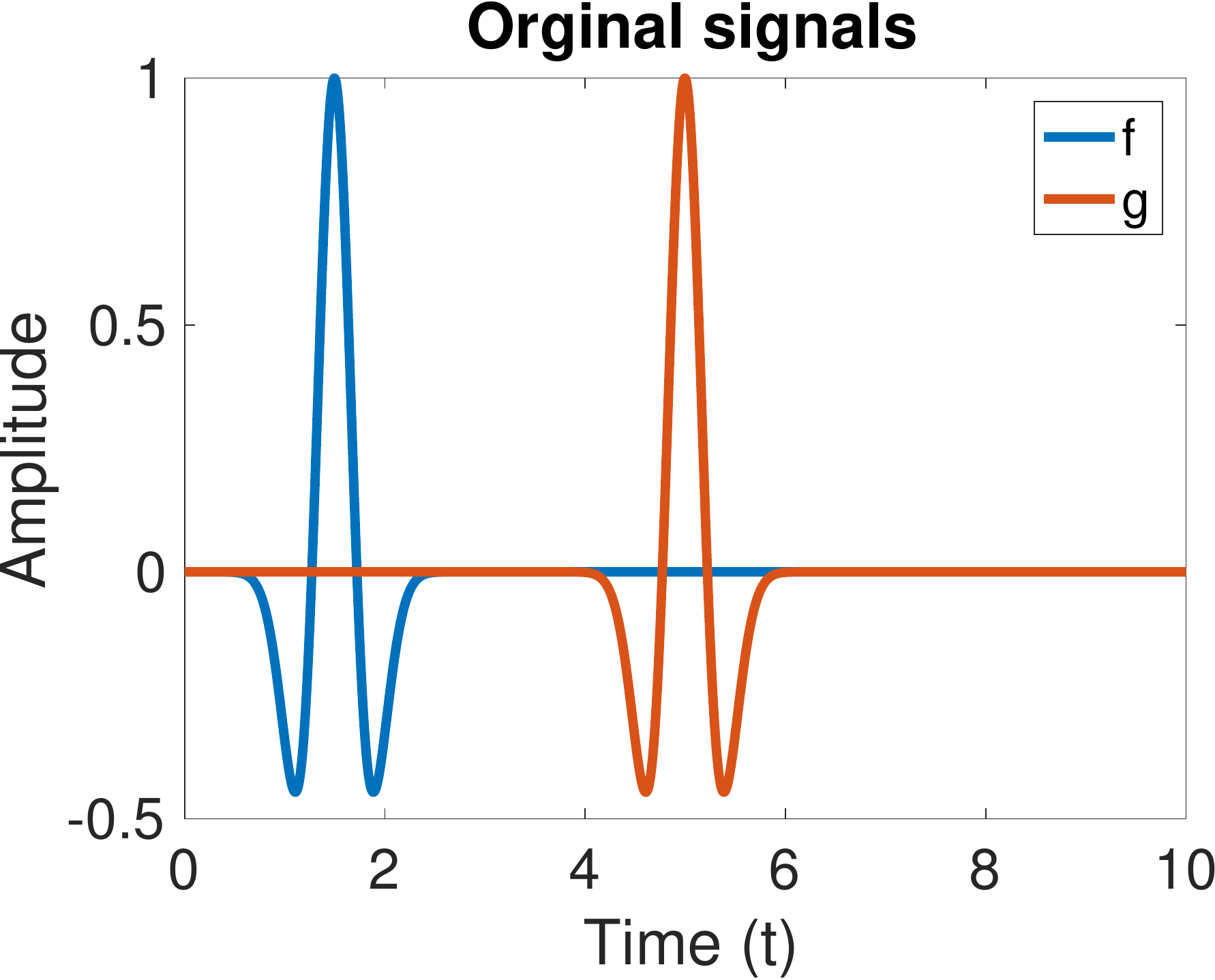}\label{fig:Huber_data0}}\hfill 
  \subfloat[][]{\includegraphics[height=0.18\textwidth]{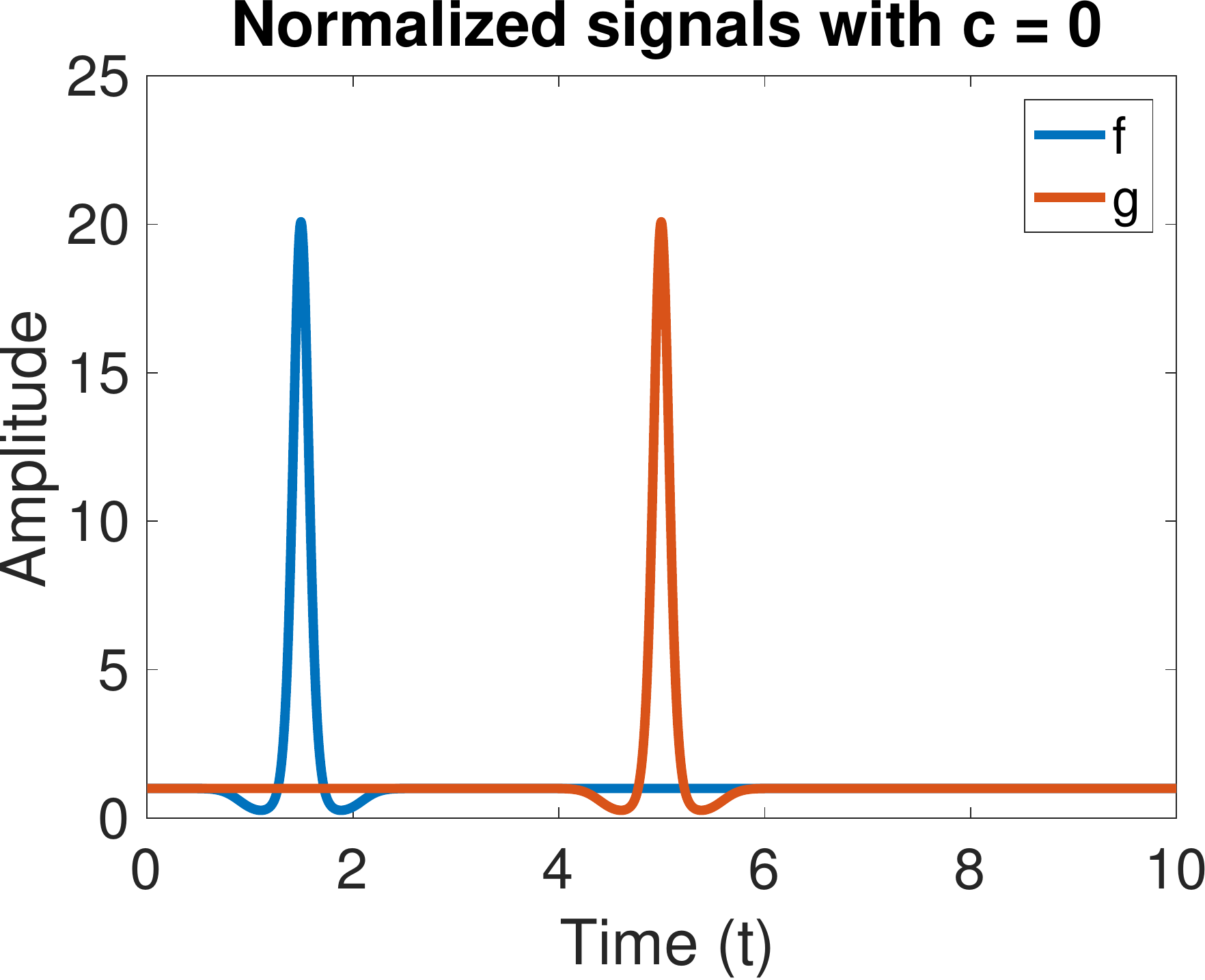}\label{fig:Huber_C0}}\hfill 
  \subfloat[][]{\includegraphics[height=0.18\textwidth]{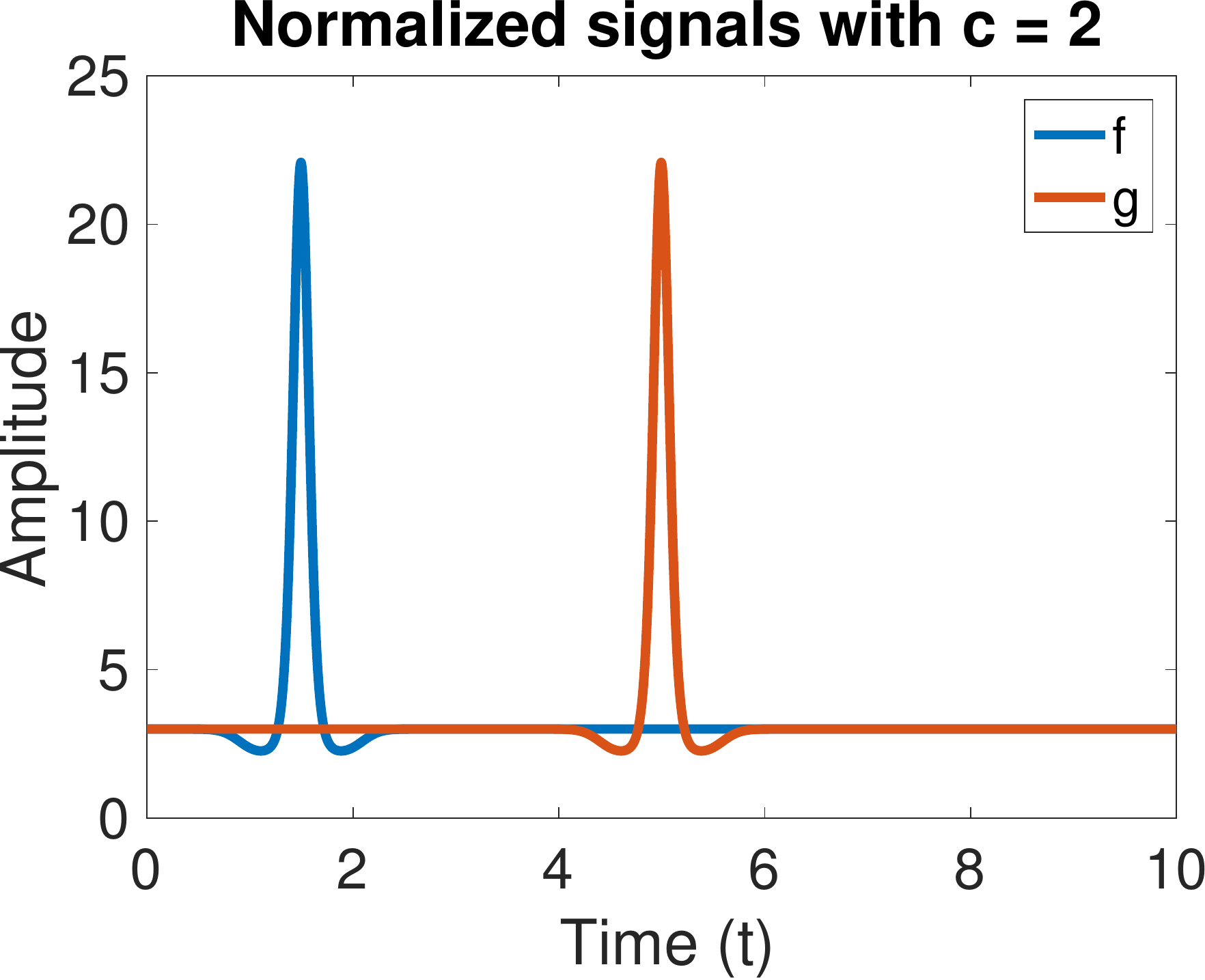}\label{fig:Huber_C2}}\hfill 
   \subfloat[][]{\includegraphics[height=0.18\textwidth]{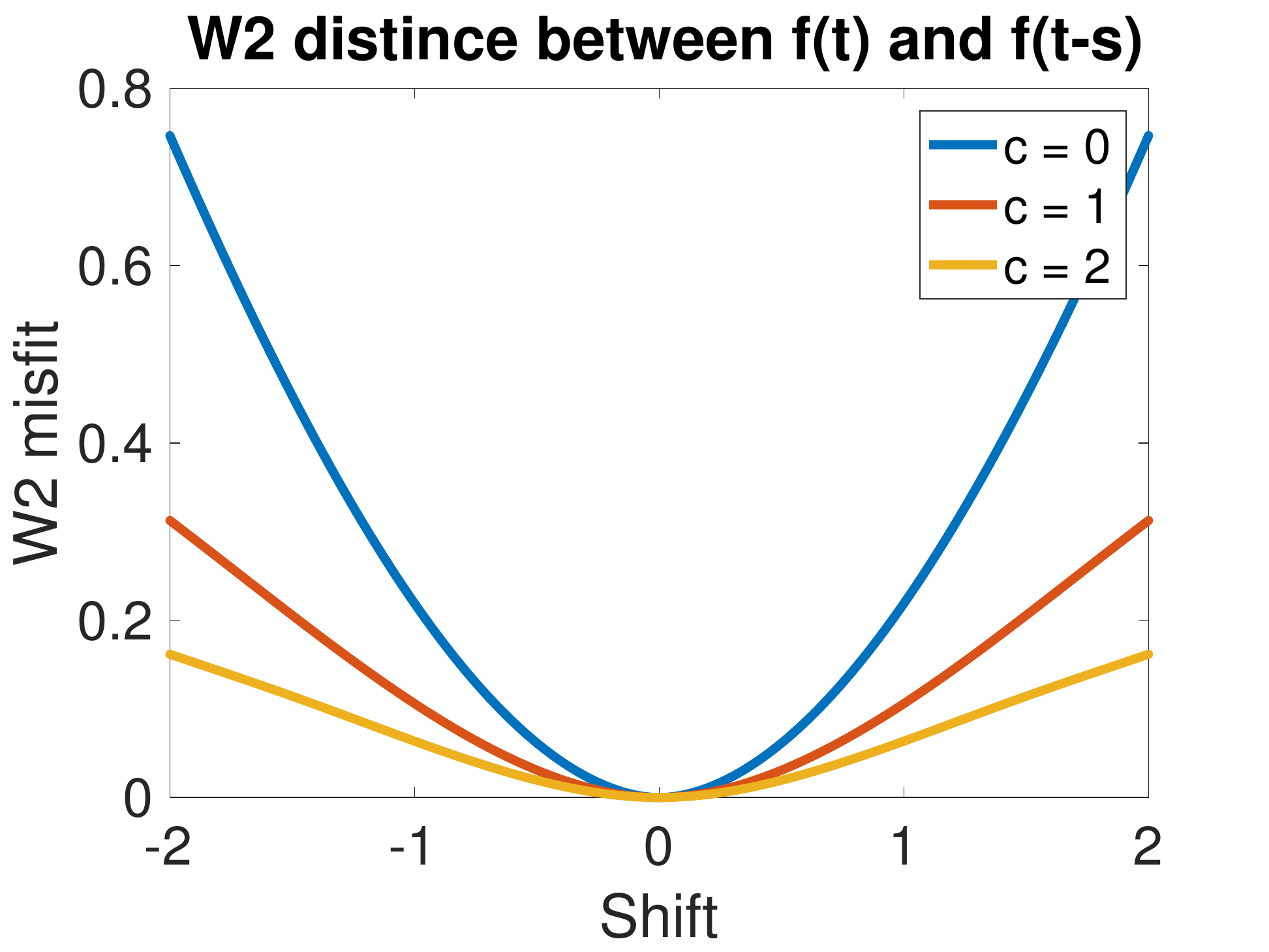}\label{fig:Huber_3c}}
  \caption{(a) The original signal $f$ and $g = f(t-s)$; (b) 
    normalized signals $\tilde{f}$ and $\tilde{g}$ with $c = 0$ and
    (c)~with $c = 2$ in~\eqref{eq:DN}; parameter $b$ is fixed in both
    (b) and (c). (d) \vph{$W(\tilde{f},\tilde{g})$}The Huber effect of linear constant $c$ on the
    loss function $W(\tilde{f},\tilde{g})$ as a function of the shift
    $s$.}
  \label{fig:Huber_1D_1}
\end{figure}

The Huber-type property is also often observed in numerical
tests. Consider functions $f$ and $g$ where $f$ is a single Ricker
wavelet and $g = f(t-s)$; see Figure~\ref{fig:Huber_data0}. We apply
the softplus scaling $\sigma_s$ in~\eqref{eq:softplus} to obtain
probability densities $\tilde{f}$ and $\tilde{g}$. According to
Corollary~\ref{thm: softplus}, with a proper choice of the
hyperparameter $b$, the convexity proved in Theorem~\ref{thm:
  biconvex} holds. Thus, the Huber-type property proved in
Theorem~\ref{thm:Huber} still applies. With $b$ fixed,
Figure~\ref{fig:Huber_3c} shows the optimization landscape of the
objective function with respect to shift $s$ for different choices of
$c$. We observe that as $c$ increases, the objective function becomes
less quadratic and more linear for the large shift, which conveys the
same message as Theorem~\ref{thm:Huber}. Figure~\ref{fig:Huber_C0}
shows the normalized Ricker wavelets for $c=0$; 
Figure~\ref{fig:Huber_C2} shows the data with constant $c=2$ applied
in~\eqref{eq:DN}. The major differences between the original signal
and the normalized ones are that we suppress the negative counterparts
of the wavelets while stretching the positive peaks. The phases remain
unchanged, and the original signal can be recovered since $\sigma_s$
is a one-to-one function that satisfies Assumption~\ref{ass:DN}.

\subsection{The Gradient-Smoothing Property}
In this section, we demonstrate another important property, which is
to improve the regularity of the optimal map $T$ as a result of
adding the constant $c$ in~\eqref{eq:DN}.

\begin{figure}
  \centering
  \subfloat[]{\includegraphics[width=0.45\textwidth]{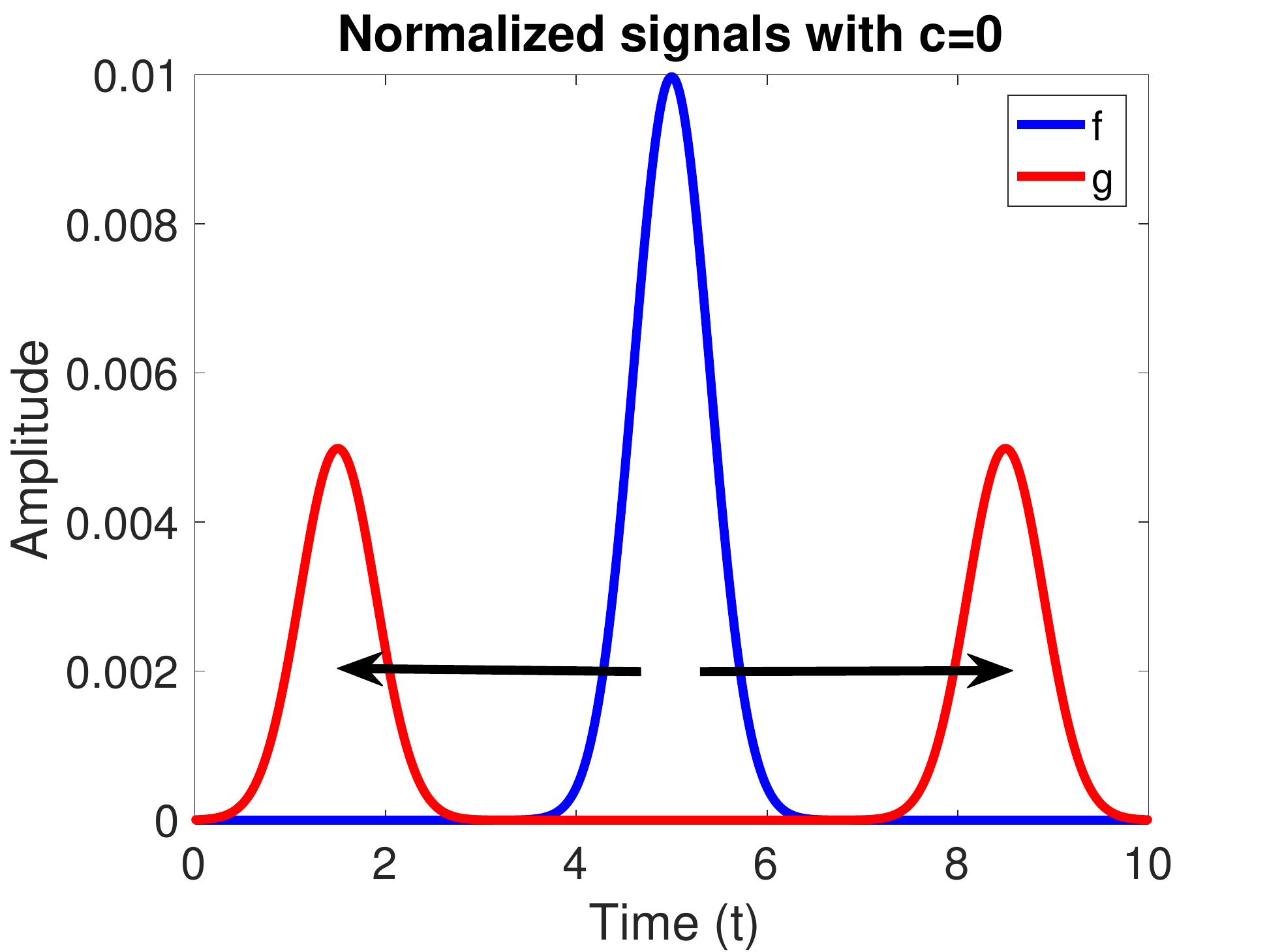}\label{fig:Huber_Map1}}
  \subfloat[]{\includegraphics[width=0.45\textwidth]{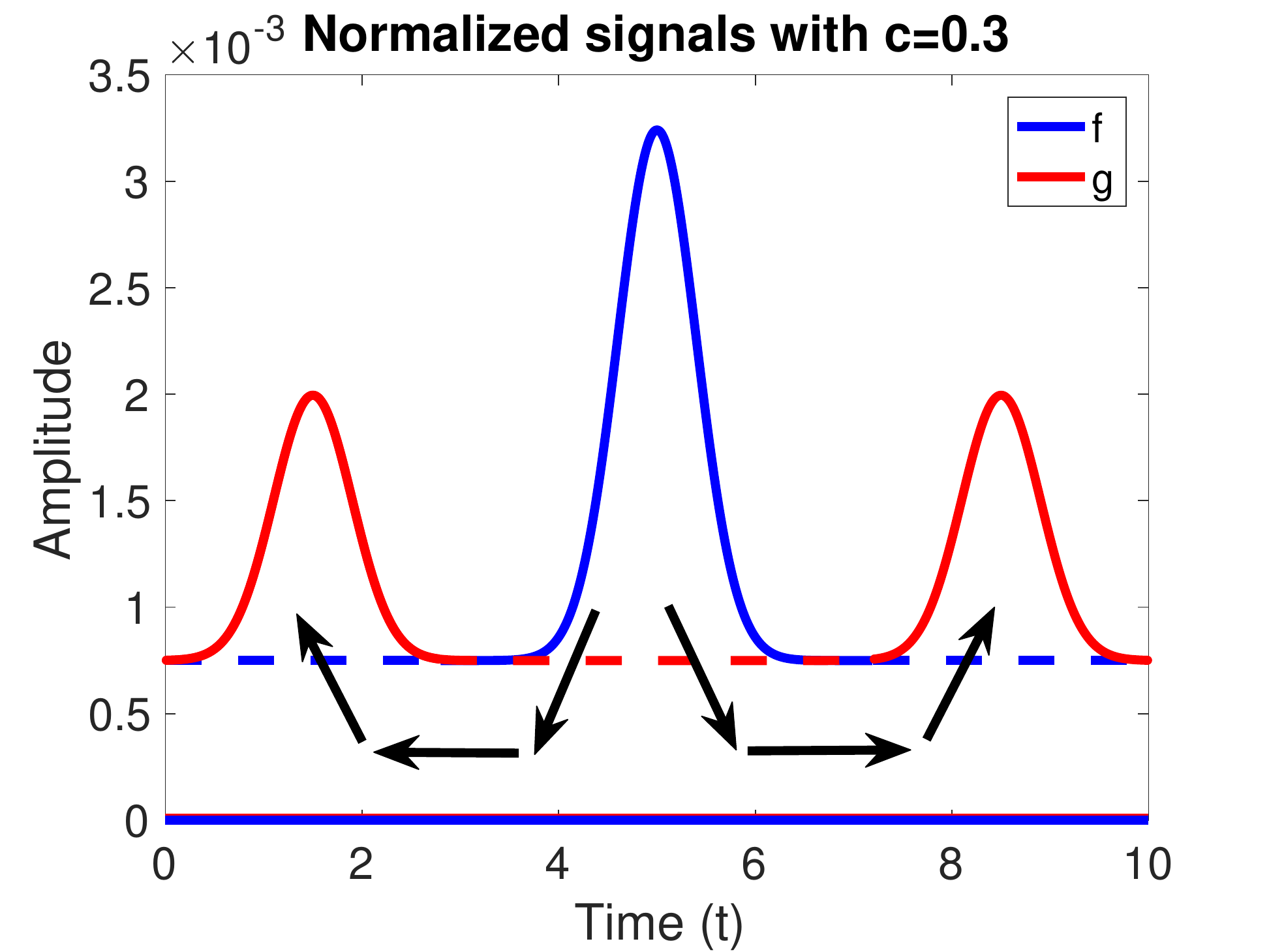}\label{fig:Huber_Map2}}\\[-2\jot]
  \caption{The black arrows represent the optimal map $T$. (a)
    Function $f$ is compactly supported on $[3.5,6.5]$. Function $g$
    is compactly supported on $[0,3]\cup[7,10]$. The optimal map
    between $f$ and $g$ is discontinuous at $t=5$. (b) Signals $f$ and
    $g$ after linear normalization. The optimal map $T$ becomes a
    continuous function by adding the constant $c=0.3$ as stated in
    Theorem~\ref{thm:map_smooth}.}\label{fig:Huber_Map}
\end{figure}

Based on the settings of the optimal transport problem, the optimal map is
often \textit{discontinuous} even if $f, g\in C^{\infty}$; see
Figure~\ref{fig:Huber_Map1} for an example in which the corresponding
cumulative distribution functions $F$ and $G$ are monotone but not
strictly monotone. However, by adding a positive constant $c$ to the
signals, we have improved the smoothness for the optimal map, as seen
in Figure~\ref{fig:Huber_Map2}. The smoothing property is based on the
following Theorem~\ref{thm:map_smooth}. Since the positive constant
$c$ is applied \textit{after} the signed functions are normalized to
be probability densities in Assumption~\ref{ass:DN}, we will hereafter
show the improved regularity by applying the linear
scaling~\eqref{eq:linear} to the probability densities with a positive
constant~$c$.

\begin{theorem}[The smoothing property in the
  map]\label{thm:map_smooth} Consider bounded probability density
  functions $f, g \in C^{k,\alpha} ({\Pi_2})$ where ${\Pi_2}$ is a
  bounded convex domain in $\mathbb{R}^d$ and $0<\alpha < 1$. If the
  normalized signals $\tilde{f}$ and $\tilde{g}$ are obtained by the
  linear scaling~\eqref{eq:linear} with $c>0$, then the optimal map
  $T$ between $\tilde{f}$and $\tilde{g}$ is $C^{k+1,\alpha}({\Pi_2})$.
\end{theorem}
\begin{proof}
  We first consider the case of $d=1$. By adding a nonzero constant to
  $f$ and $g$, we guarantee that the cumulative distribution function
  of $\tilde{f}$ and $\tilde{g}$, i.e., $\tilde{F}$ and $\tilde{G}$,
  to be strictly monotone and thus invertible in the classical
  sense. The optimal map $T$ is explicitly determined
  in~\eqref{eq:map1D}. Since $f,g \in C^{k,\alpha}$,
  $\tilde{F}$, $\tilde{G}$, $\tilde{F}^{-1}$, $\tilde{G}^{-1}$, and $T$ are
  all in $C^{k+1,\alpha}$.

  For higher dimensions $d\geq 2$, one can characterize the optimal
  map by the following \MA equation~\cite{brenier1991polar}:
  \begin{equation} \label{eq:MAA} \det (D^2 u(x)) =
  \frac{\tilde{f}(x)}{\tilde{g}(\nabla u(x))}, \quad x \in {\Pi_2}.
\end{equation}
Note that $\tilde{f}$ and $\tilde{g}$ are bounded from below by
  $\frac{c}{1 + c|{\Pi_2}|}$.  Thus,
  $
  {\tilde f}/{\tilde g}\in C^{k,\alpha} ({\Pi_2})$. Since $\Pi_2$
  is convex, Caffarelli's regularity theory for optimal transportation
  applies~\cite{caffarelli1992regularity}. Thus, the solution $u$
  to~\eqref{eq:MAA} is $C^{k+2,\alpha}$ and the optimal transport map
  $T = \nabla u$ is $C^{k+1,\alpha}$~\cite{brenier1991polar}.
\end{proof}

One advantage of adding a constant in the data
normalization~\eqref{eq:DN} is to enlarge $\mathcal{M}$, the set of
all maps that rearrange the distribution $f$ into $g$. For example,
$f$ and $g$ shown in Figure~\ref{fig:Huber_Map1} are $C^\infty$
functions, but all feasible measure-preserving maps between them are
discontinuous. After normalizing $f$ and $g$ by adding a positive
constant such that they are strictly positive, smooth rearrangement
maps are available, and the optimal one is a $C^{\infty}$ function for
the case in Figure~\ref{fig:Huber_Map2}.
\note{Please redo Figure 4.4(b) to correct the \\typo ``Spetrum'' $\to$ ``Spectrum''.}

The regularity of the optimal map is crucial for
optimal-transport-based seismic inversion, particularly for
low-wavenumber reconstruction. A smooth optimal map improves the
smoothness of the data Fr\'{e}chet derivative
${\delta J}/{\delta f}$ where $J$ is the squared $W_2$ metric. We
will prove the statement in Corollary~\ref{thm:gradient_smooth}. The
term ${\delta J}/{\delta f}$ is also the source of the adjoint
wave equation~\eqref{eq:FWI_adj}, which directly determines the
frequency content of the adjoint wavefield $w$. We recall that the
model gradient $\frac{\delta J}{\delta m}$ is computed as a
convolution of the forward and the adjoint wavefields, $u$ and $w$
in~\eqref{eq:adj_grad3}. Thus, a smoother source term for the adjoint
wave equation~\eqref{eq:FWI_adj} results in a smoother adjoint
wavefield $w$, and therefore a smoother model gradient
$\frac{\delta J}{\delta m}$. A smoother gradient can be seen as
low-pass filtering of the original gradient and thus focuses on
low-wavenumber content for the subsurface model update.

\begin{figure}
  \centering
  \subfloat[]{\includegraphics[width=0.45\textwidth]{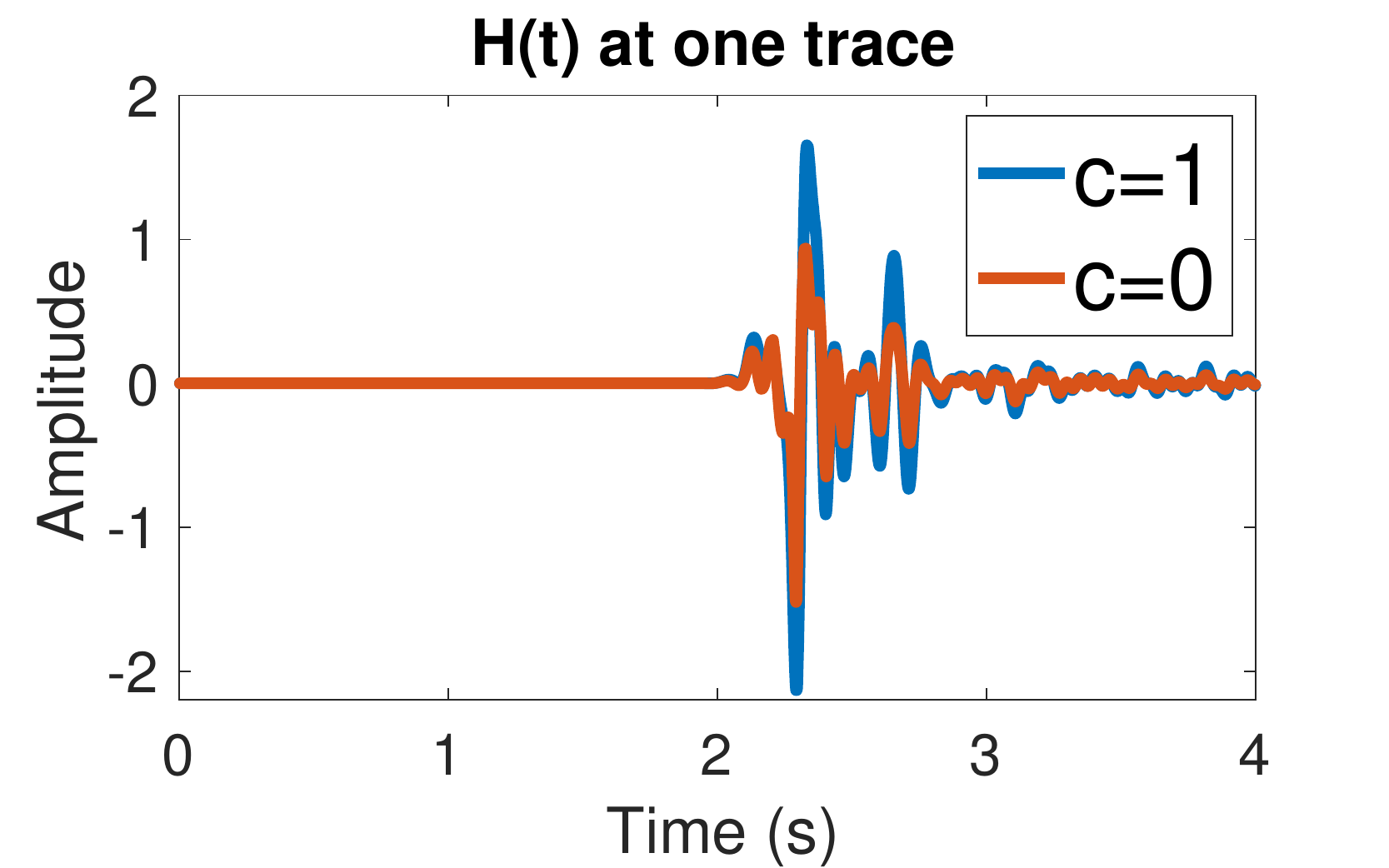}\label{fig:MarmDW2}}
  \hspace{0.04\textwidth}
  \subfloat[]{\includegraphics[width=0.45\textwidth]{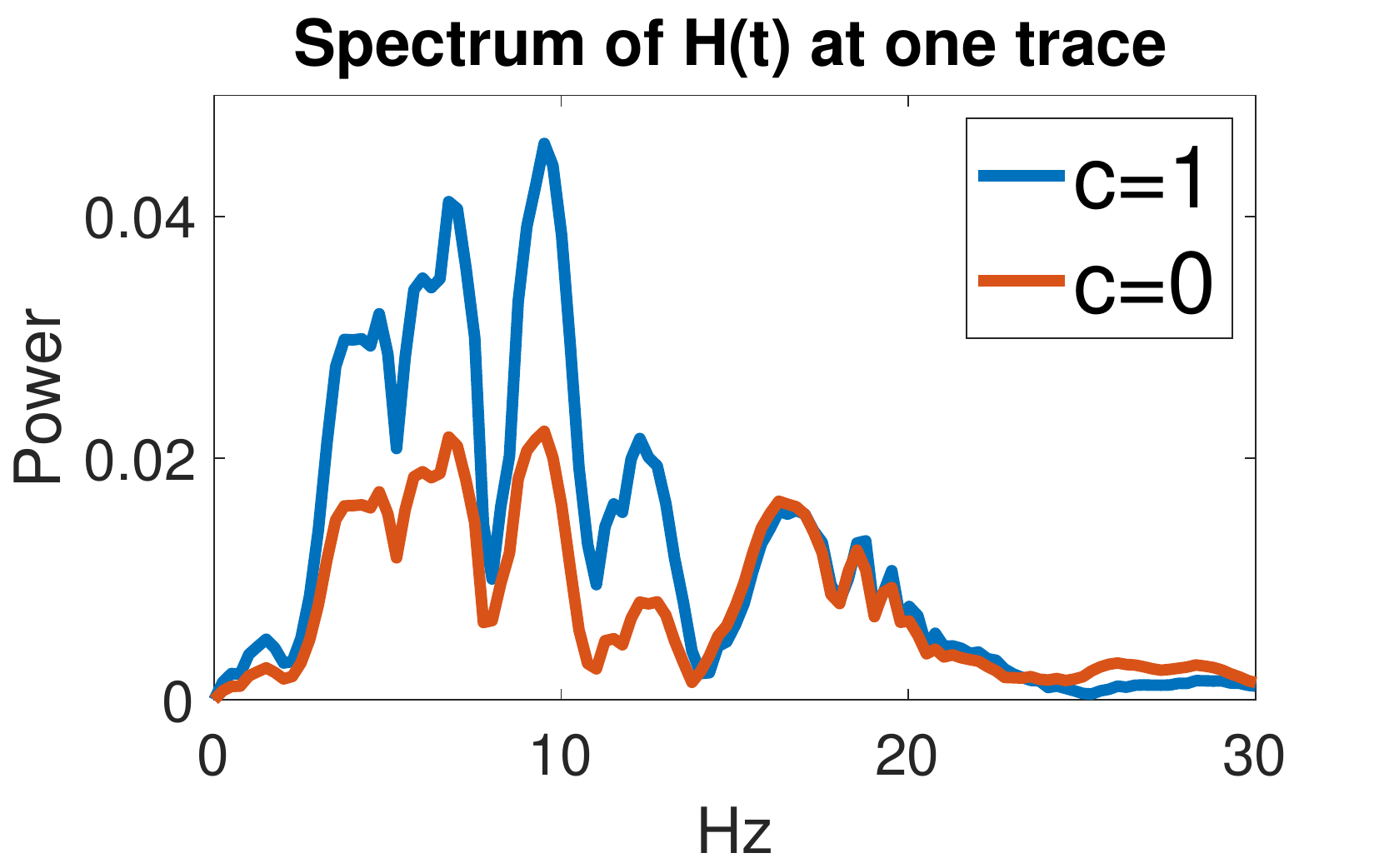}\label{fig:MarmDW2fft}}\\
  \caption{(a)~The $W_2$ data gradient~\eqref{eq:W2_Adj2} by linear
    scaling with $c=1$ (blue) and $c=0$ (red); (b) the Fourier
    transform of the functions in (a).}\label{fig:MarmDW2all}
\vspace{-1cm}
\end{figure}

Figure~\ref{fig:MarmDW2all} shows a 1D example in which the adjoint
source has stronger low-frequency modes after data normalization
following~\eqref{eq:DN} with $c>0$ than the one corresponding to the
normalization with $c=0$. The following corollary states the general
smoothing effect in the gradient by adding a positive constant $c$ in
the normalization of data.\smallskip

\begin{corollary}[The smoothing effect in the
  gradient]~\label{thm:gradient_smooth} Under the same assumptions of
  Theorem~\ref{thm:map_smooth}, the Fr\'{e}chet derivative of
  $W_2(\tilde{f},\tilde{g})$ with respect to $f$ is $C^{k+1,\alpha}$.
\end{corollary}
\begin{proof}
  By chain rule of the Fr\'{e}chet derivative, we have
\begin{equation}
  ~\label{eq:W2_Adj} I = \frac{\delta
    W_2^2(\tilde{f},\tilde{g})}{\delta f}= \frac{\delta
    W_2^2(\tilde{f},\tilde{g})}{\delta \tilde{f}} \frac{\delta
    \tilde{f}}{\delta f}.  \end{equation}

Since
  $\frac{\delta \tilde{f}}{\delta f} \in C^{\infty}$ for all scaling
  methods that satisfy~Assumption~\ref{ass:DN}, we only check the
  regularity of
\begin{equation} ~\label{eq:W2_Adj2} H = \frac{\delta
    W_2^2(\tilde{f},\tilde{g})}{\delta \tilde{f}}.  \end{equation}

  We first consider the case of $d=1$, where $H(t)$ is a 1D function
  defined on $\R$. Based on the 1D explicit
  formula~\eqref{eq:cost1D}, we can write down $H$ explicitly by
  taking the first variation~\cite{yang2017application}. An important
  observation is that
\begin{equation} ~\label{eq:W2_Adj1} \frac{dH}{dt} = 2\ (t-
  G^{-1}F(t) ) = 2(t - T(t)), \end{equation}
where $T$ is the optimal map between
  $\tilde{f}$ and $\tilde{g}$. By Theorem~\ref{thm:map_smooth},
  $T\in C^{k+1,\alpha}$, and thus $H\in C^{k+2,\alpha}$, which proves
  that $I$ in~\eqref{eq:W2_Adj2} is at least $C^{k+1,\alpha}$.

  Next we discuss the case of $d\geq 2$. One can linearize the \MA
  equation in~\eqref{eq:MAA} and obtain a second-order linear elliptic
  PDE of $\psi$~\cite{engquist2016optimal}:
\begin{equation} \label{eq:du}
  \begin{cases}
    \tilde{g}(\nabla u)\tr\left((D^2u)_{adj}D^2\psi\right) + \det(D^2 u)\nabla \tilde{g}(\nabla u)\cdot \nabla\psi = \df,\\
    \nabla \psi \cdot \mathbf{n} = 0\quad \text{on}\ \partial {\Pi_2},
  \end{cases}
\end{equation}
where $A_\text{adj} = \det(A)A^{-1}$ is the adjugate of matrix
  $A$, $u$ is the solution to~\eqref{eq:MAA}, and $\df$ is the
  \textit{mean-zero} perturbation to $\tilde{f}$. If we
  denote~\eqref{eq:du} as a linear operator $\Lf$ where
  $\Lf \psi = \df$, $H$ in~\eqref{eq:W2_Adj2} becomes
  \[
    H = \abs{x-\nabla u}^2-2(\Lf^{-1})^*(\nabla\cdot(
        (x-\nabla u)\tilde{f})).
  \]
  In Theorem~\ref{thm:map_smooth}, we have already shown that
  $u\in C^{k+2,\alpha}$. The right-hand side of the elliptic operator,
  $\nabla\cdot( (x-\nabla u)\tilde{f})$, is then in
  $C^{k,\alpha}$. As a result, $H$ is $C^{k+2,\alpha}$ if $\Lf$ is not
  degenerate, and $C^{k+1,\alpha}$ if it is
  degenerate~\cite{ERY2019}. Therefore, the Fr\'{e}chet derivative of
  $W_2(\tilde{f},\tilde{g})$ with respect to $f$ is at least
  $C^{k+1,\alpha}$ for $d\geq 1$.
\end{proof}

Theorem~\ref{thm:map_smooth} and Corollary~\ref{thm:gradient_smooth}
demonstrate the smoothing effects achieved by adding a positive $c$
in~\eqref{ass:DN}. Nevertheless, as we have shown in~\cite{ERY2019},
the ``smoothing'' property plays a much more significant role in using
the quadratic Wasserstein metric as a measure of data discrepancy in
computational solutions of inverse problems.  We have characterized
in~\cite{ERY2019}, analytically and numerically, many principal
benefits of the $W_2$ metric, such as the insensitivity to noise
(Theorem~\ref{thm:noise}) and the convex optimization landscape
(Theorem~\ref{thm: biconvex}), can be explained by the intrinsic
smoothing effects of the quadratic Wasserstein metric. In
Section~\ref{sec:three-layer}, we will see how the smoothing property,
which is related to the frequency bias of the $W_2$ metric, keeps
tackling the challenging inversion scenarios that are noise-free and
beyond the scope of local-minima trapping.

\section{Model Recovery Below the Reflectors} \label{sec:three-layer}
Subsurface velocities are often discontinuous, which arises naturally
due to the material properties of the underlying physical system. Due
to data acquisition limitations, seismic reflections are often the
only reliable information to interpret the geophysical properties in
deeper areas. Inspired by the realistic problem with salt inclusion,
we create a particular layered model whose velocity only varies
vertically. Despite its simple structure, it is notably challenging
for conventional methods to invert with reflections. No seismic waves
return to the surface from below the reflecting layer. Nevertheless,
we will see that partial inversion for velocity below the reflecting
interface is still possible by using $W_2$ as the objective function
in this PDE-constrained optimization.

\begin{figure}[h]
  \centering \subfloat[True velocity]{\includegraphics[width
    =0.48\textwidth]{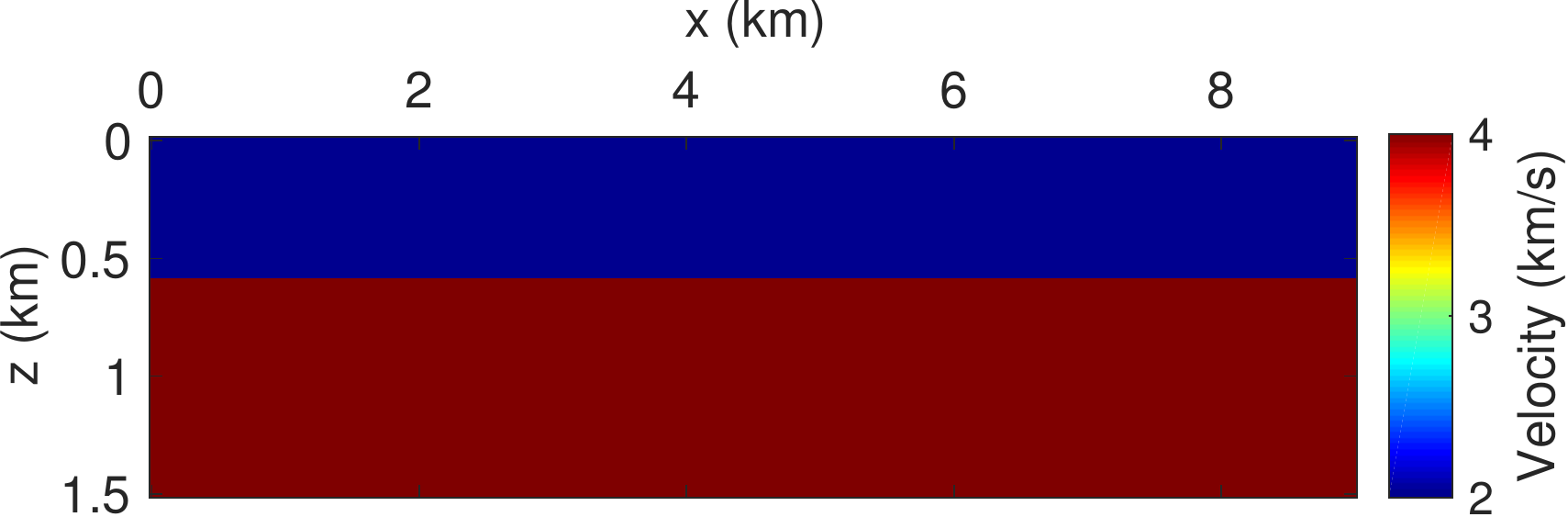}\label{fig:test1_vG}}\quad 
  \subfloat[Initial velocity]{\includegraphics[width =0.48\textwidth]{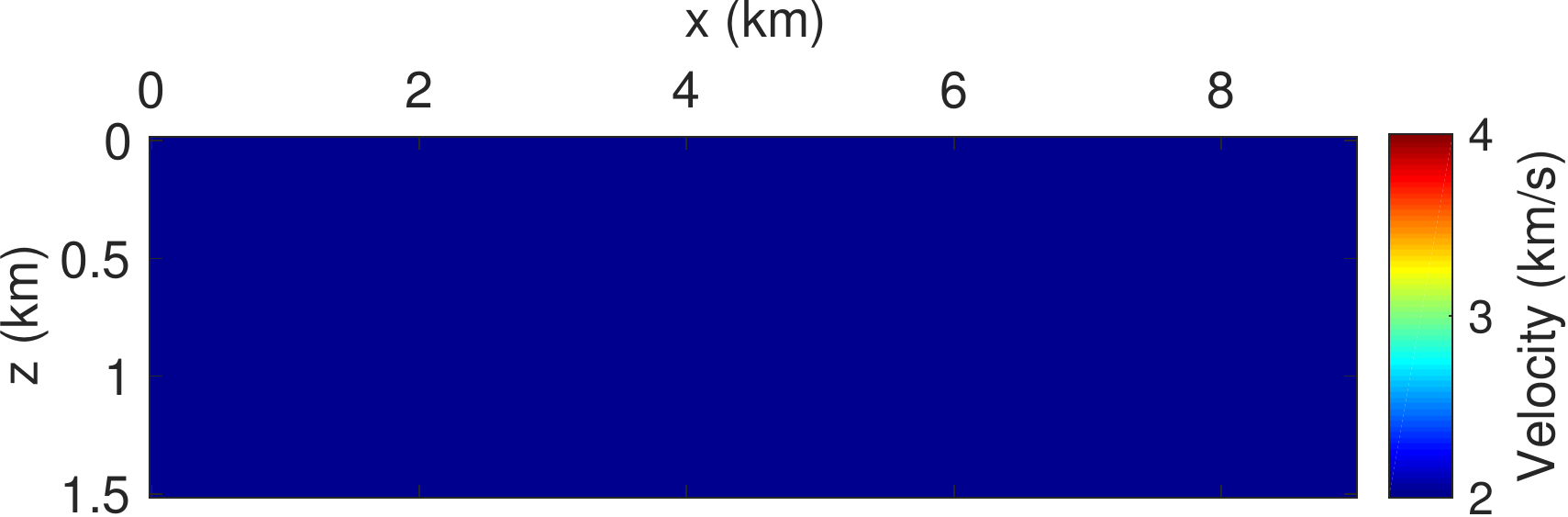}\label{fig:test1_vF}}   \\
  \subfloat[$L^2$-based inversion]{\includegraphics[width
    =0.48\textwidth]{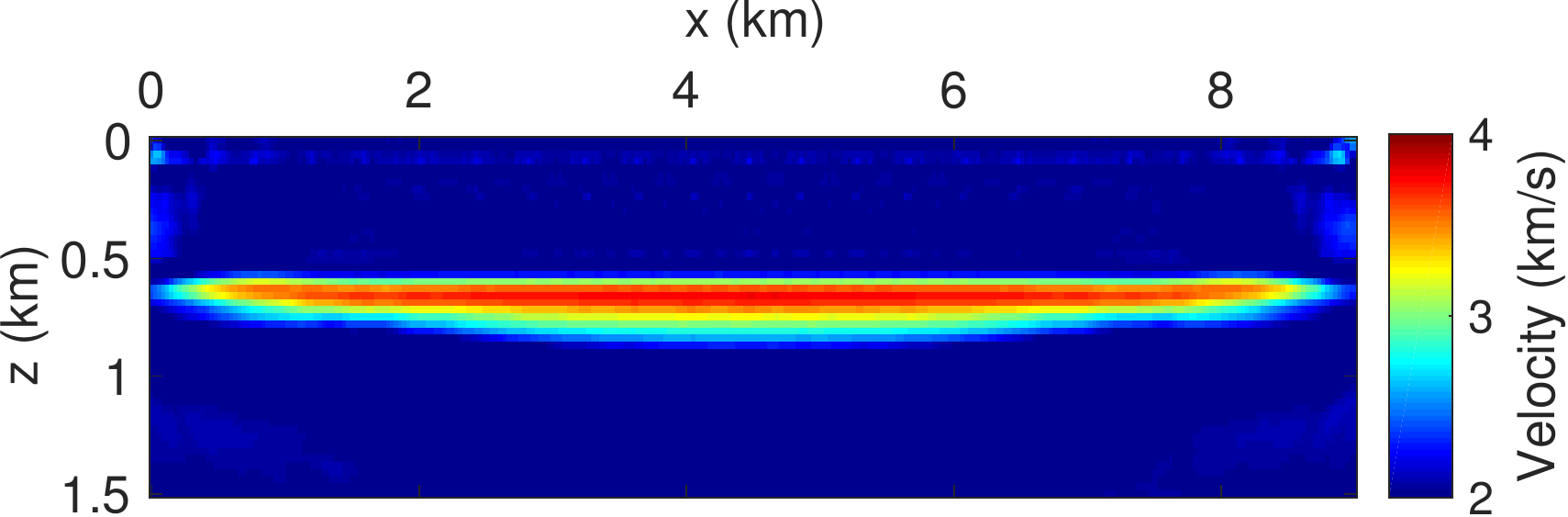}\label{fig:test1_L2}}\quad 
  \subfloat[$W_2$-based inversion]{\includegraphics[width
    =0.48\textwidth]{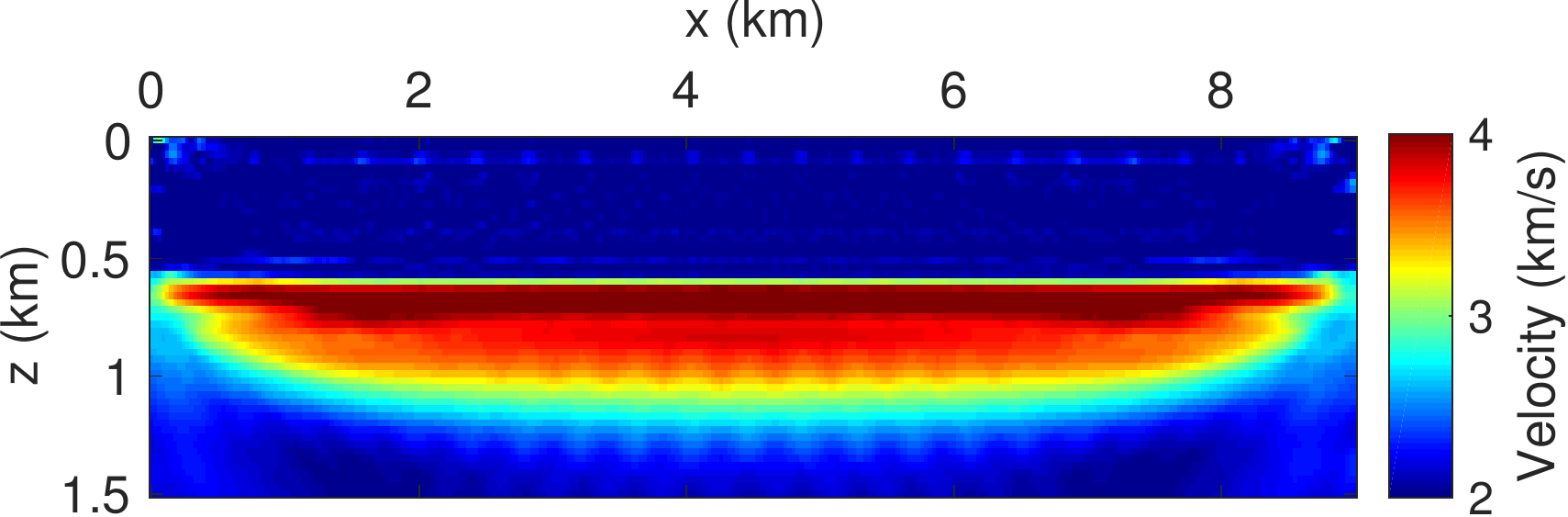}\label{fig:test1_W2}}
  \caption{The layered model: (a) true velocity, (b) initial velocity,
    (c)~$L^2$-based velocity inversion, and (d)~$W_2$-based velocity
    inversion.} \label{fig:test1_vel2D}
\end{figure}

\subsection{The Layered Example}
The target velocity model we aim to reconstruct contains two layers of
homogenous constant velocity; see Figure~\ref{fig:test1_vG}. The
second layer with wave speed 4 km/s is unknown to the initial model
(Figure~\ref{fig:test1_vF}). The wave source in the test is a Ricker
wavelet with a peak frequency of 5 Hz. There are 23 sources and 301
receivers on top in the first layer with wave speed 2 km/s. The total
recording time is 3 seconds. There is naturally no back-scattered
information from the interior of the second layer returning to the
receivers. Due to both the physical and numerical boundary
conditions~\cite{engquist1977absorbing}, reflections from the
interface are the only information for the reconstruction.

The numerical inversion is solved iteratively as an optimization
problem. Both experiments are manually stopped after 260
iterations. Figure~\ref{fig:test1_L2} shows the final result using the
$L^2$ norm. The vertical velocity changes so slowly that it does not
give much more information than indicating the location of the
interface. Nevertheless, the $W_2$ result in Figure~\ref{fig:test1_W2}
gradually recovers not only the layer interface but also the majority
of the sublayer velocity.

\begin{figure}[tb]
  \subfloat[Objective functions $J_{L^2}$ and
  $J_{W_2}$]{\includegraphics[width=0.45\textwidth]{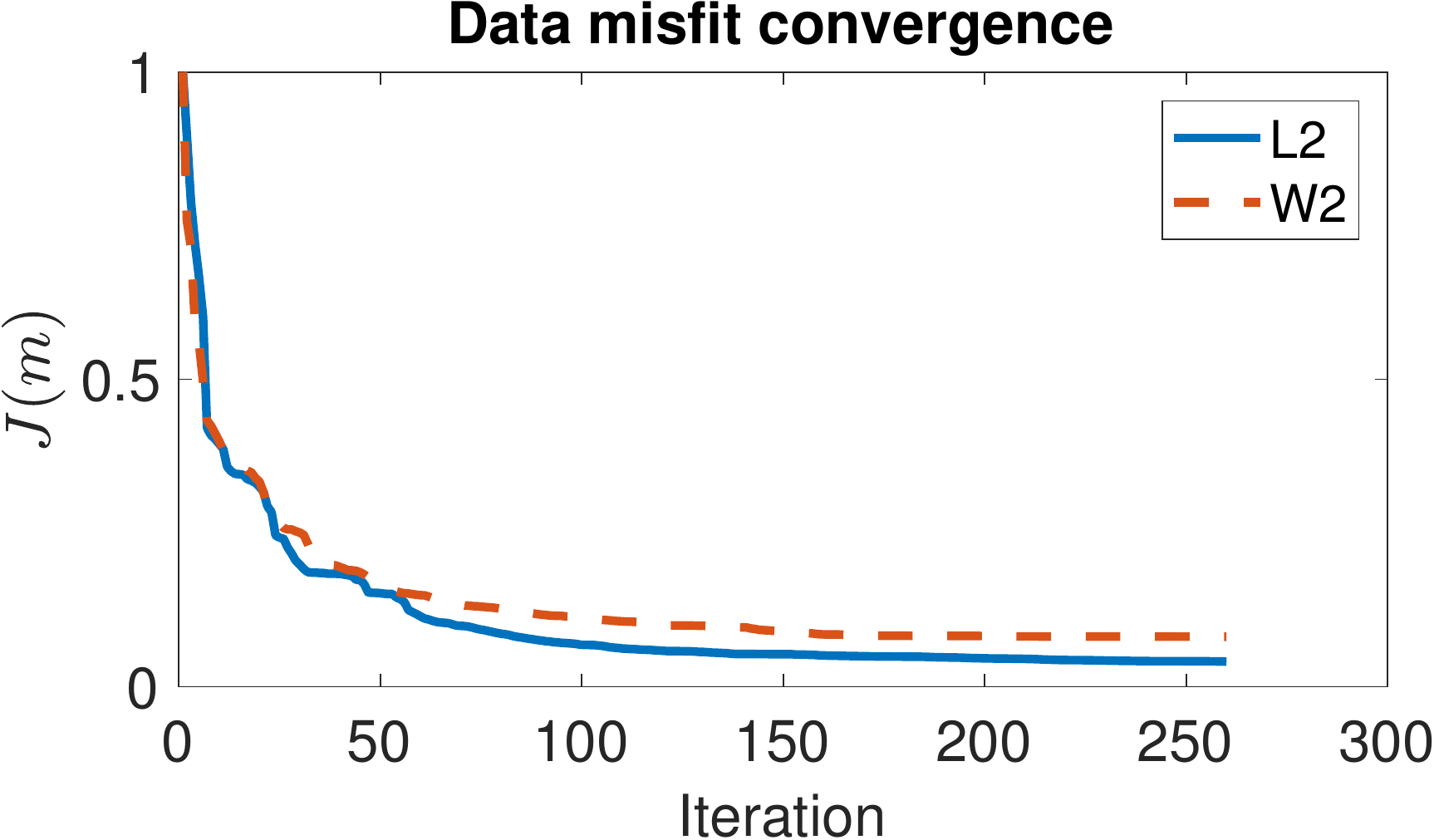}\label{fig:CPAM_data_conv}}\qquad 
  \subfloat[Model error
  $\|m_\text{iter}-m_*\|_F$]{\includegraphics[width=0.45\textwidth]{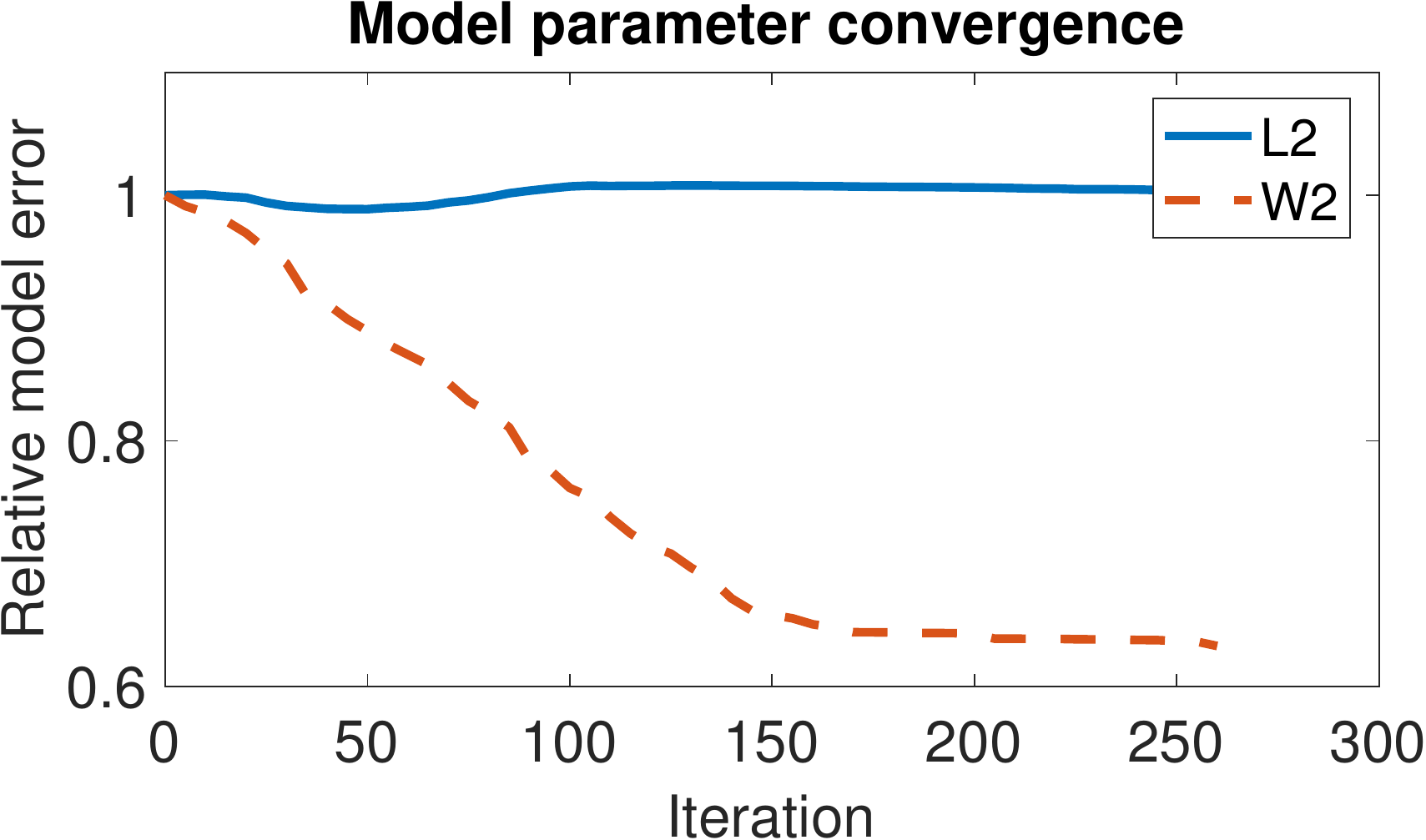}\label{fig:CPAM_model_conv}}
  \caption{(a) Normalized data misfit convergence curve based on the
    objective function and (b) normalized $2$-norm of the model
    Error.}~\label{fig:errors}
\end{figure}

\subsection{Tackling Issues Beyond Local Minima}
In Figure~\ref{fig:CPAM_data_conv}, both the normalized $L^2$ norm and
$W_2$ distance are reduced from~1 to nearly 0. The plots indicate that
$L^2$-based inversion \textit{does not suffer from local minima} in
this layered example. The velocity is initiated correctly above the
interface. In Figure~\ref{fig:CPAM_model_conv}, $L^2$-based inversion
has a model convergence curve that is radically different from its
data convergence in Figure~\ref{fig:CPAM_data_conv}. The normalized
model error, measured by the Frobenius norm of $m_\text{iter}-m_*$
where $m_\text{iter}$ is the reconstruction at the current iteration
and $m_*$ is the truth, remains unchanged. On the other hand, the
$W_2$-based inversion has both the $W_2$ distance and the model error
decreasing rapidly in the first 150 iterations.  Since the
computational cost per iteration is the same in both cases by
computing the $W_2$ distance explicitly in
1D~\cite{yang2017application}, the $W_2$-based inversion lowers the
model error much more quickly. Both the inversion results in
Figure~\ref{fig:test1_vel2D} and the convergence curves
in~Figure~\ref{fig:errors} illustrate that the $W_2$-based inversion
gives a better reconstruction. Other features of $W_2$ other than the
convexity for translations and dilations (Theorem~\ref{thm: biconvex})
play 
critical roles in this velocity inversion.

At first glance, it seems puzzling that $W_2$-based FWI can even
recover velocity in the model where no seismic wave goes through. The
fact that there is no reflection from below the known interface in the
measured data is, of course, also informative. After analyzing the
layered example more carefully, we have summarized two essential
properties of the quadratic Wasserstein metric that contribute to the
better inversion result, \textit{small-amplitude sensitivity} and
\textit{frequency bias}.

\begin{figure}[h]
  \centering \subfloat[]{\includegraphics[width =
    0.45\textwidth]{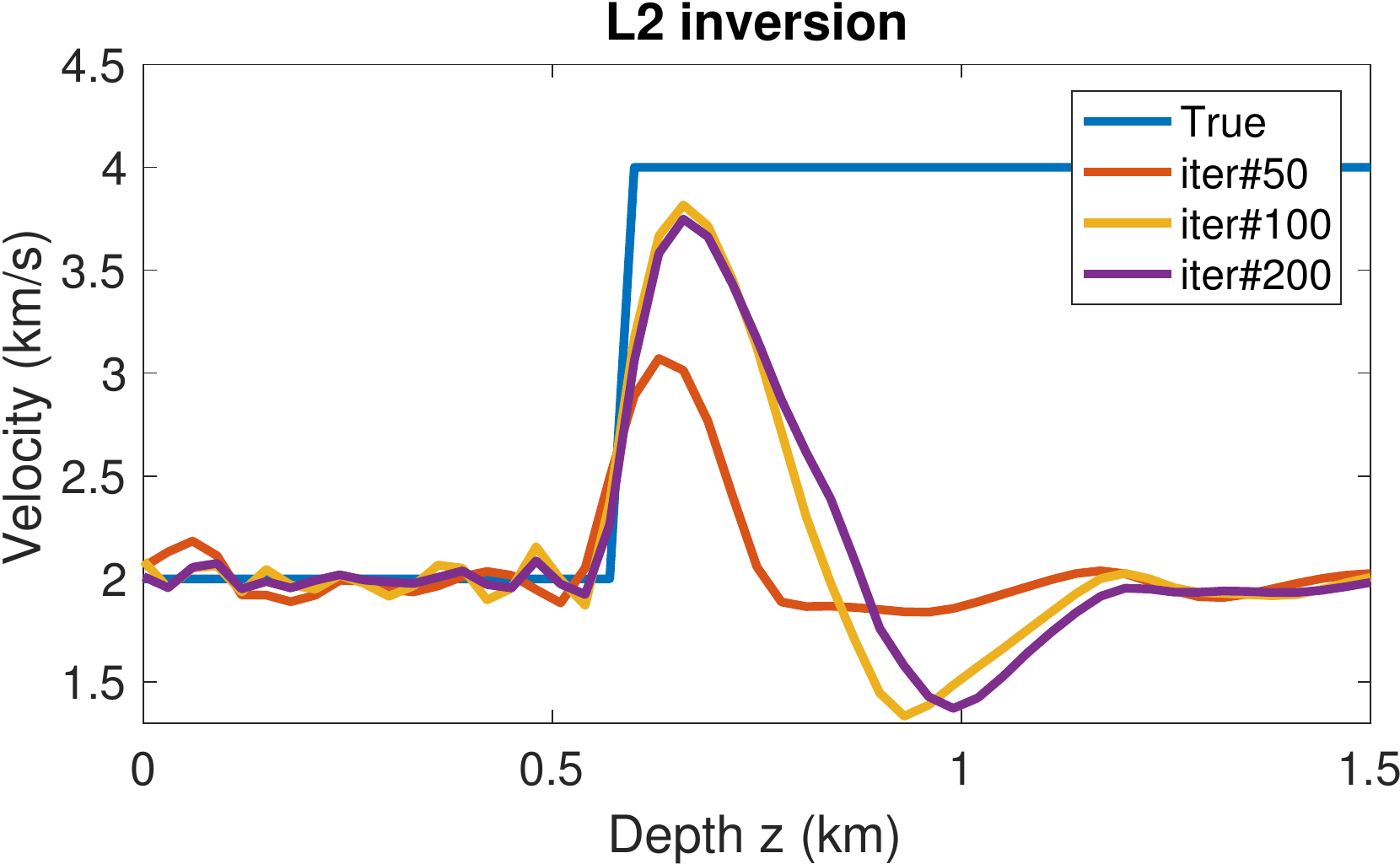}\label{fig:test1_L2_iter}}
  \subfloat[]{\includegraphics[width =
    0.45\textwidth]{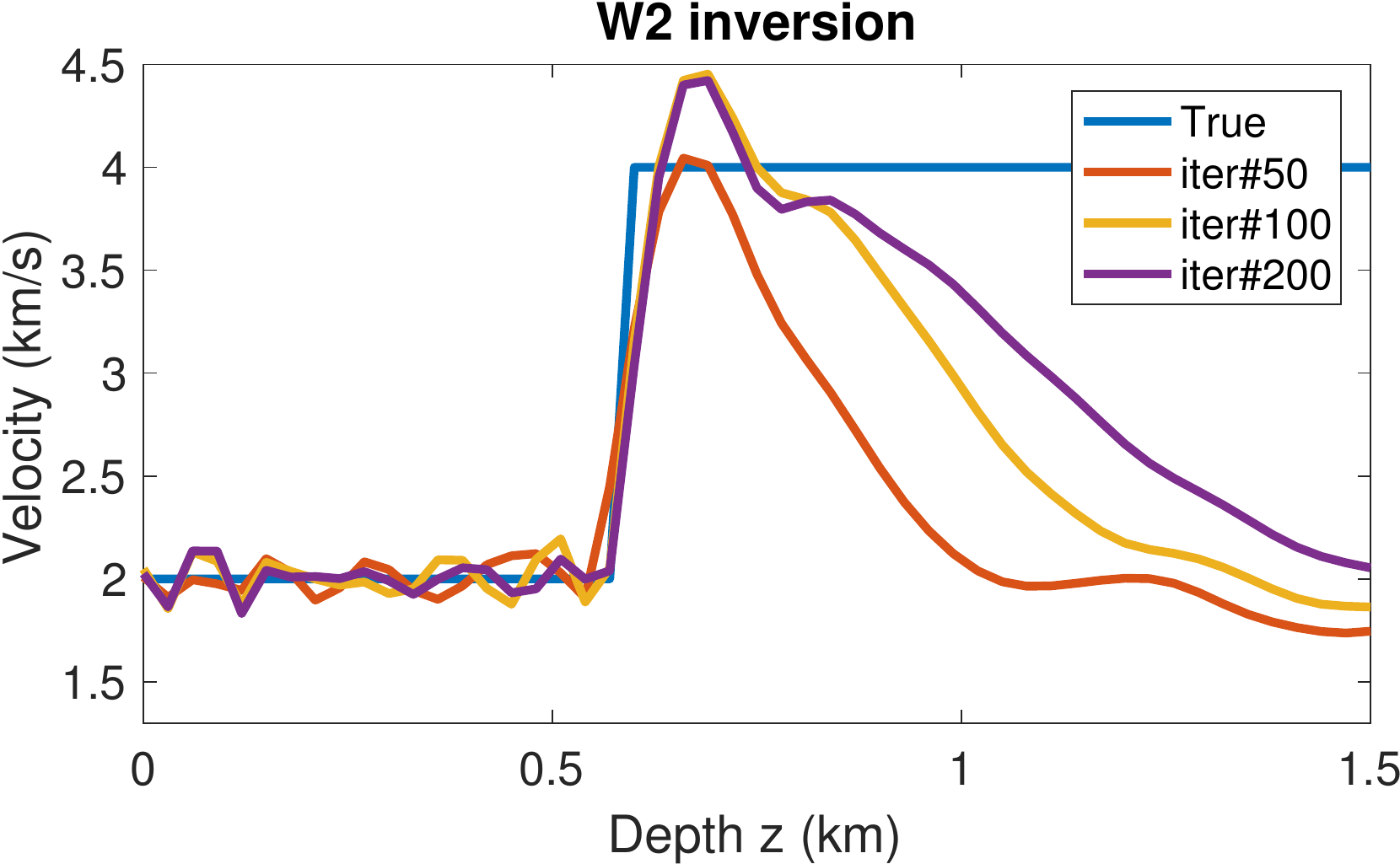}\label{fig:test1_W2_iter}}
  \caption{The layered model: (a)~$L^2$ and (b) $W_2$ inversion
    velocity after 50, 100, and 200 iterations.}\label{fig:test1_iter}
\end{figure}

\begin{figure}[h]
  \centering
  \includegraphics[width=0.93\textwidth]{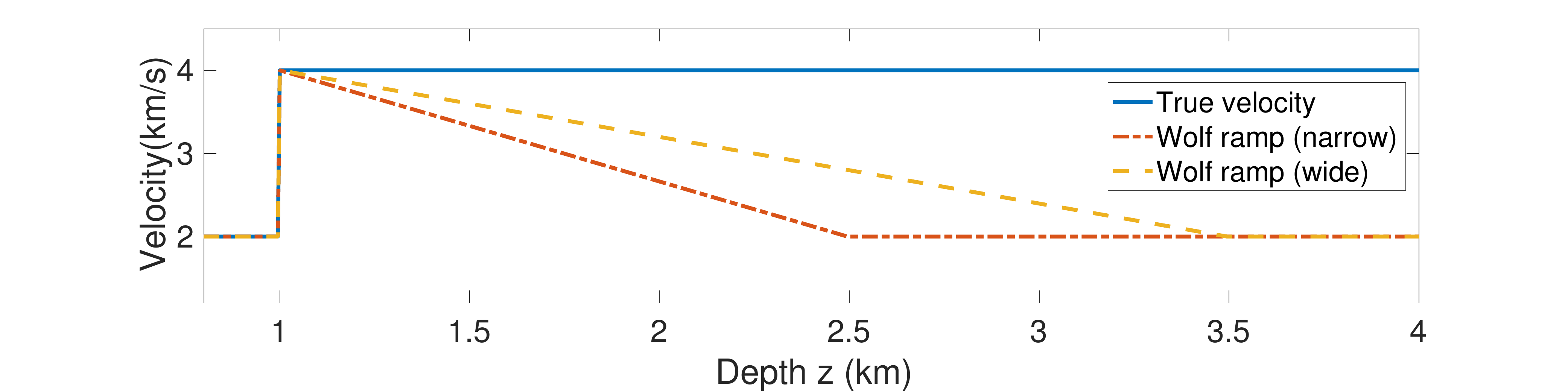} \caption{Analogy
    for the true two-layer velocity (solid line), the narrower
    (dash-dot line), and the wider Wolf ramps (dashed
    line).}\label{fig:Wolf-vel}
\end{figure}

\subsection{Small-Amplitude Sensitivity} \label{sec:SmallAmplitude}
Figure~\ref{fig:test1_L2_iter} and Figure~\ref{fig:test1_W2_iter} are
the vertical velocity profiles of $L^2$-based and $W_2$-based FWI at
the 50th, 100th, and 200th iterations. The inverted velocity profile at
the 50th iteration has a gradual transition from 4 km/s back to 2 km/s
around $z=1$ km. The simulated reflector is present in the earlier
iterations of both $L^2$- and $W_2$-based inversion. It is often
referred to as overshoot. However, this simulated reflector is the key
information to uncover the velocity model below the interface.

Although the inverted velocity at 50th iteration does not have the
discontinuity at $z=0.6$ km as in the true velocity, there is
a \textit{ frequency-dependent reflectivity} caused by the linear
velocity transition zone from $z=0.5$ km to $z=1$ km. In 1937 Alfred
Wolf~\cite{wolf1937reflection} first analyzed this particular type of
reflectivity dispersion. We extract the linear transition zones and
create the analogous versions in Figure~\ref{fig:Wolf-vel}. The solid
plot in Figure~\ref{fig:Wolf-vel} represents our target model, while
the dashed plots are two types of Wolf ramps, similar to the velocity
profiles at the 50th and the 100th iteration in
Figure~\ref{fig:test1_iter}.

All three velocity models in Figure~\ref{fig:Wolf-vel} produce strong
reflections of the same phase due to the jump in velocity at $z=1$
km. However, the energy reflected is relatively smaller from the ones
with Wolf ramps. In addition to that, the linear transition zone
generates another reflection which has extremely small energy. If the
amplitude of the difference in reflection is $\varepsilon$, the $L^2$
misfit is $\bO(\varepsilon^2)$ while the $W_2$ misfit is
$\bO(\varepsilon)$. Since the reflection amplitude
$\varepsilon \ll 1$, the $W_2$ metric measures the misfit
$\bO(\varepsilon) \gg \bO(\varepsilon^2)$. Consequently, $W_2$-based
inversion can correct the velocity model furthermore based on the
relatively bigger residual.

\begin{figure}
  \centering \subfloat[Reflections from the narrow Wolf ramp]
  {\includegraphics[width=0.46\textwidth]{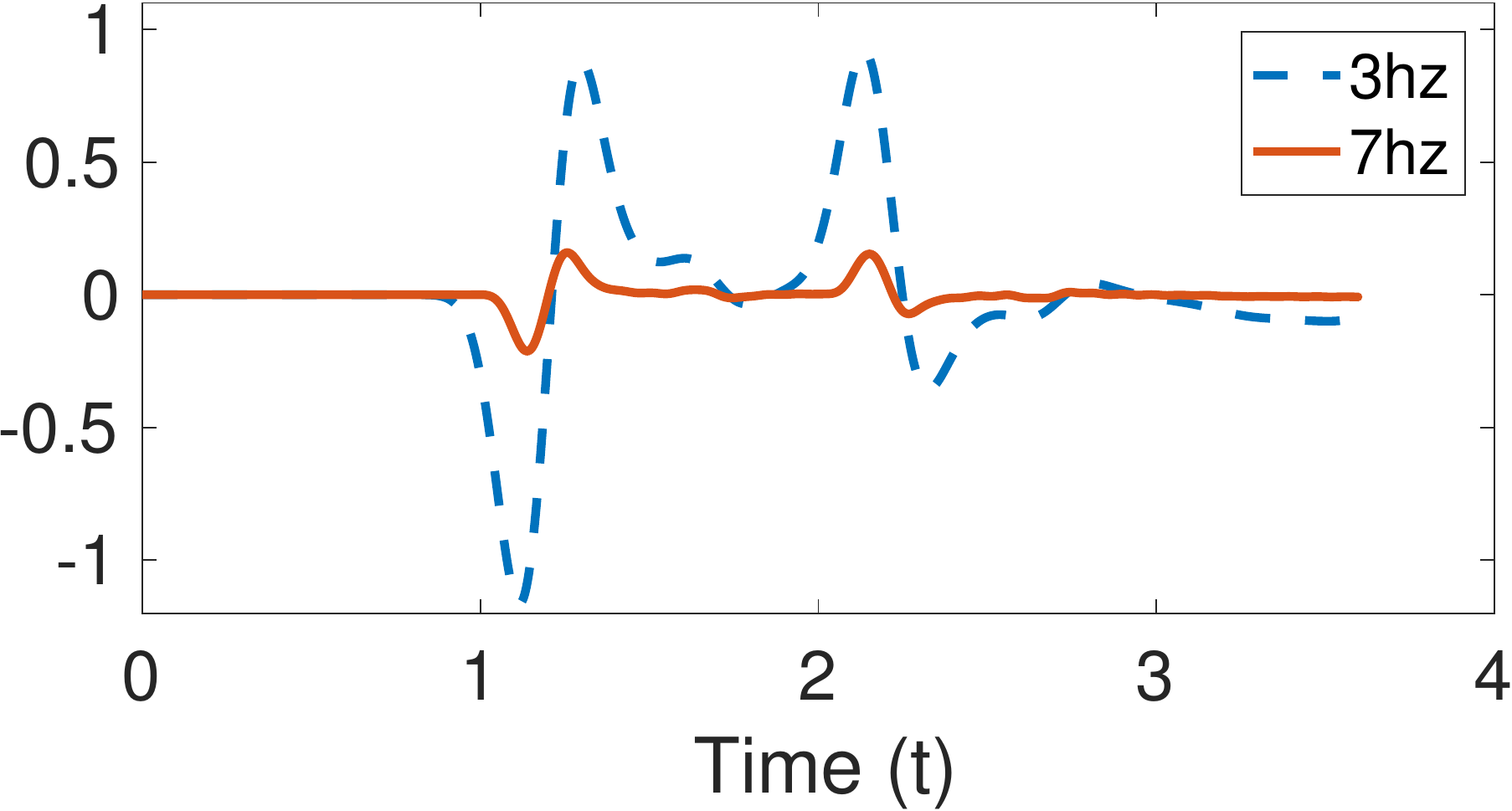}\label{fig:Wolf-wide}}\quad 
  \subfloat[Reflections from the wide Wolf
  ramp]{\includegraphics[width=0.46\textwidth]{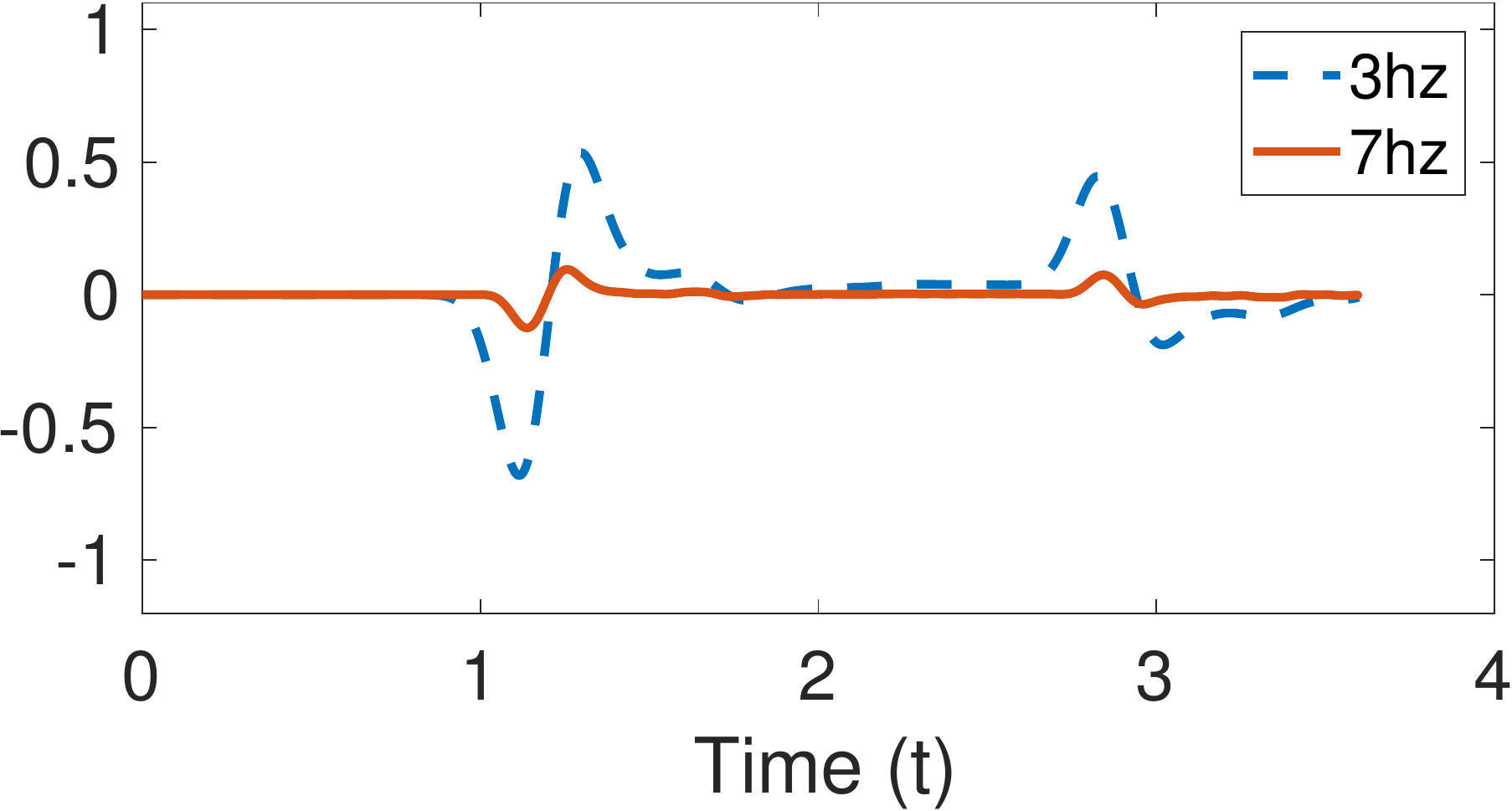}\label{fig:Wolf-narrow}}
  \caption{Reflections produced by the 3 Hz and the 7 Hz Ricker wavelets
    from the (a) narrow and (b) wide Wolf ramps in
    Figure~\ref{fig:Wolf-vel}.}~\label{fig:Wolf-data}
   \vspace{-7ex}
\end{figure}

\subsection{Frequency bias}
As illustrated in~\cite[fig.~2]{liner2010wolf}, the amplitude of the
Wolf ramp reflection coefficient at the normal incident is bigger for
lower frequencies and smaller for higher frequencies. This is also
observed in Figure~\ref{fig:Wolf-data}. The difference between the
true data reflection and the one from Wolf ramps is more significant
for 3-Hz data than the 7-Hz one. When the simulated data and the
observed data are sufficiently close to each other, inverse matching
with the quadratic Wasserstein metric can be viewed as the weighted
$\dot H^{-1}$ seminorm~\cite{peyre2018comparison,ERY2019}, which has a
$1/\mathbf{k}$ weighting on the data spectrum with $\mathbf{k}$
representing the wavenumber. As a result, the $W_2$ objective function
``sees'' more of the stronger low-frequency reflections caused by the
Wolf ramp than the $L^2$ norm. Based on its better sensitivity to the
low-frequency modes, the $W_2$-based inversion can keep updating the
velocity model and reconstruct the second layer
in~Figure~\ref{fig:test1_vG} by minimizing the ``seen'' data misfit.

\begin{figure}
  \centering
  \subfloat[$\bigl\|m^\text{low}_\text{iter}-m^\text{low}_*\bigr\|_F$]{\includegraphics[width=0.45\textwidth]{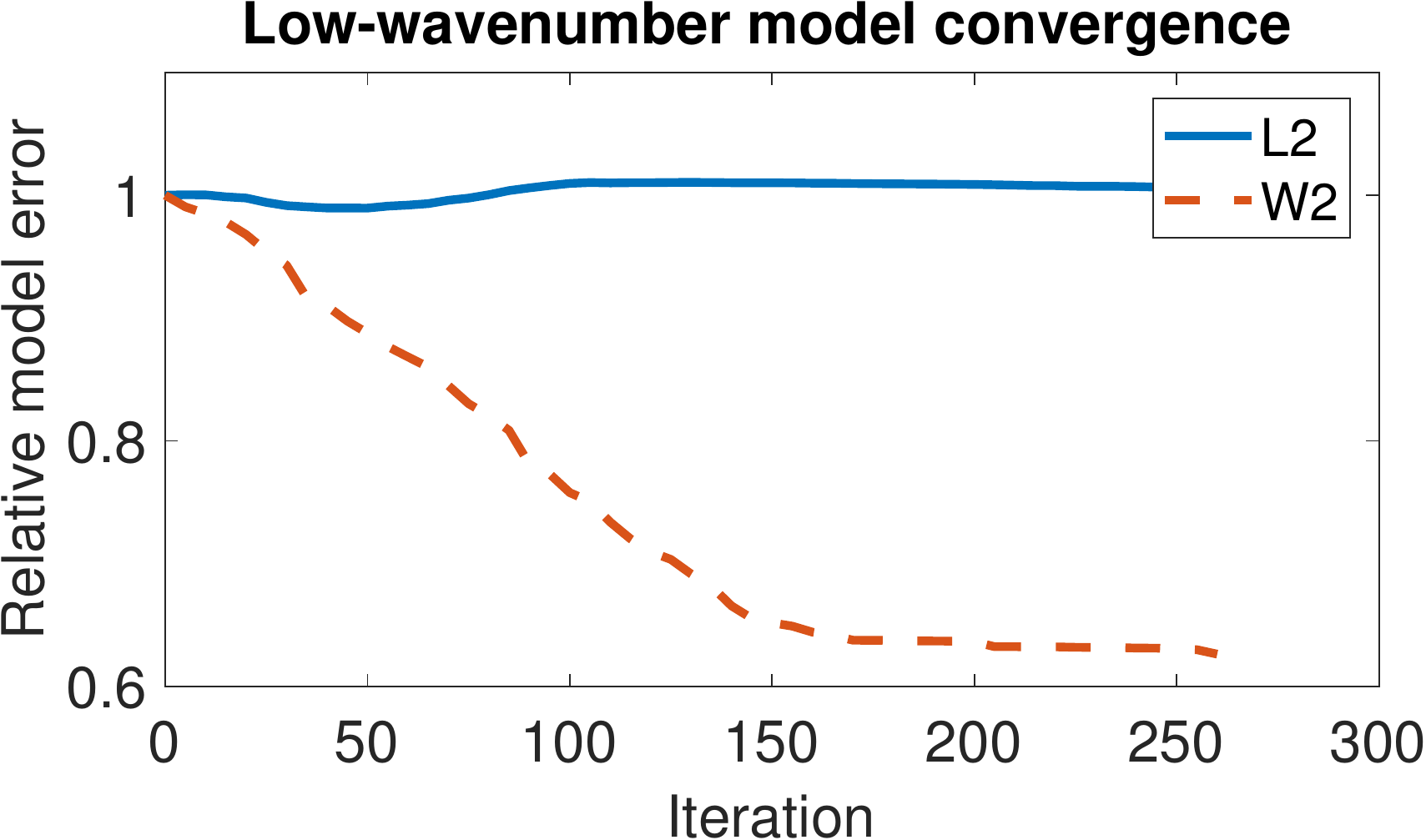}\label{fig:low_model_conv}}\qquad 
  \subfloat[$\bigl\|m^\text{high}_\text{iter}-m^\text{high}_*\bigr\|_F$]{\includegraphics[width=0.45\textwidth]{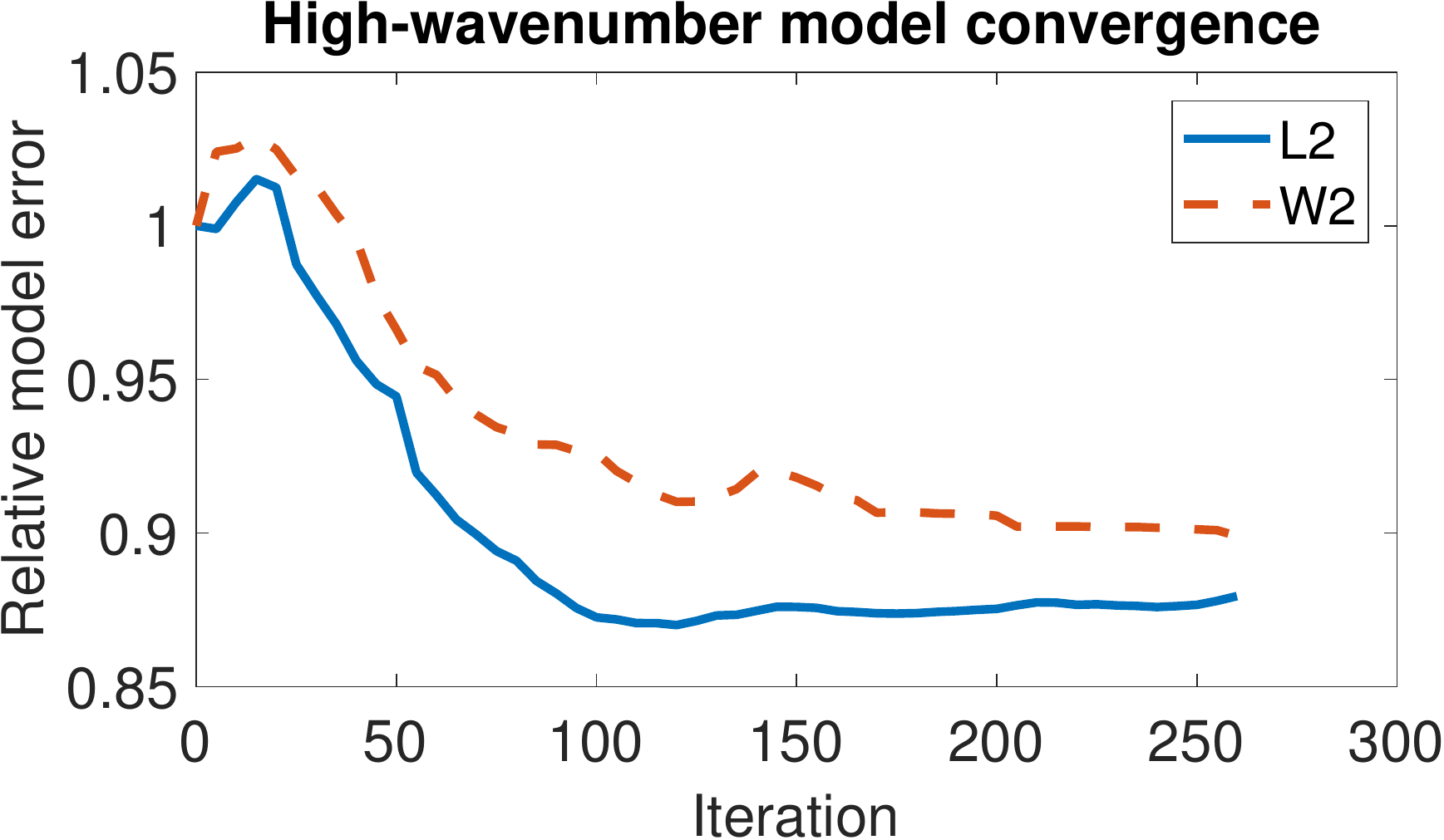}\label{fig:high_model_conv}}\\
  \subfloat[$\bigl\|f^\text{low}_\text{iter}-g^\text{low}\bigr\|_2$]{\includegraphics[width=0.45\textwidth]{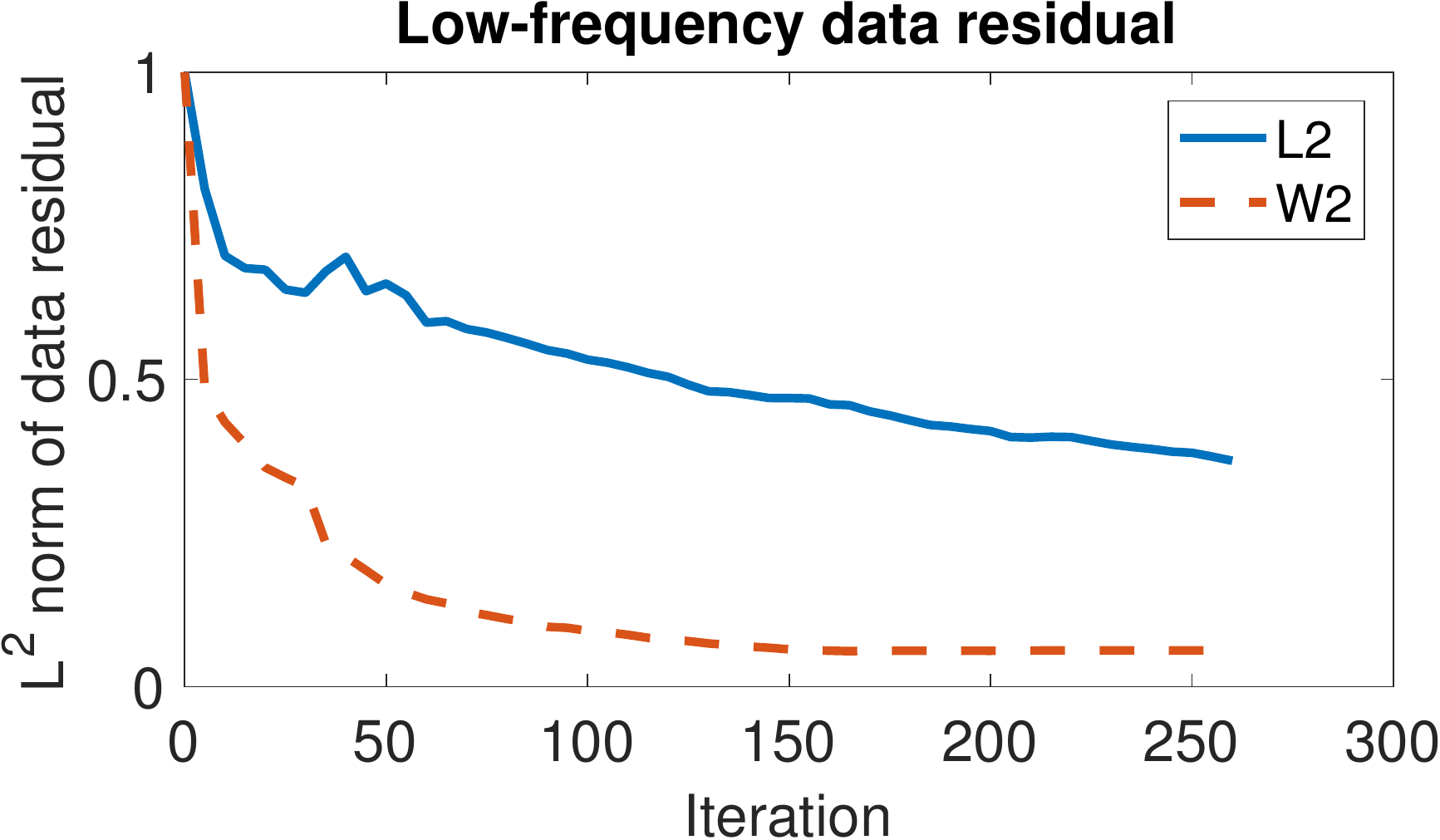}\label{fig:low_data_res}}\qquad 
  \subfloat[$\bigl\|f^\text{high}_\text{iter}-g^\text{high}\bigr\|_2$]{\includegraphics[width=0.45\textwidth]{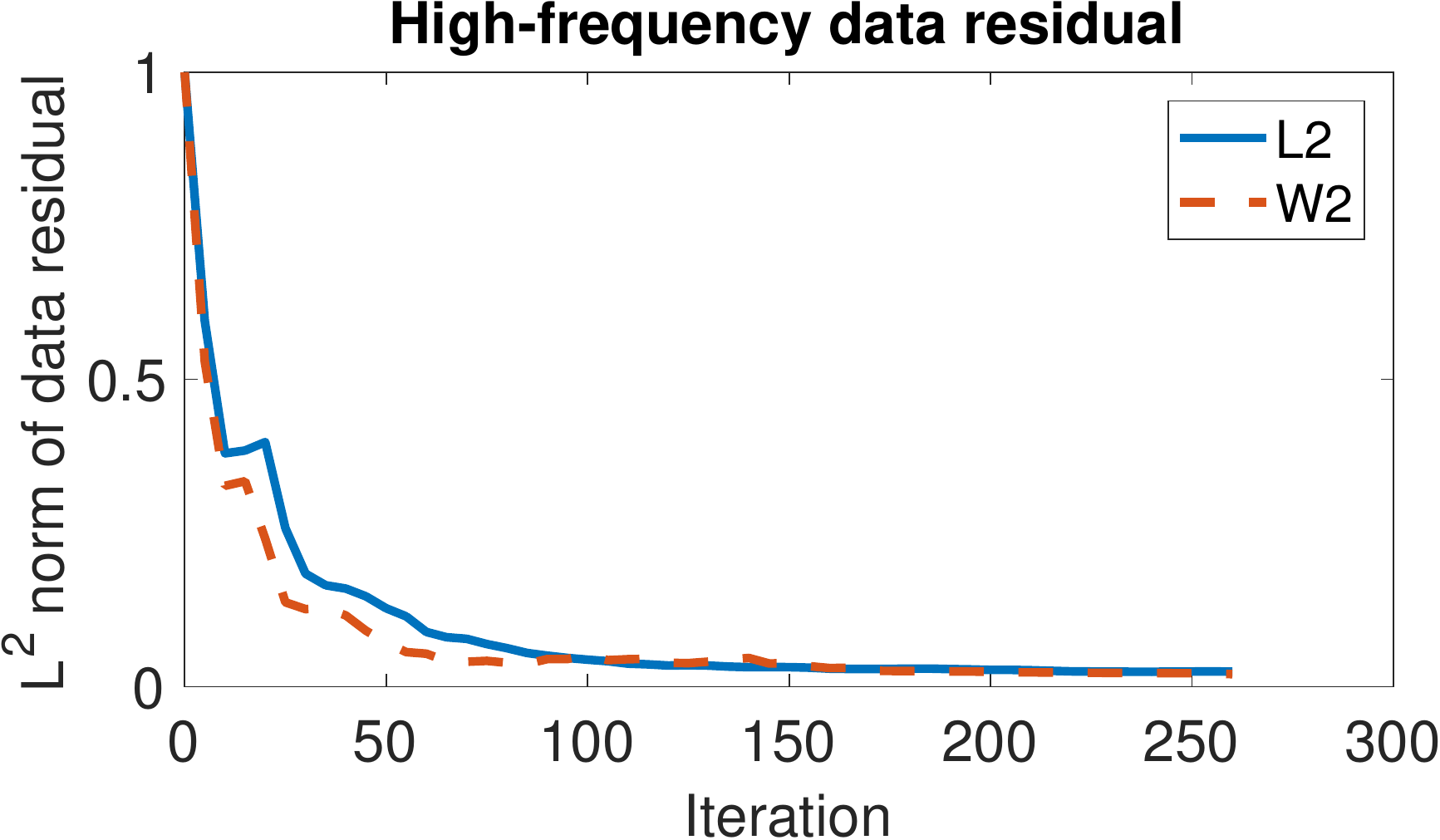}\label{fig:high_data_res}}
  \caption{(a) Low-wavenumber and (b) high-wavenumber model
    convergence; (c) low-frequency and (d) high-frequency $2$-norm
    data residual for $L^2$-based and $W_2$-based
    inversion.}~\label{fig:residual}
\end{figure}

\vfill
\begin{figure}
  \centering \subfloat[$L^2$-FWI at iteration
  0]{\includegraphics[width=0.3\textwidth]{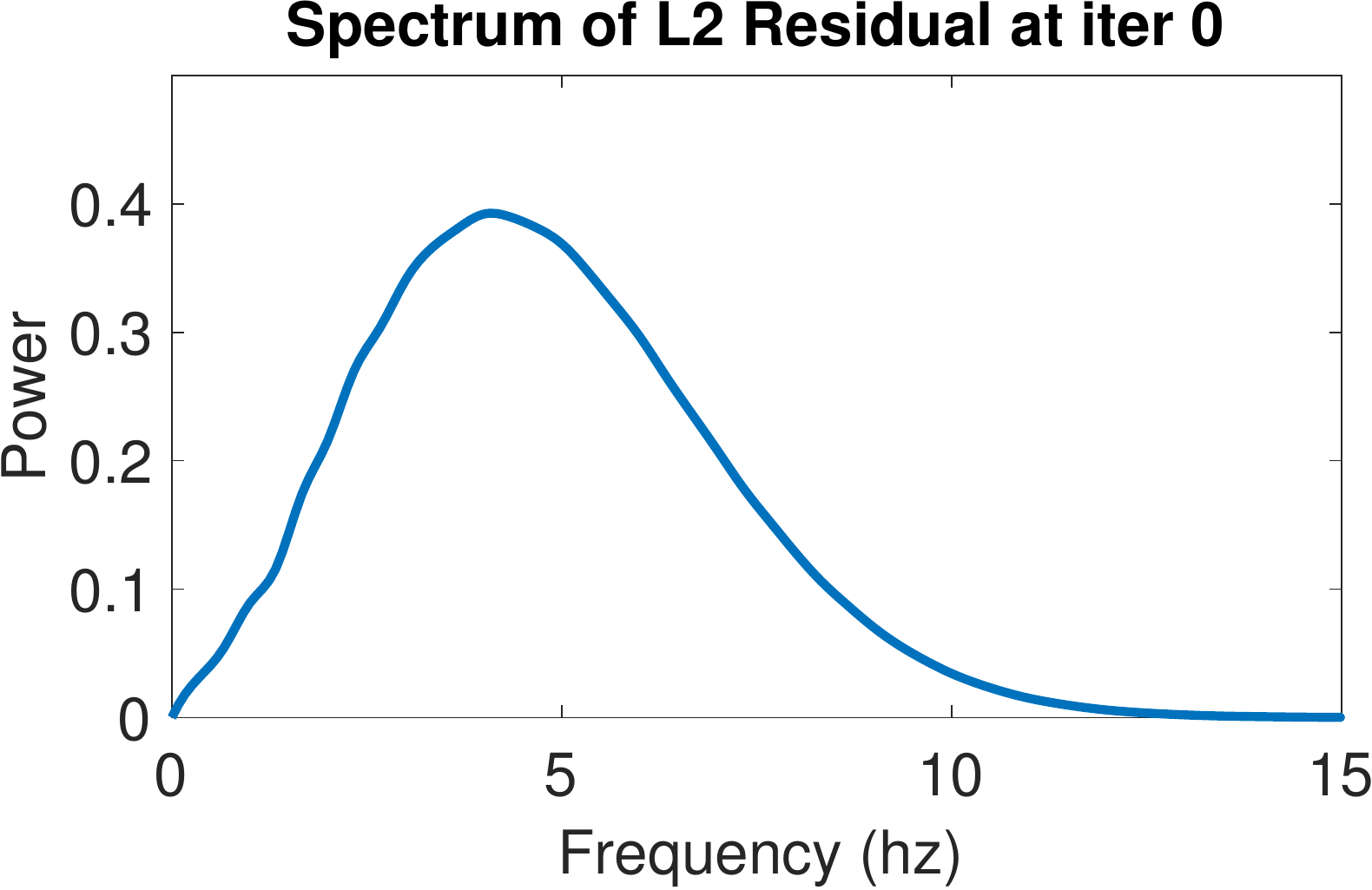}\label{fig:refl_L2_iter1_fft}}\quad 
  \subfloat[$L^2$-FWI at iteration
  100]{\includegraphics[width=0.3\textwidth]{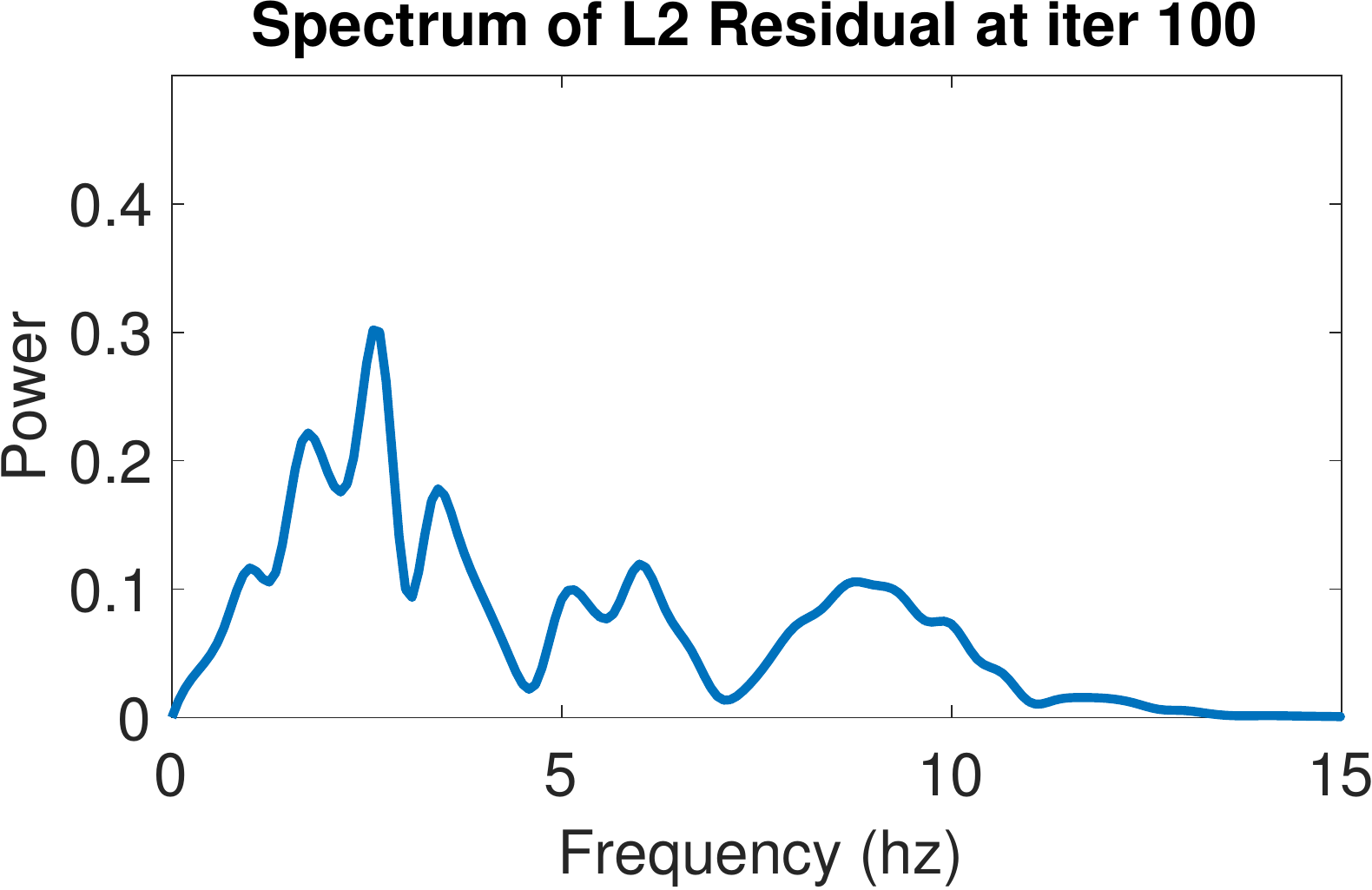}\label{fig:refl_L2_iter51_fft}}\quad 
  \subfloat[$L^2$-FWI at iteration 200]{\includegraphics[width=0.3\textwidth]{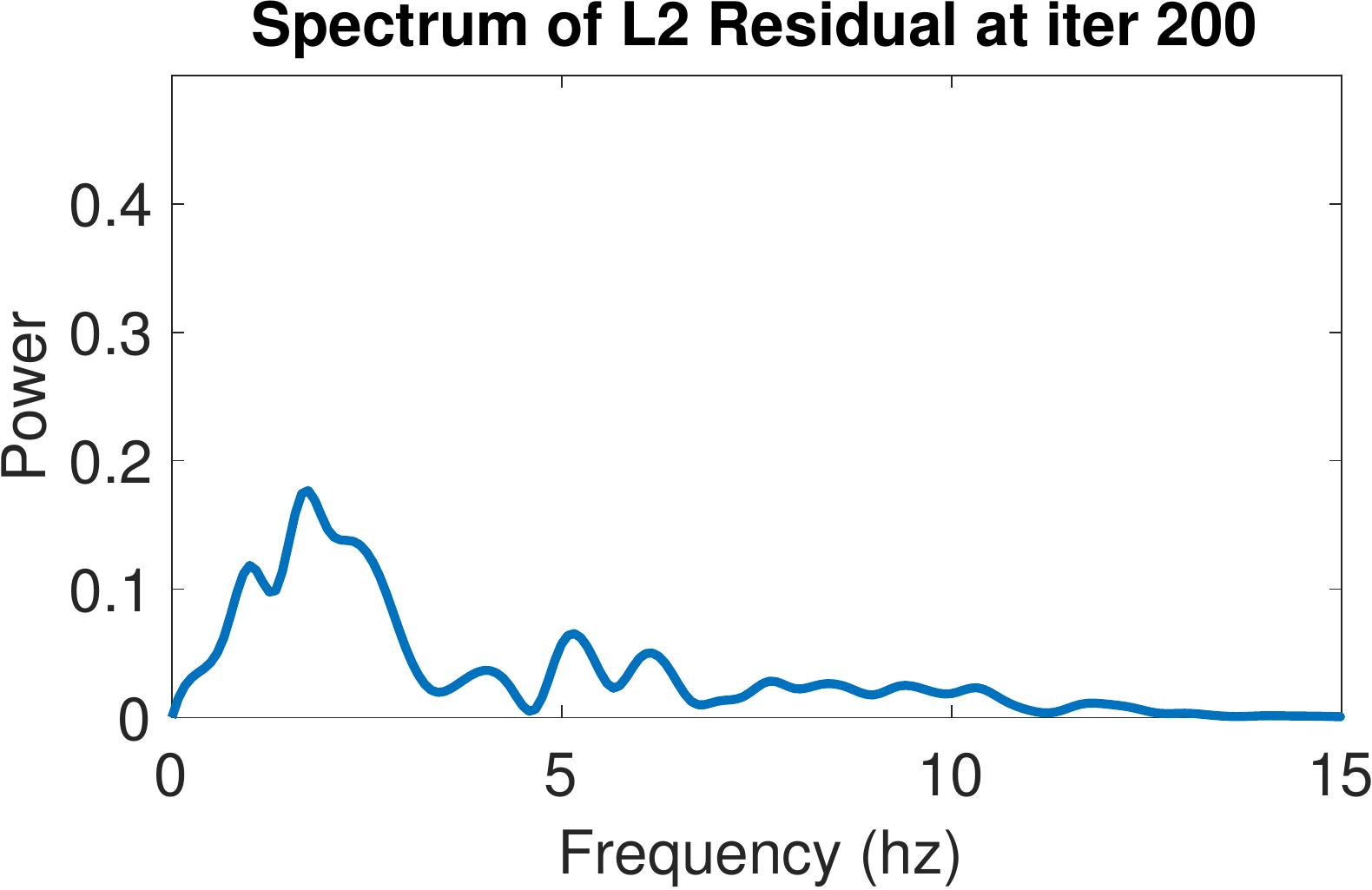}\label{fig:refl_L2_iter101_fft}}\\
  \subfloat[$W_2$-FWI at iteration
  0]{\includegraphics[width=0.3\textwidth]{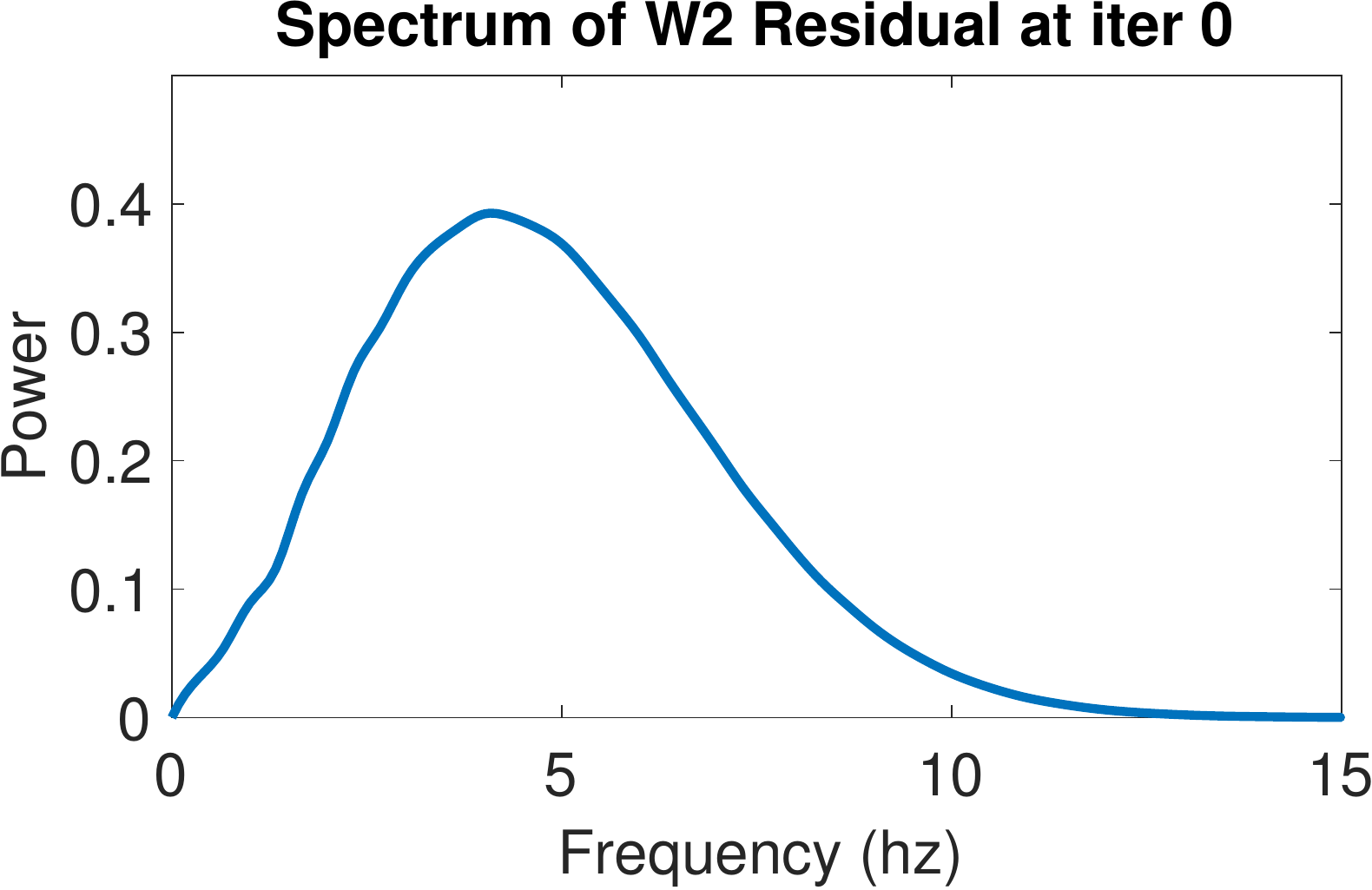}\label{fig:refl_W2_iter1_fft}}\quad 
  \subfloat[$W_2$-FWI at iteration
  100]{\includegraphics[width=0.3\textwidth]{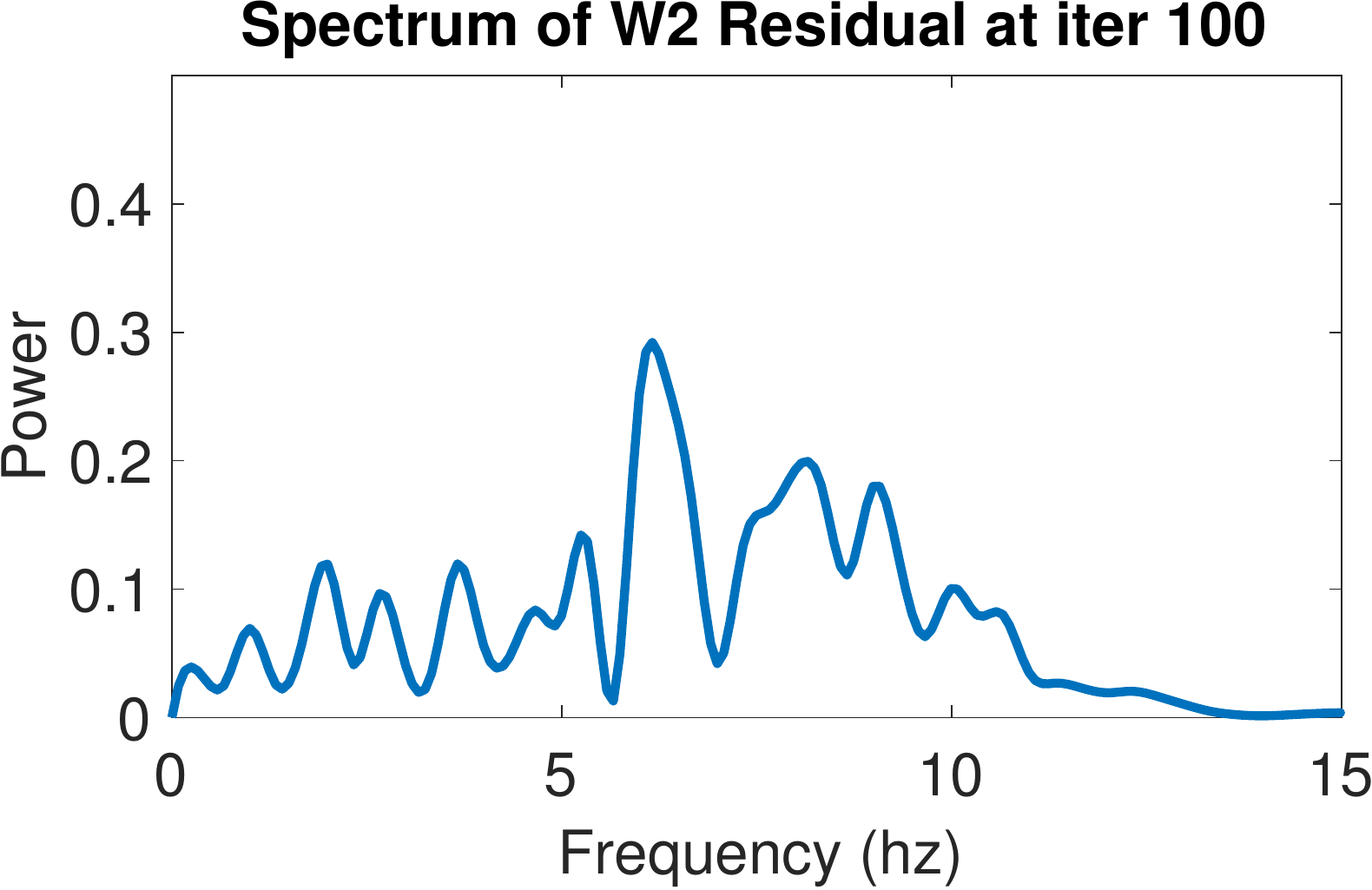}\label{fig:refl_W2_iter51_fft}}\quad 
  \subfloat[$W_2$-FWI at iteration
  200]{\includegraphics[width=0.3\textwidth]{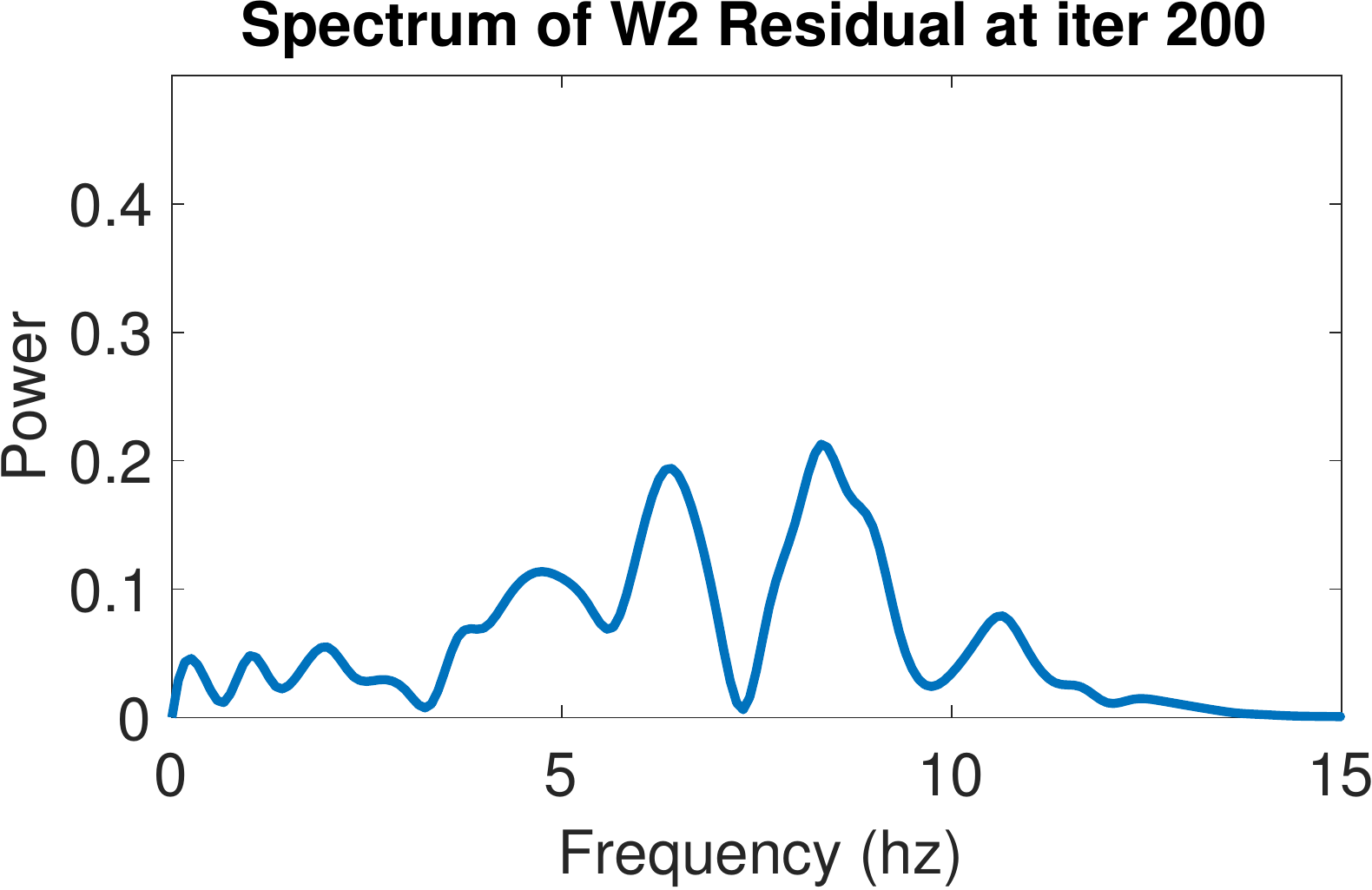}\label{fig:refl_W2_iter101_fft}}
  \caption{Three-layer model: the residual spectrum for $L^2$-based
    and $W_2$-based inversion at iteration 0, 100, and
    200.}~\label{fig:refl_res_fft}
\end{figure}

In Figure~\ref{fig:CPAM_model_conv}, we observe that the $2$-norm of
the model error in $L^2$-based inversion barely changes. We decompose
all velocity models in the experiment into low-wavenumber (with
Fourier modes $|\mathbf{k}|\leq30$) and high-wavenumber (with Fourier
modes $|\mathbf{k}|>30$ ) parts by bandpass filters. We are interested
in the model error decay of each part in the inversion.  Similarly, we
divide the data residual of both inversions into the low- and
high-frequency parts. The residual is computed as the difference
between synthetic data $f(\mathbf{x},t;m)$ generated by the model $m$
at current iteration and the true data $g(\mathbf{x},t)$:
\[
  \text{residual} = f(\mathbf{x},t;m) - g(\mathbf{x},t).
\]

The differences between the two objective functions are more clear in
Figure~\ref{fig:residual}, which illustrates different convergence
patterns for the smooth and oscillatory parts of both the model and
the data. The low-wavenumber model error
(Figure~\ref{fig:low_model_conv}) and the low-frequency data residual
(Figure~\ref{fig:low_data_res}) of the $W_2$-based inversion decreases
much more rapidly than the $L^2$-based inversion, while the latter
shows sensitivity in reducing the high-wavenumber model error and
high-frequency residuals, based on Figure~\ref{fig:high_model_conv}
and Figure~\ref{fig:high_data_res}.

Another way of analyzing the error reduction with respect to different
objective functions is to look at their data residual in the Fourier
domain. Figure~\ref{fig:refl_res_fft} consists of six plots of the
residual spectrum of two inversion schemes at three different
iterations. By comparing the change of spectrum as the iteration
number increases, one can observe that the inversion driven by the $L^2$
norm has a different pattern in changing the residual spectrum with
the $W_2$-based inversion. It focuses on reducing the high-frequency
residual in early iterations and slowly decreasing the low-frequency
residual later. On the other hand, inversion using the $W_2$ metric
reduces the smooth parts of the residual first
(Figure~\ref{fig:refl_W2_iter51_fft}) and then gradually switches to
the oscillatory parts (Figure~\ref{fig:refl_W2_iter101_fft}).

The fact that $W_2$ is more robust than $L^2$ in reconstructing
low-wavenumber components while $L^2$-FWI converges faster and
achieves higher resolution for the high-wavenumber features was
already observed in~\cite{yang2017application}, in an inversion
example with difficulties of local minima. For the two-layer example
discussed above, the $L^2$-based inversion does not suffer from local
minima trapping, but the properties for these two inversion schemes
still hold. Rigorous analysis has been done in~\cite{ERY2019}, where
the $L^2$ norm and the $W_2$ metric were discussed under both the
asymptotic and nonasymptotic regimes for data comparison. Theorems
have demonstrated that $L^2$ norm is good at achieving high resolution
in the reconstruction, while the $W_2$ metric gives better stability
with respect to data perturbation. Using $W_2$-based inversion to
build a good starting model for the $L^2$-based inversion in a later
stage is one way to combine useful features of both methods.

\section{Numerical Examples}\label{sec:6}
In this section, we demonstrate several numerical examples under both
synthetic and somewhat more realistic settings. We will see
applications of the $W_2$ convexity analysis in Section~\ref{sec:cvx},
and, more importantly, other improvements in
Section~\ref{sec:three-layer} that are beyond local-minima
trapping. The $L^2$ norm and the $W_2$ metric will be used as the
objective function, and their inversion results will be compared. In
particular, we stick with the so-called trace-by-trace
approach~\eqref{eqn:Wp1D} for inversions using the $W_2$ metric.
\begin{figure}
  \centering \subfloat[True
  velocity]{\includegraphics[width=0.4\textwidth]{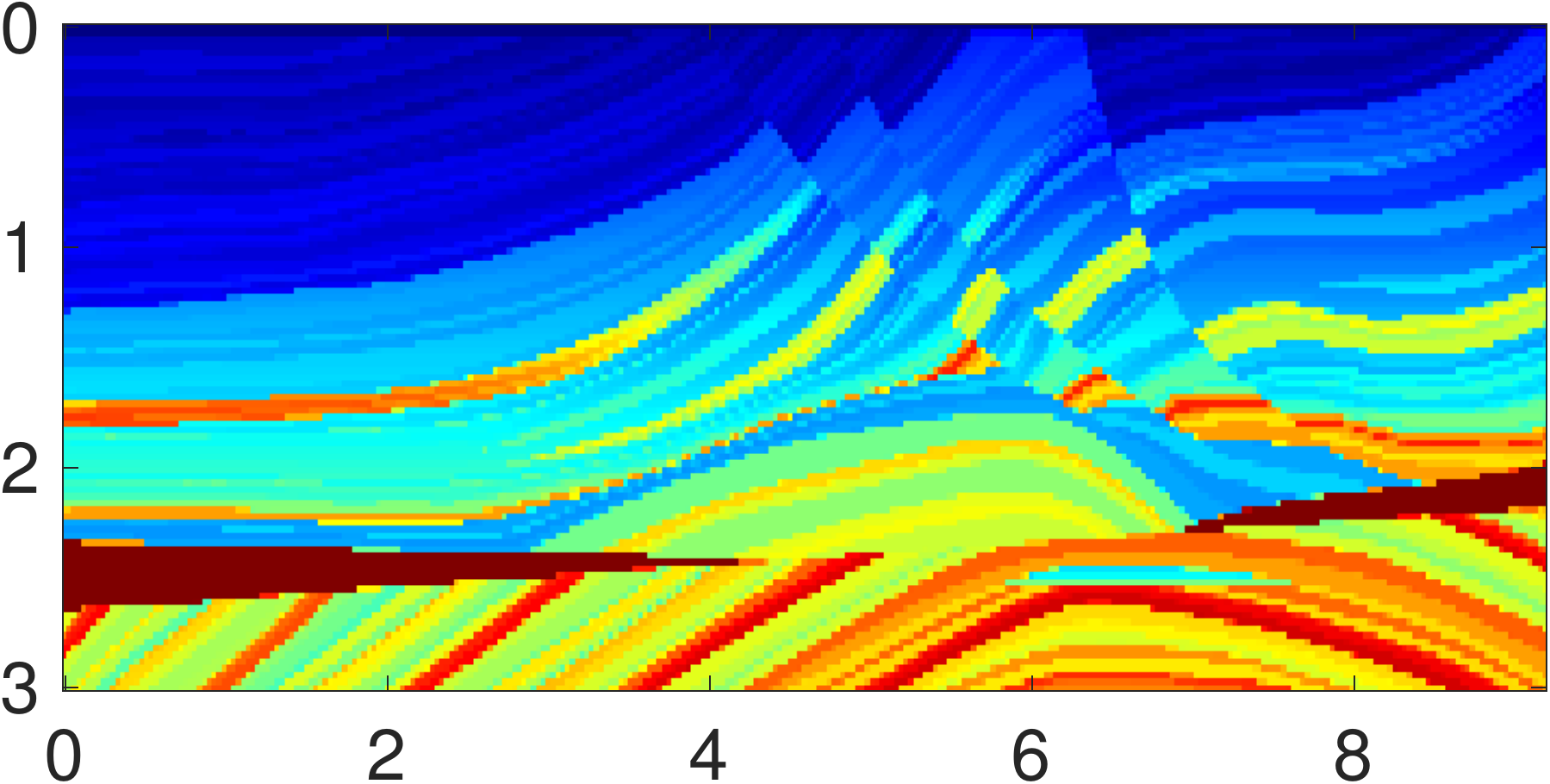}\label{fig:Marm-vG}}
  \hspace{0.08\textwidth}
  \subfloat[Initial velocity]{\includegraphics[width=0.4\textwidth]{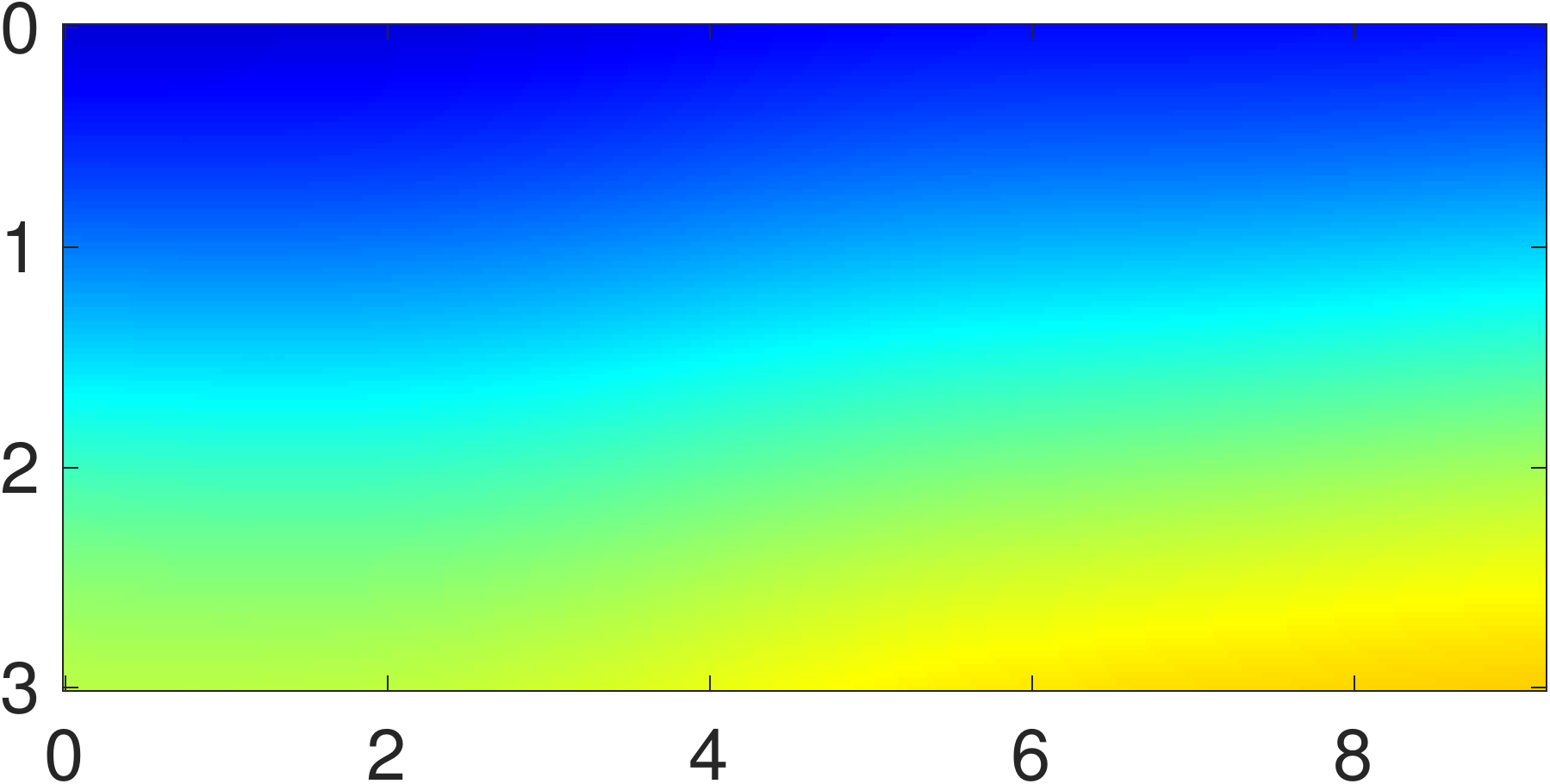}\label{fig:Marm-vF}}  \\
  \subfloat[$L^2$
  inversion]{\includegraphics[width=0.4\textwidth]{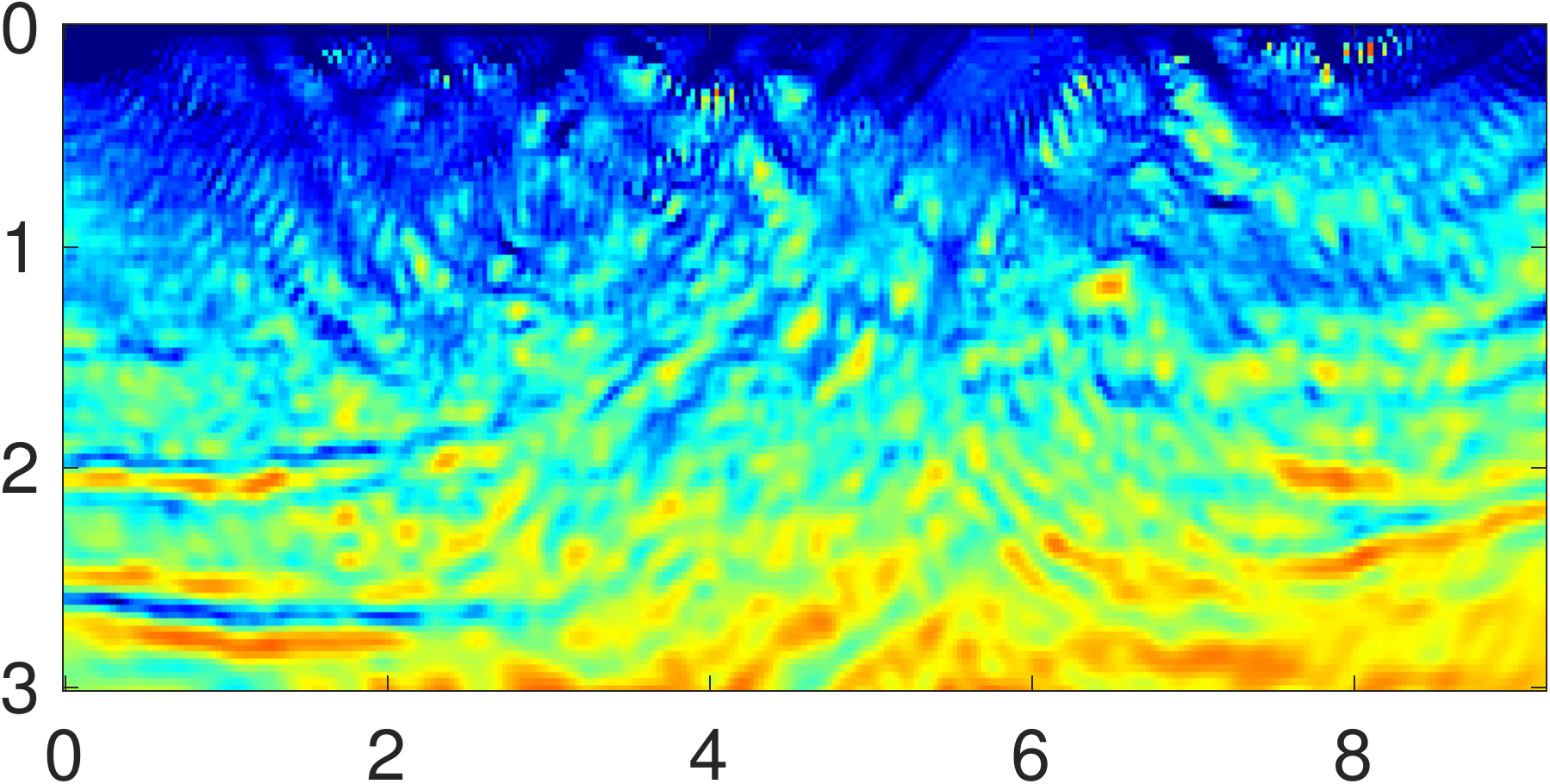}\label{fig:Marm-L2}}
  \hspace{0.08\textwidth}
  \subfloat[$W_2$ inversion]{\includegraphics[width=0.4\textwidth]{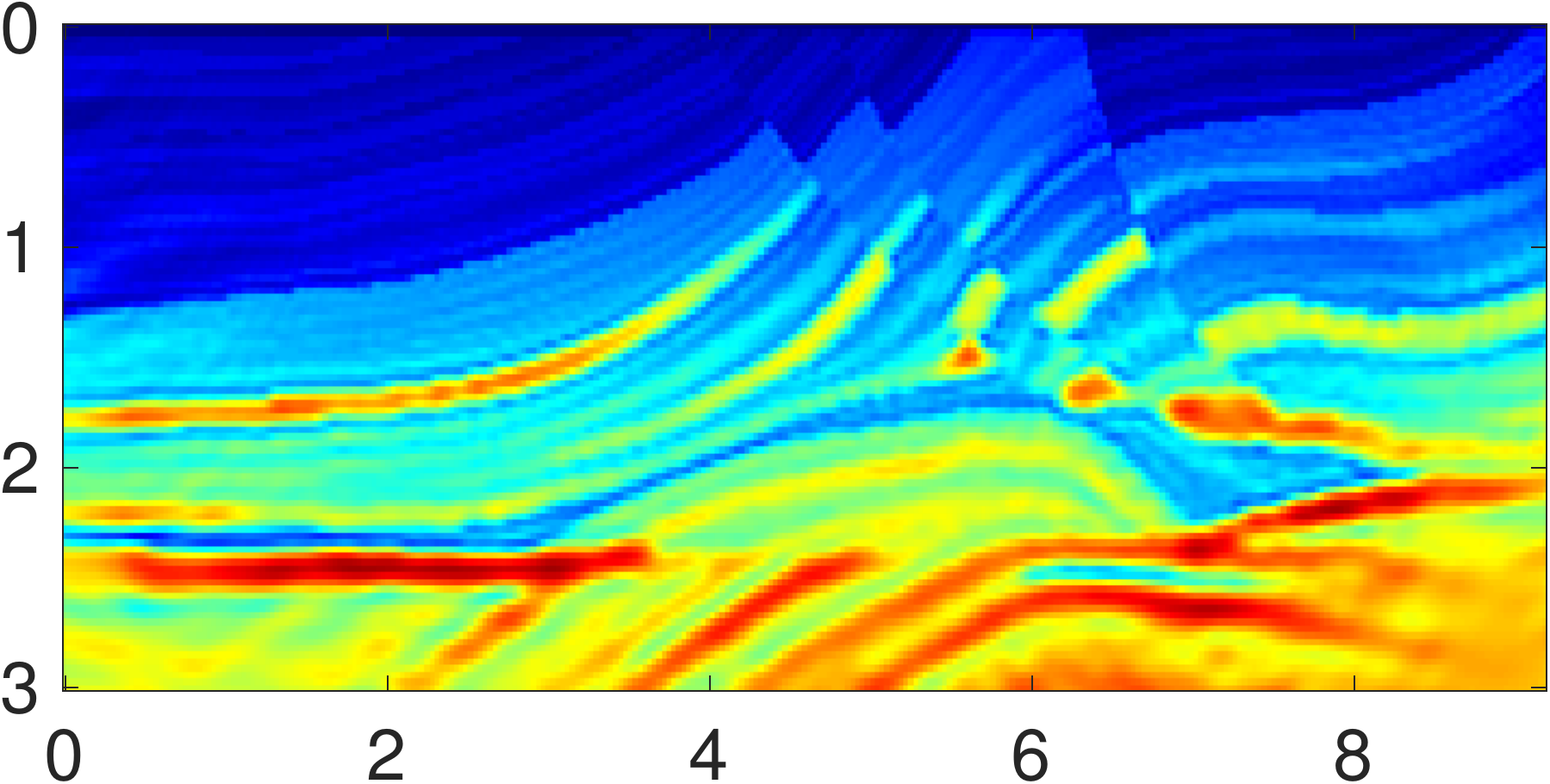}\label{fig:Marm-W2}}  \\
  \subfloat[$L^2$ inversion (with
  noise)]{\includegraphics[width=0.4\textwidth]{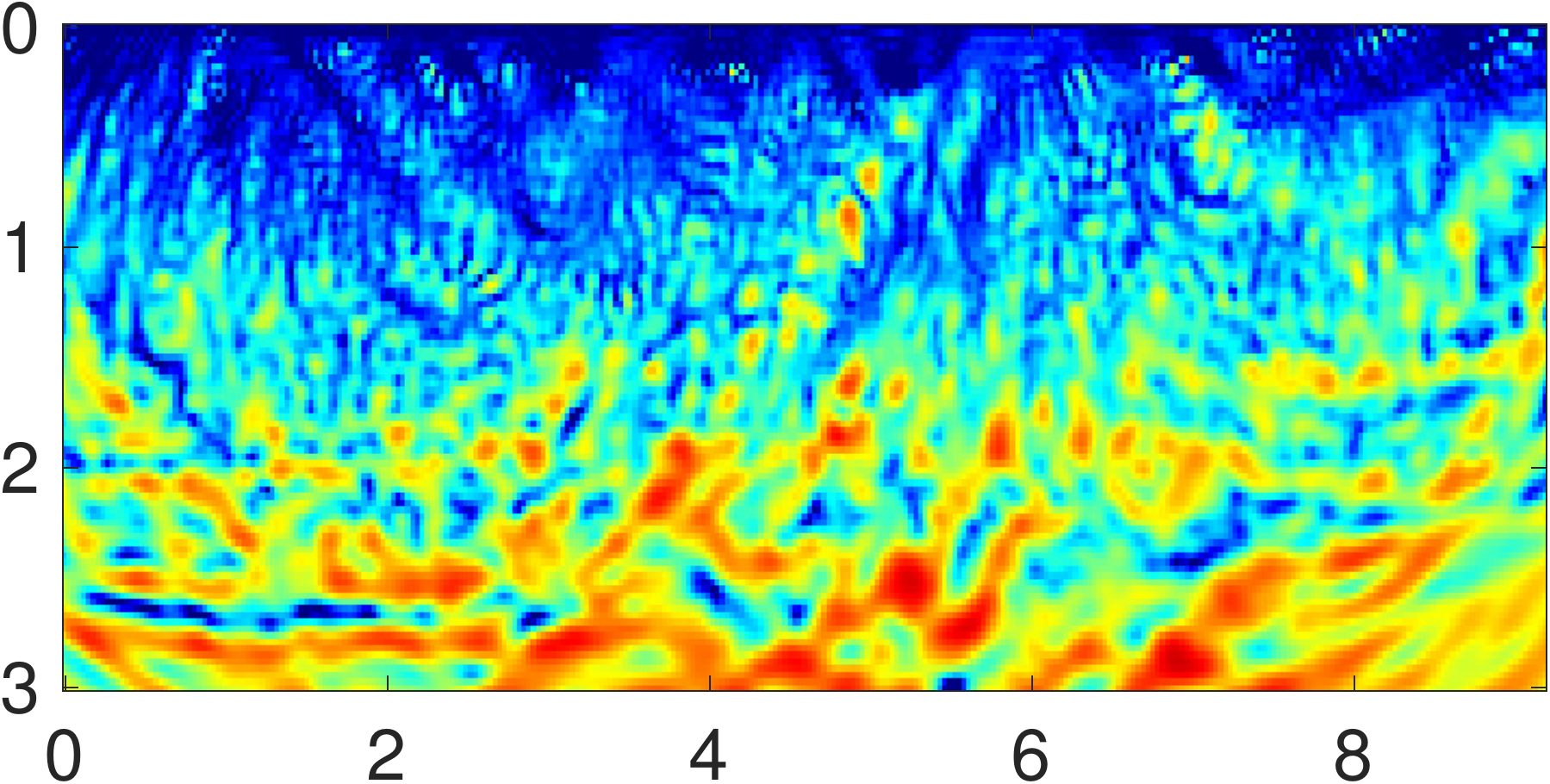}\label{fig:Marm-L2-noise}}
  \hspace{0.08\textwidth}
  \subfloat[$W_2$ inversion (with noise)]{\includegraphics[width=0.4\textwidth]{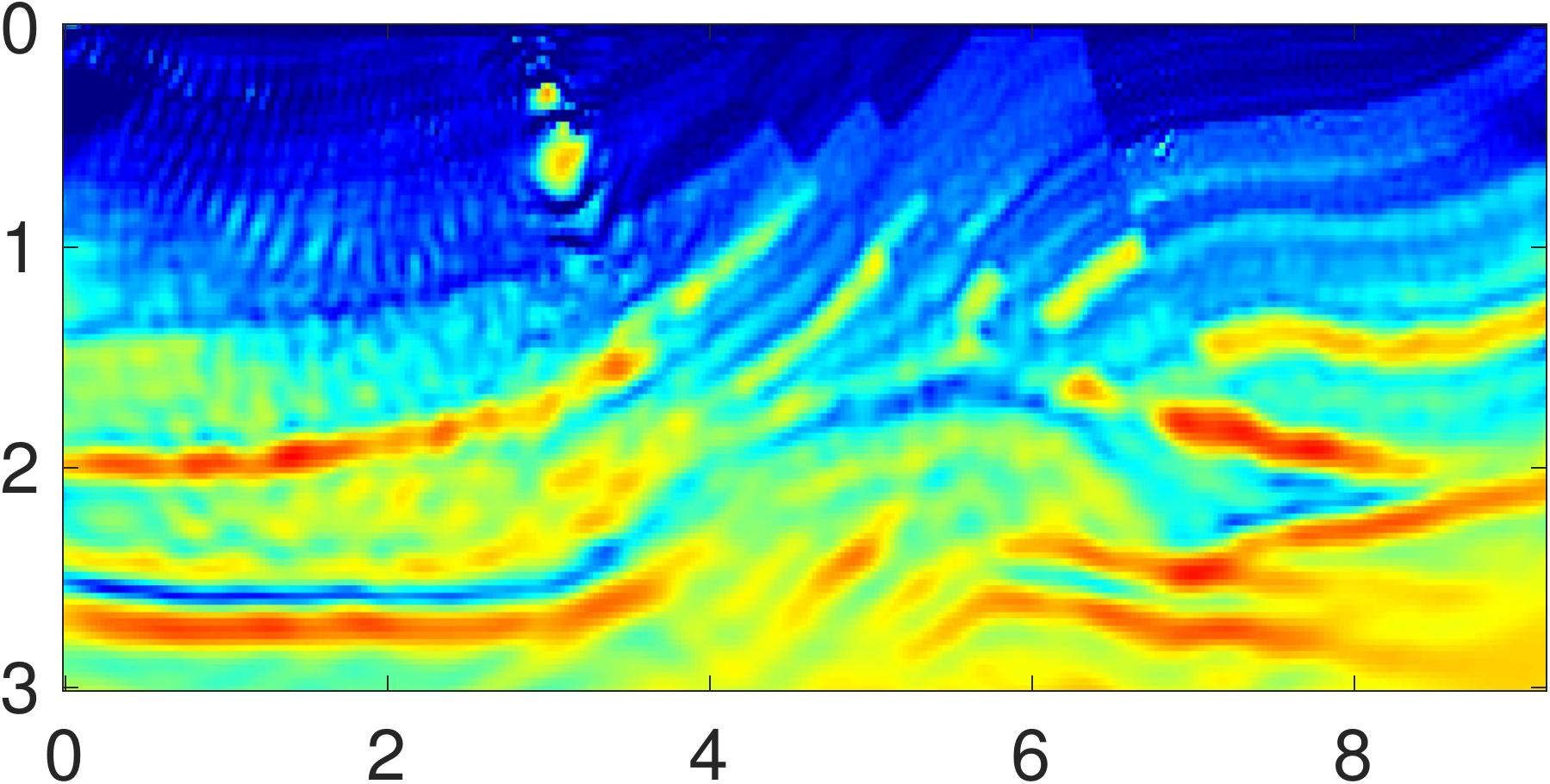}\label{fig:Marm-W2-noise}}  \\
  \caption{The Marmousi model inversion:~(a)~the true~and~(b)~initial
    velocities;~(c) and (d): the $L^2$ and $W_2$ inversion results for the
    synthetic setting;~(e) and (f): the $L^2$ and $W_2$ inversion results
    for the realistic setting. We use the linear
    scaling~\eqref{eq:linear} as the data normalization to calculate
    the $W_2$ metric. All axes are associated with the unit
    km.}~\label{fig:Marm}
\end{figure}

\subsection{The Marmousi Model}\label{sec:marm}
The true velocity of the Marmousi model is presented in
Figure~\ref{fig:Marm-vG}, which was created to produce complex seismic
data that require advanced processing techniques to obtain a correct
Earth image. It has become a standard benchmark for methods and
algorithms for seismic imaging since
1988~\cite{versteeg1994marmousi}. The initial model shown in
Figure~\ref{fig:Marm-vF} lacks all the layered features. We will
invert the Marmousi model numerically using the $L^2$ norm and the
$W_2$ metric under one synthetic setting and a more realistic setting
in terms of the observed true data.

\subsubsection{Synthetic Setting}
Under the same synthetic setting, the true data and the synthetic data
are generated by the same numerical scheme. The wave source is a 10 Hz
Ricker wavelet. A major challenge for the $L^2$-based inversion is the
phase mismatches in the data generated by the true and initial
velocities. After 300 L-BFGS iterations, the $L^2$-based inversion
converges to a local minimum, as shown in Figure~\ref{fig:Marm-L2},
which is also a sign of cycle skipping due to the lack of convexity
with respect to data translation. The $W_2$-based inversion, on the
other hand, recovers the velocity model correctly, as seen in
Figure~\ref{fig:Marm-W2}. As discussed in Section~\ref{sec:cvx}, the
$W_2$ metric has the important convexity with respect to data
translation and dilation. Hence, the initial model is within the basin
of attraction for the $W_2$-based inversion. Once the background
velocity is correctly recovered, the missing high-wavenumber features
can also be reconstructed correctly.

\subsubsection{A More Realistic Setting} \label{sec:MarmReal} For
field data inversions, source approximation, the elastic effects,
anisotropy, attenuation, noisy measurement, and many other factors
could bring modeling errors to the forward propagation. To discuss the
robustness of the $W_2$-based method with respect to the accuracy of
the source estimation and the noise in the measurements, we present
another test with more challenging settings.


\begin{figure}
  \centering \subfloat[Noise in the
  data]{\includegraphics[width=0.48\textwidth]{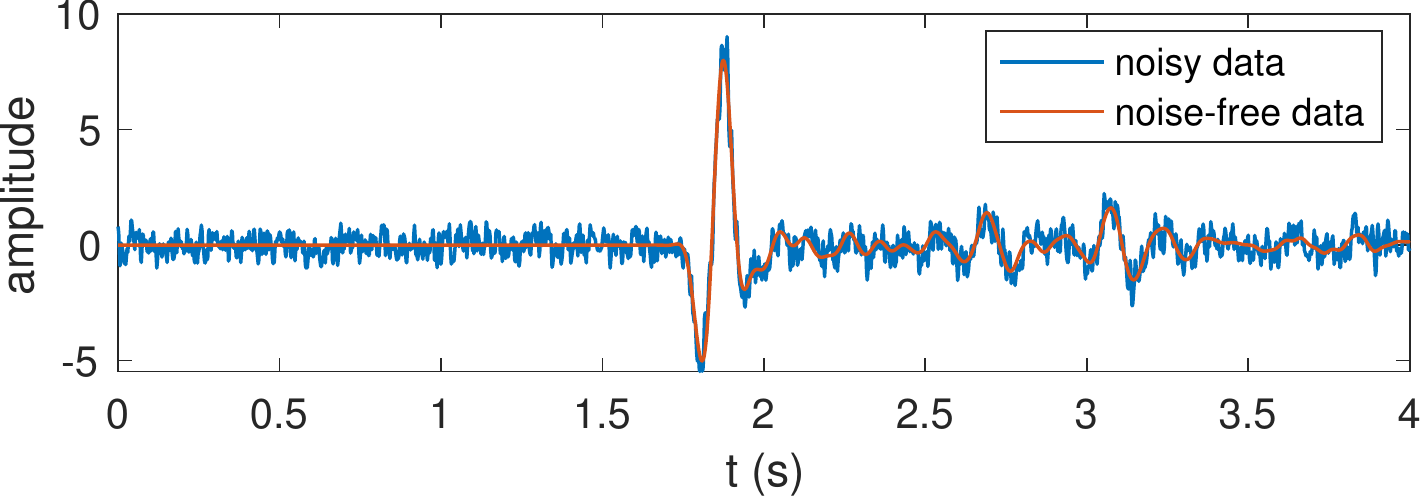}\label{fig:Marm-data}}\quad  
  \subfloat[Estimation of the source
  wavelet]{\includegraphics[width=0.48\textwidth]{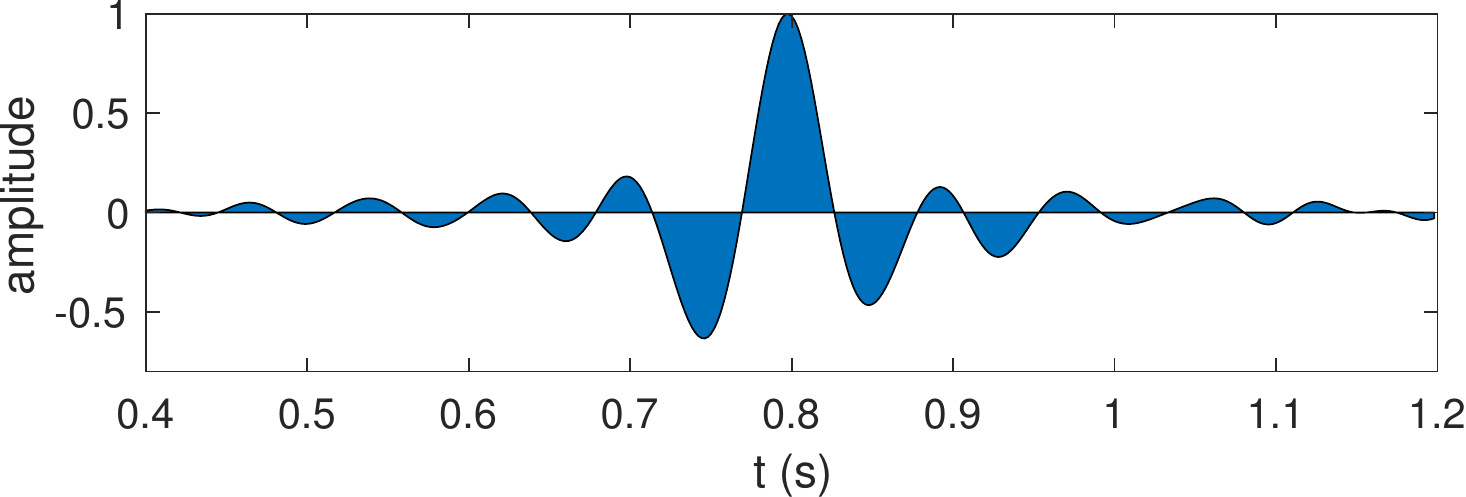}\label{fig:Marm-new-src}}
  \caption{The realistic setting for the Marmousi model:~(a)~noisy and
    clean observed datasets comparison at one trace; (b)~estimation of
    the source wavelet by the linear waveform
    inversion.}~\label{fig:noise}
\end{figure}

The observed data is generated by a Ricker wavelet centered at 10 Hz
as in the previous test, which is additionally polluted by mean-zero
correlated random noise; see an illustration of the noisy data in
Figure~\ref{fig:Marm-data}. We estimate the source profile by linear
least-squares waveform inversion of the direct
wave~\cite{pratt1990inverse1,ribodetti2011joint}. A homogeneous medium
of $1.5$ km/s is used as the velocity model in the linear
inversion. The reconstructed source wavelet is shown in
Figure~\ref{fig:Marm-new-src}.  Inversion results are illustrated in
Figure~\ref{fig:Marm-L2-noise} and Figure~\ref{fig:Marm-W2-noise}. The
presence of noise and the inaccurate wave source deteriorate the
$L^2$ result, but only mildly change the $W_2$-based inversion. Most
structures are recovered with relatively lower
resolution. Inaccuracies are present for the shallow
part. Nevertheless, in comparison, the $W_2$ metric is more robust
than the $L^2$ norm for small perturbations in the data that come from
modeling error or measurement noise.

\subsubsection{The Scaling Method}
In the $W_2$-based Marmousi model inversions, we transform the
wavefields into probability densities by the linear
scaling~\eqref{eq:linear}. The constant $c$ is chosen to be the
$\ell^\infty$ norm of the observed data. Although the normalized $W_2$
metric lacks strict convexity in terms of data translations as we 
pointed out in~\cite{Survey2}, the linear scaling still works
remarkably well in practice for most cases, including the realistic
inversions in the industry. Also, normalization methods~\eqref{eq:exp}
and~\eqref{eq:softplus} give similar results and are therefore not
included. Nevertheless, data normalization is an important step for
$W_2$-based full-waveform inversion. More analysis and discussions are
presented in Section~\ref{sec:Data_Normalization}.

It is observed here and in many other works on optimal-transport-based
FWI that a slight improvement in the convexity of the misfit function
seems to produce significant effects on the inversion in mitigating
cycle-skipping issues~\cite{yang2017application,W1_3D}. It is also the
case for the linear scaling~\eqref{eq:linear}. The linear scaling does
not preserve the convexity of the $W_2$ metric with respect to
translation, but as a sign of improvement, the basin of attraction is
larger than the one for the $L^2$ norm, as shown previously in
Figure~\ref{fig:2_ricker_L2} and~Figure~\ref{fig:W2-linear}.

We can
point to three reasons that may contribute to the success of the
linear scaling in practice. First, as discussed in
Section~\ref{sec:Data_Normalization}, adding a constant to the signals
before being compared by the $W_2$ metric brings the Huber effect for
better robustness with respect to outliers. Second, it smooths the FWI
gradient and thus emphasizes the low-wavenumber components of the
model parameter. Third, the convexity with respect to signal
translation only covers one aspect of the challenges in realistic
inversions. The synthetic data is rarely a perfect translation of the
observed data in practice. The nonconvexity caused by the linear
scaling might not be an issue in realistic settings, while the
outstanding benefits of adding a positive constant dominate.

\subsection{The Circular Inclusion Model}\label{sec:cheese}
We have presented the Marmousi model to demonstrate that $W_2$-based
inversion is superior to $L^2$. We also show that the linear
scaling~\eqref{eq:linear} is often good enough as a normalization
method. However, to address the importance of data normalization in
applying optimal transport, we create a synthetic example in which the
linear scaling~\eqref{eq:linear} affects the global convexity of
$W_2$.

We want to demonstrate the issues above with a circular inclusion
model in a homogeneous medium (with $6$ km in width and $4.8$ km in
depth). It is also referred to as the ``Camembert''
model~\cite{gauthier1986two}. The true velocity is shown in
Figure~\ref{fig:cheese-vG}, where the anomaly in the middle has wave
speed $4.6$ km/s while the rest is $4$ km/s. The initial velocity we
use in the inversion is a homogeneous model of $4$ km/s. We have 13
sources of 10 Hz Ricker wavelet equally aligned on the top of the
domain, while 201 receivers are on the bottom. The recorded signals
contain mainly transmissions, which are also illustrated in
Figure~\ref{fig:ck} as an example of the cycle-skipping
issues. Figure~\ref{fig:ckdiff} shows the initial data fit, which is
the difference between the observed data and the synthetic data
generated by the initial velocity model. As seen in
Figure~\ref{fig:cheese-l2}, the inversion with the $L^2$ norm as the
objective function suffers from local minima trapping.


\begin{figure}
  \centering \subfloat[True
  velocity]{\includegraphics[width=0.33\textwidth]{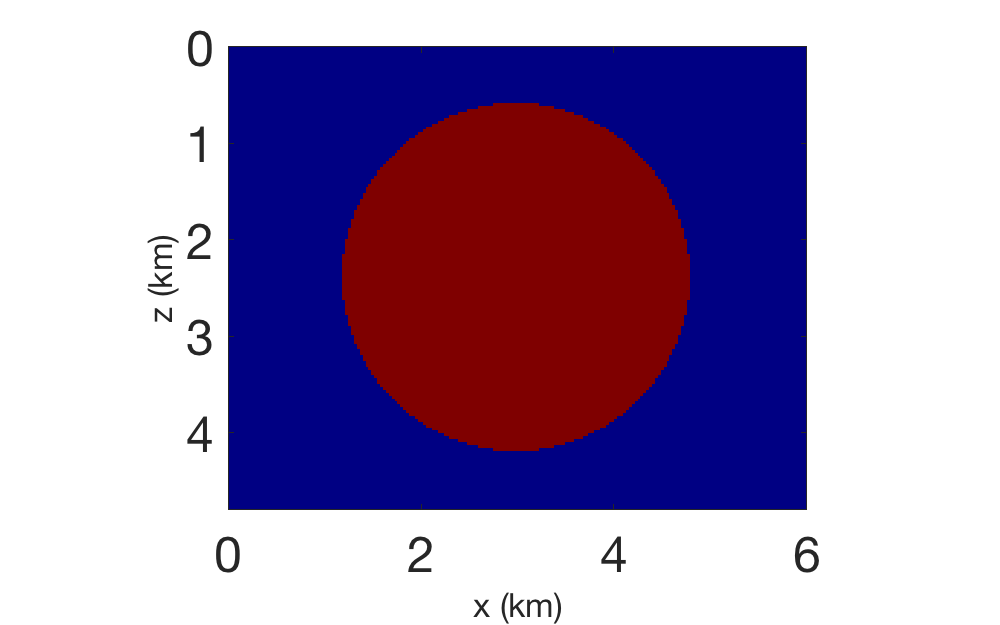}\label{fig:cheese-vG}}
  \subfloat[$L^2$
  inversion]{\includegraphics[width=0.33\textwidth]{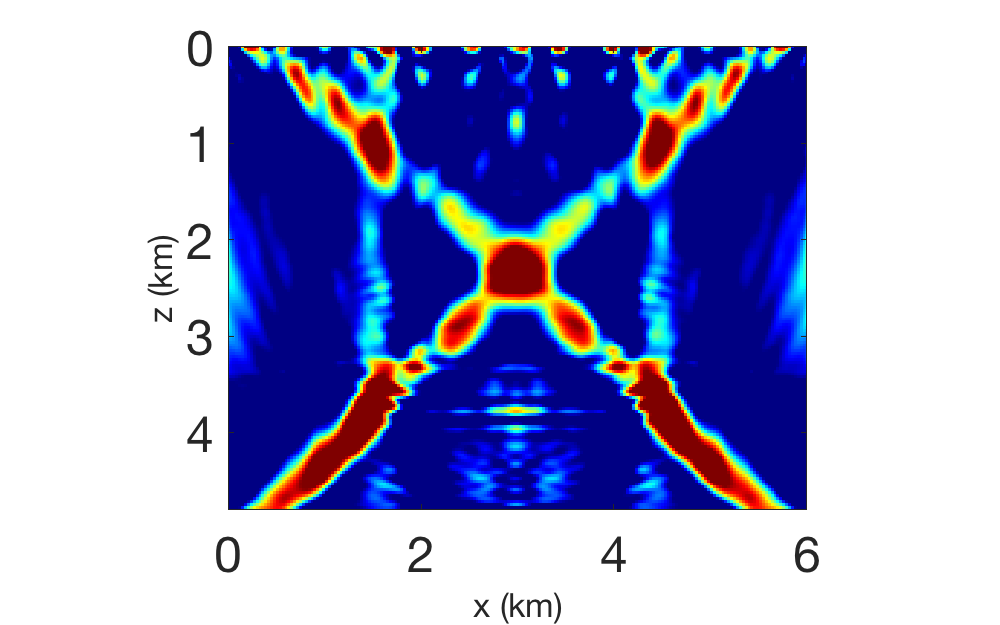}\label{fig:cheese-l2}}
  \subfloat[$W_2$, square]{\includegraphics[width=0.33\textwidth]{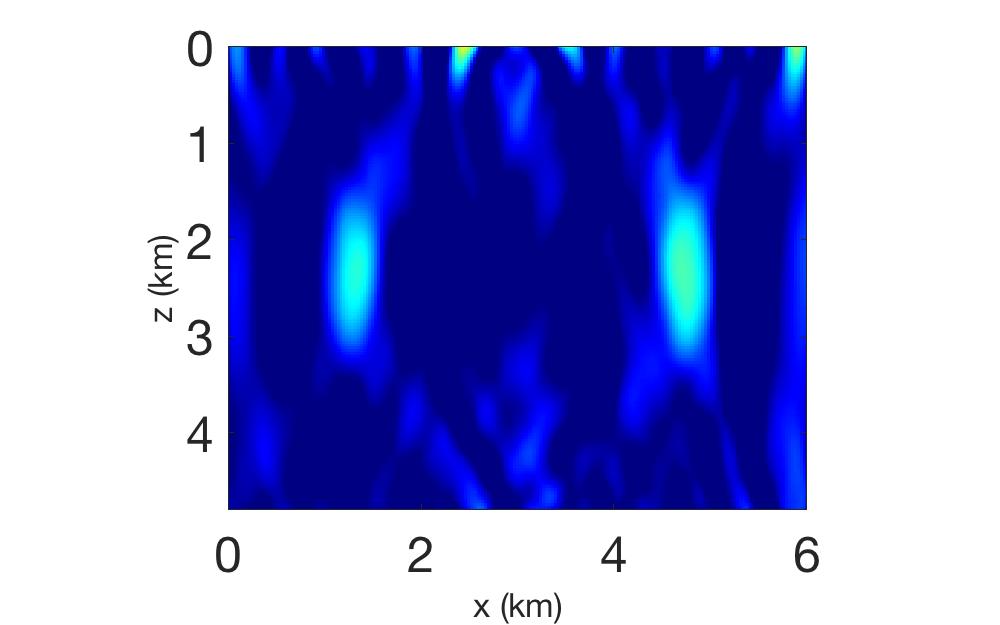}\label{fig:cheese-f2}}  \\
  \subfloat[$W_2$,
  linear]{\includegraphics[width=0.33\textwidth]{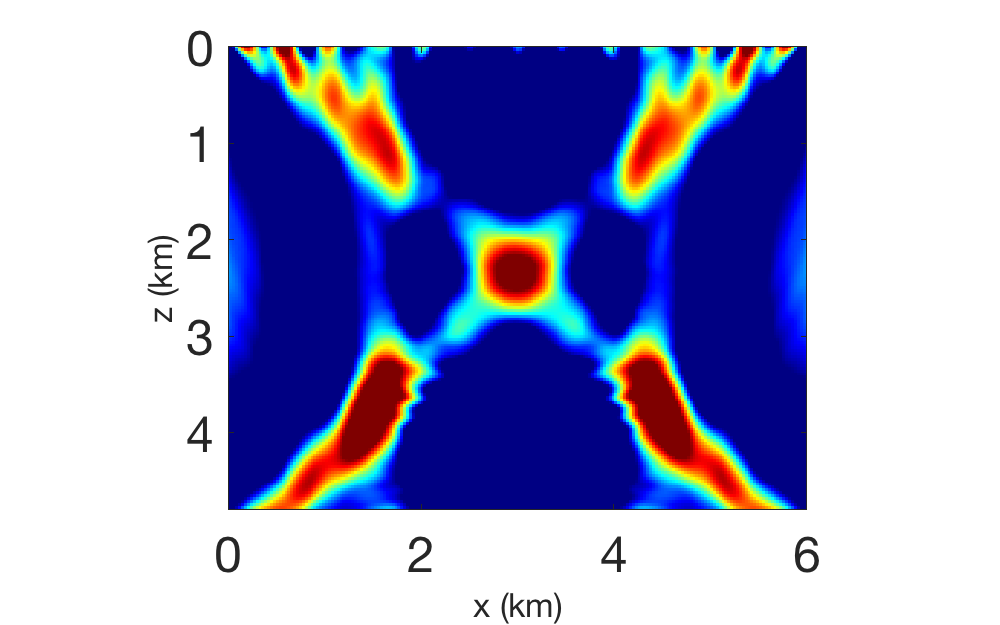}\label{fig:cheese-linear}}
  \subfloat[$W_2$,
  exponential]{\includegraphics[width=0.33\textwidth]{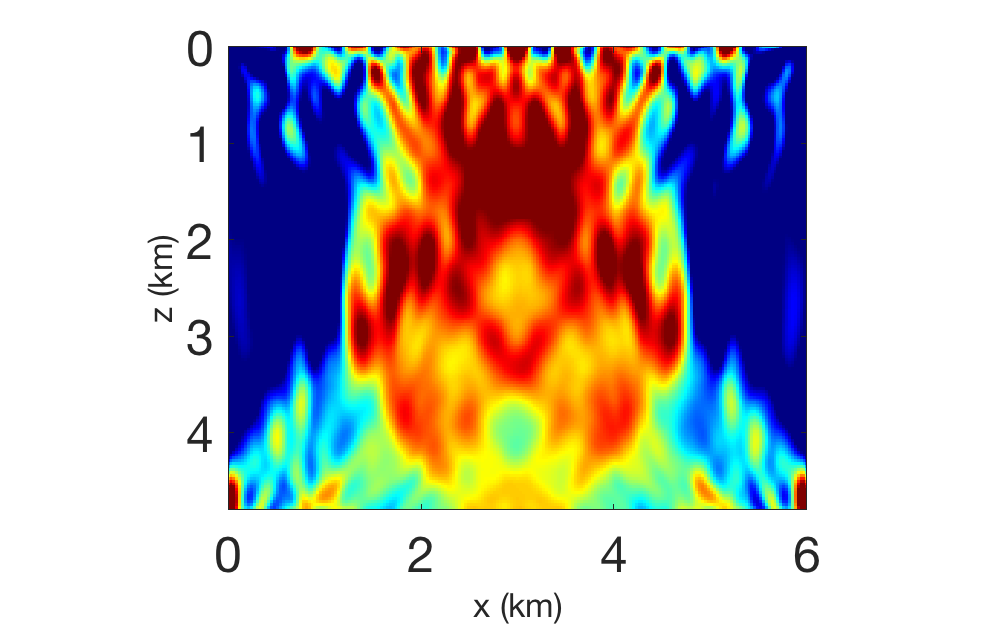}\label{fig:cheese-exp}}
  \subfloat[$W_2$,
  softplus]{\includegraphics[width=0.33\textwidth]{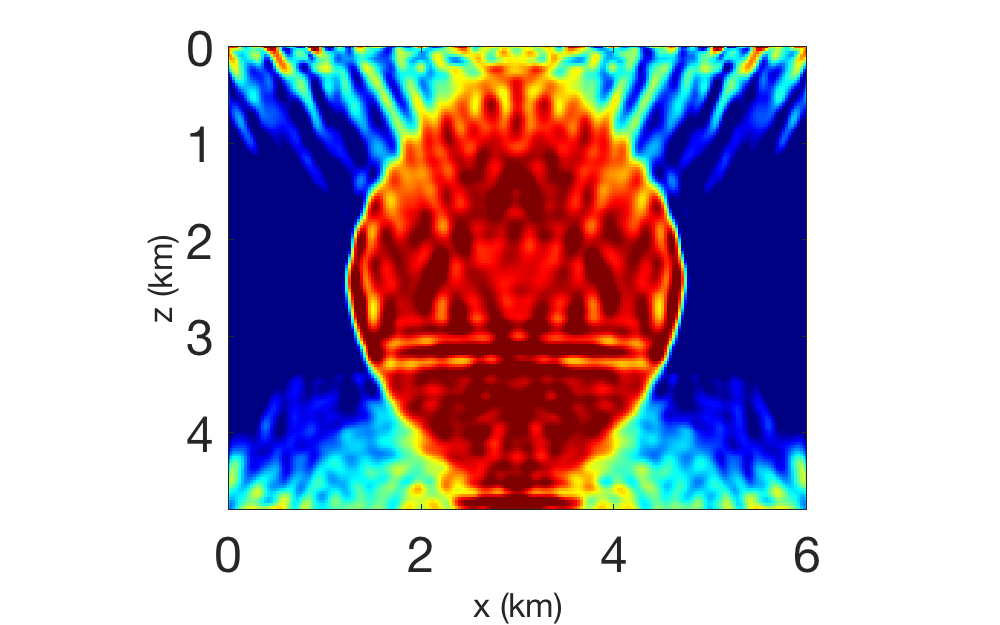}\label{fig:cheese-logi}}
  \caption{Velocity model with circular inclusion: (a) True velocity;
    (b) $L^2$ inversion; (c)--(f): $W_2$-based inversion using the
    square, the linear, the exponential ($b=0.2$), and the softplus
    scaling ($b=0.2$).}~\label{fig:cheese}
\end{figure}


\begin{figure}
  \centering \subfloat[$L^2$
  inversion]{\includegraphics[width=0.2\textwidth]{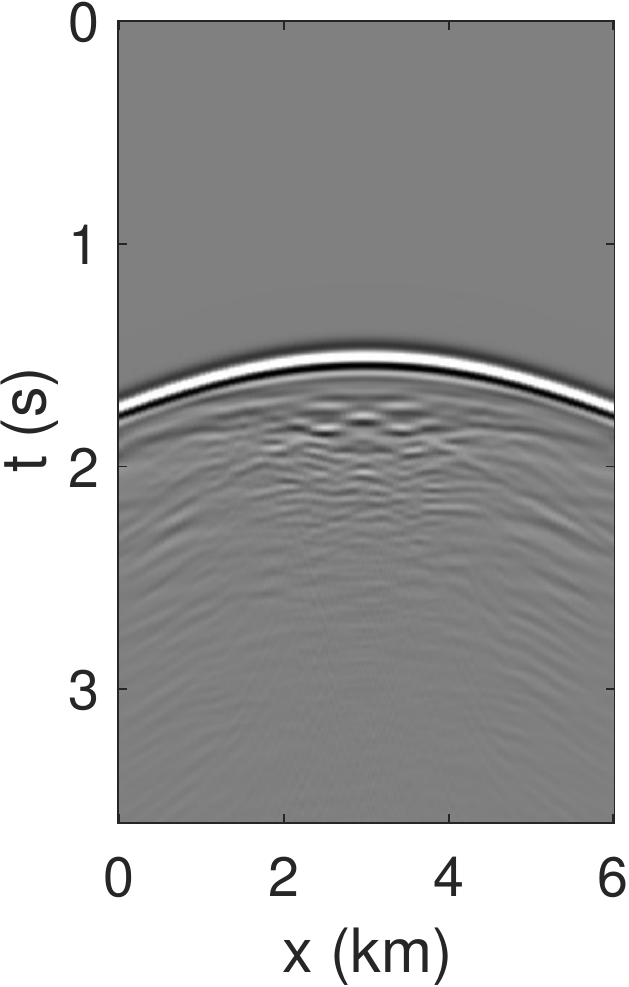}\label{fig:cheese-l2-res}}
  \subfloat[$W_2$,
  square]{\includegraphics[width=0.2\textwidth]{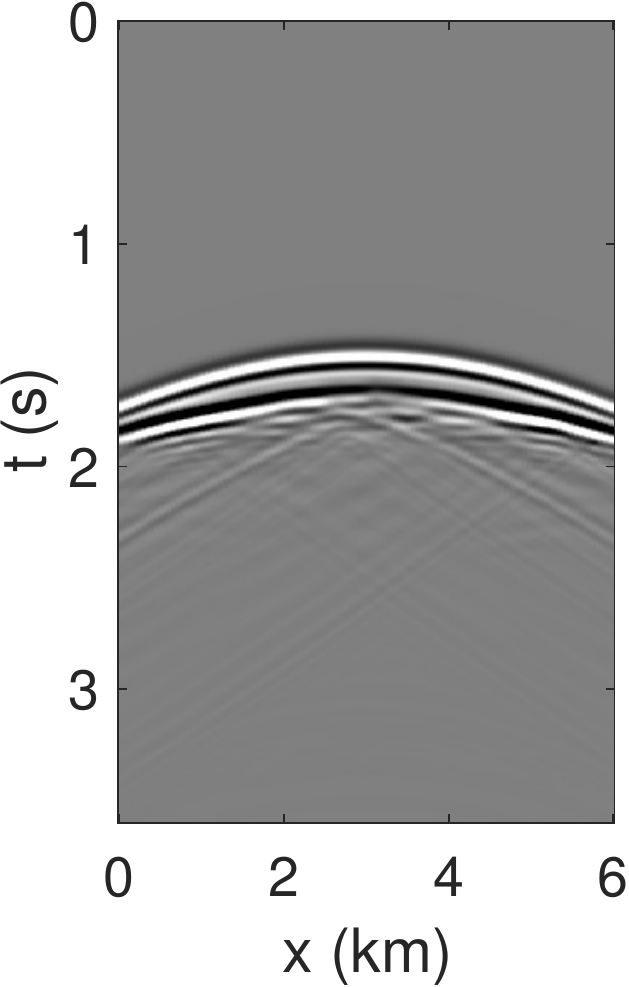}\label{fig:cheese-f2-res}}
  \subfloat[$W_2$,
  linear]{\includegraphics[width=0.2\textwidth]{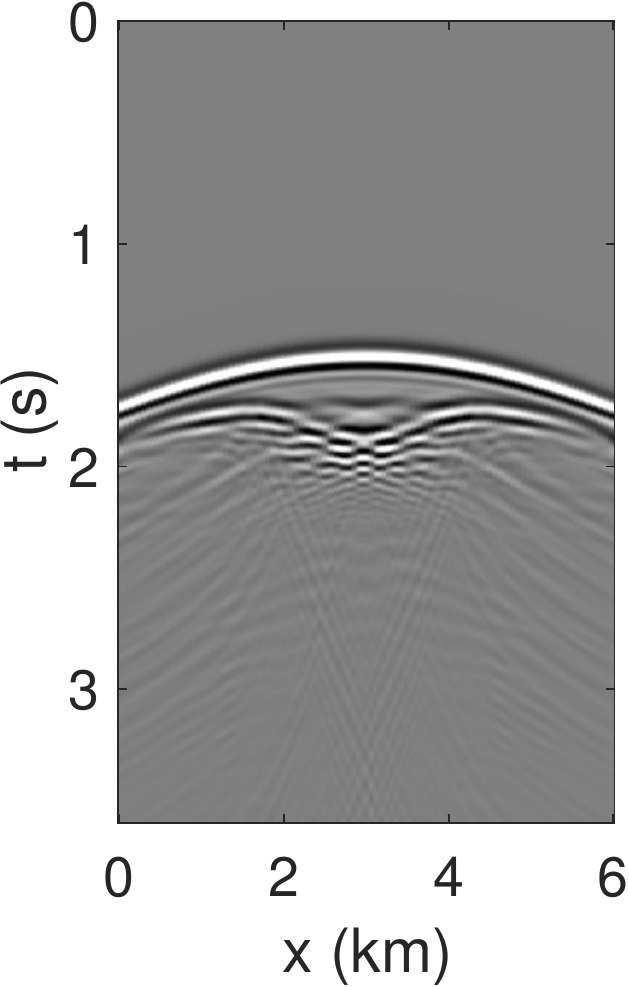}\label{fig:cheese-linear-res}}
  \subfloat[exponential]{\includegraphics[width=0.2\textwidth]{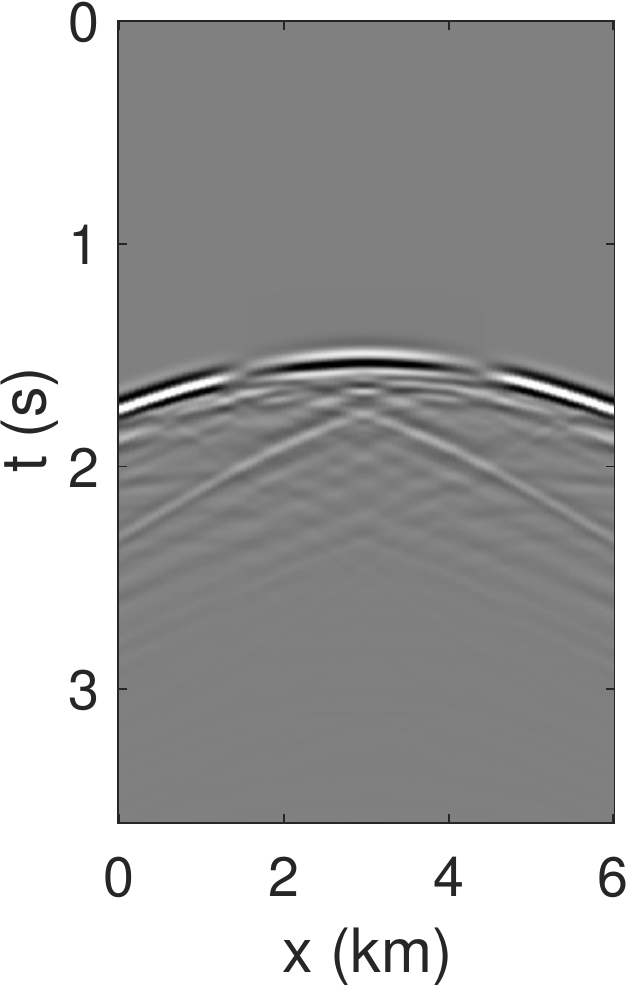}\label{fig:cheese-exp-res}}
  \subfloat[$W_2$,
  softplus]{\includegraphics[width=0.2\textwidth]{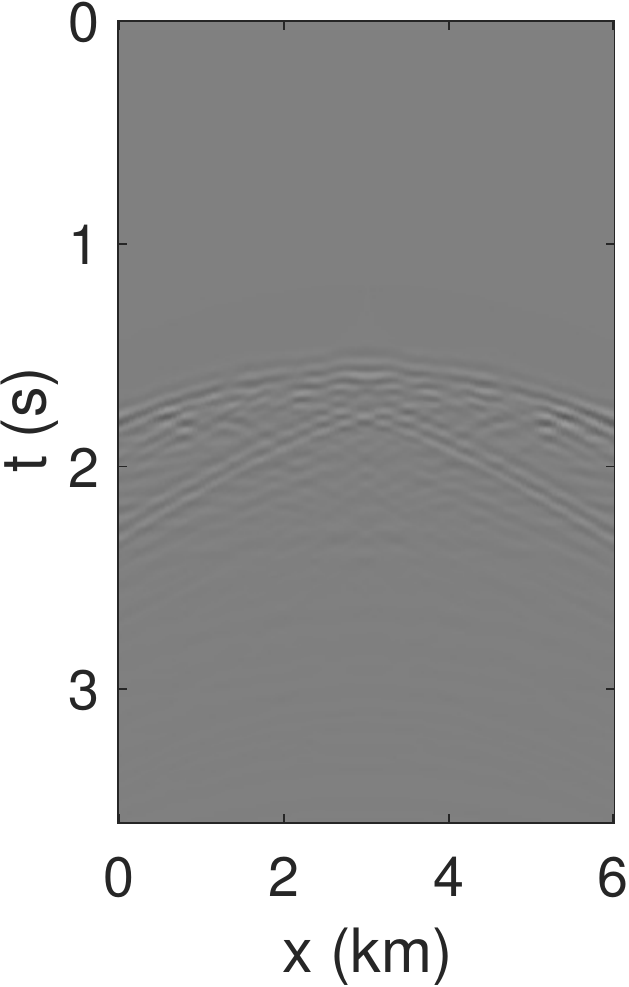}\label{fig:cheese-logi-res}}
  \caption{Circular inversion: the data fit in the final models for
    the $L^2$ inversion and the $W_2$ inversions with the square,
    linear, exponential, and softplus scalings. All figures are
    plotted under the same colormap. The initial data fit is presented
    in Figure~\ref{fig:ckdiff}.}~\label{fig:cheese-res}
\end{figure}

Figure~\ref{fig:cheese-f2} to Figure~\ref{fig:cheese-logi} present
inversion results by using the $W_2$ metric as the objective function,
but with different data normalization methods: the square
scaling~\cite{Survey2}, the linear scaling $\sigma_l$, the exponential
scaling $\sigma_e$, and the softplus scaling~$\sigma_s$,
respectively. As shown in Figure~\ref{fig:cheese-f2} and
Figure~\ref{fig:cheese-linear}, $W_2$-based inversion under the square
scaling and the linear scaling also suffer from cycle-skipping issues
whose final inversion results share similarities with the
reconstruction by the $L^2$ norm. Theoretically, the $W_2$ metric is
equipped with better convexity, but the data normalization step may
weaken the property if an improper scaling method is used. On the
other hand, the exponential scaling and the softplus scaling can keep
the convexity of the $W_2$ metric when applied to signed
functions. The additional hyperparameter $b$ in the scaling functions
helps to preserve the convexity; see Corollary~\ref{thm: softplus} and
the discussions in Section~\ref{sec:b}.

Exponential functions amplify both the signal and the noise
significantly, making the softplus scaling a more stable method,
especially when applied with the same hyperparameter $b$. In this
Camembert model, Figure~\ref{fig:cheese-logi} with the softplus
scaling and $b=0.2$ is the closest to the truth, while
Figure~\ref{fig:cheese-exp} by the exponential scaling with $b=0.2$
lacks good resolution around the bottom part. All the figures are
plotted under the same color scale. We think data normalization is the
most important issue for optimal-transport-based seismic inversion. We
have devoted the entire Section~\ref{sec:Data_Normalization} to this
important topic, and more developments in the optimal transport theory
are necessary to ultimately resolve the issue.

Finally, we present the data misfit of the converged models in
Figure~\ref{fig:cheese-res}, which are the differences between the
observed data and the synthetic data at convergence from different
methods. Compared with the initial data fit in
Figure~\ref{fig:ckdiff}, the square scaling for the $W_2$ inversion
hardly fits any data, while inversions with the $L^2$ norm and the
linear scaling reduce partial initial data residual. The exponential
and the softplus scaling are better in performance, while the latter
stands out for the best data fitting under the setup of this
experiment.

\begin{figure}[h]
  \centering
  \subfloat[Adjoint sources comparison at one trace]{\includegraphics[width=0.95\textwidth]{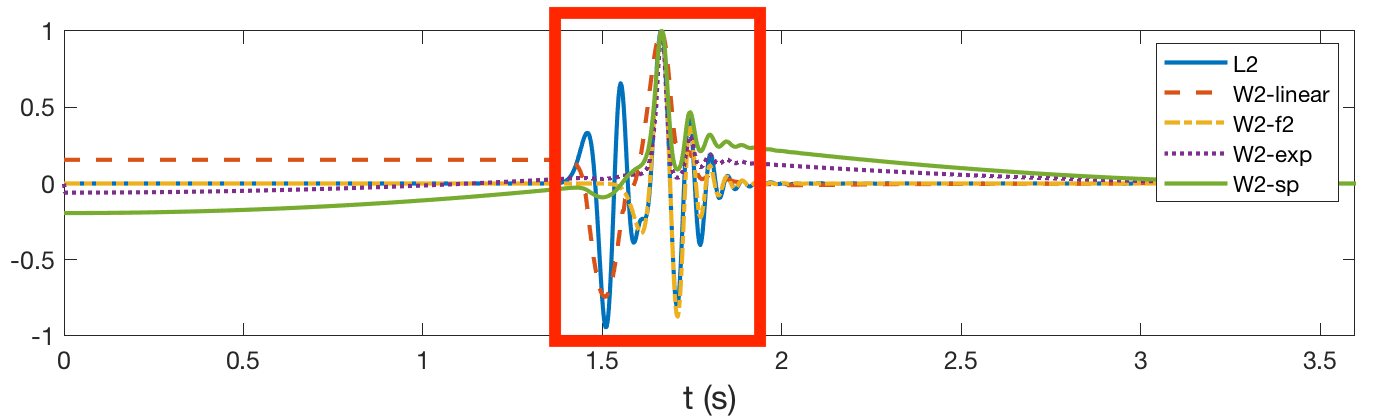}\label{fig:adjsrc1}} \\
  \subfloat[Detailed view inside the red box in
  (a)]{\includegraphics[width=0.95\textwidth]{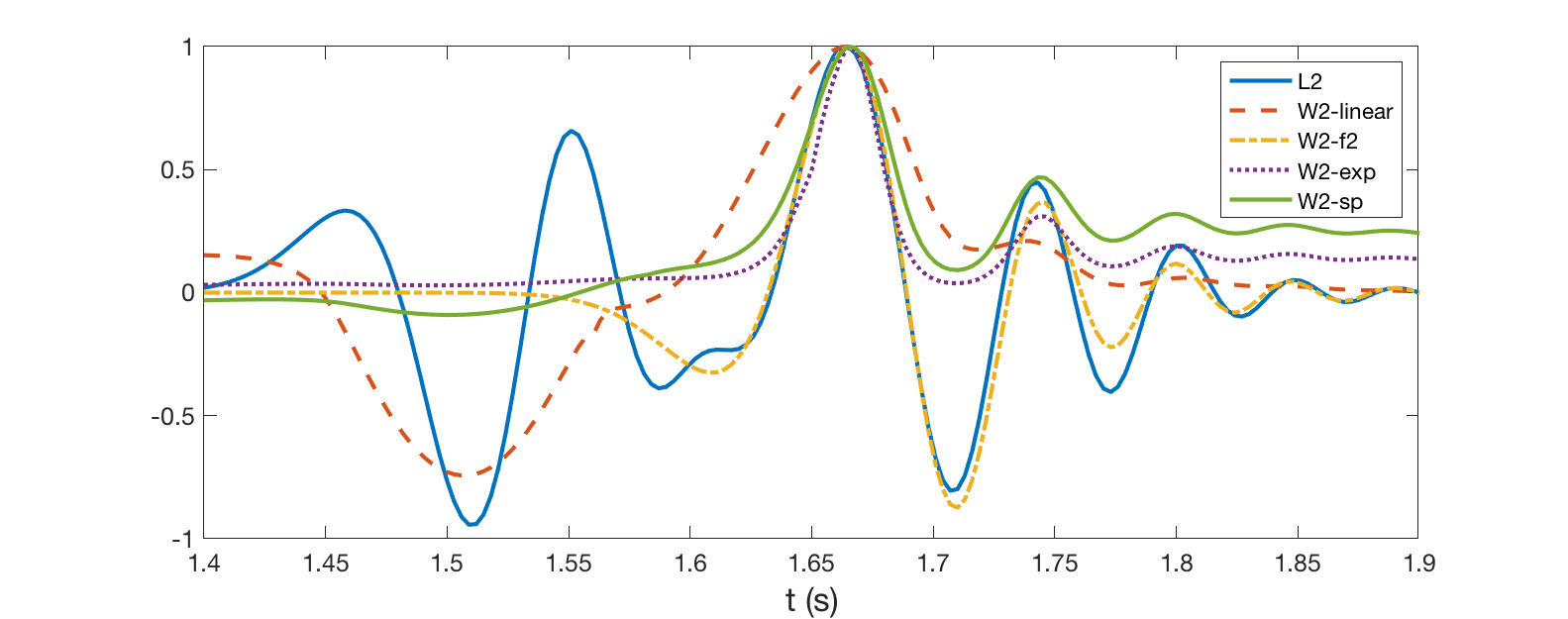}\label{fig:adjsrc2}}
  \caption{Circular inversion: comparison of normalized adjoint
    sources (the Fr\'echet derivative of the objective function with
    respect to the synthetic data) at the first iteration for the
    $L^2$ inversion and the $W_2$ inversion with the square, linear,
    exponential, and softplus scalings.}\label{fig:cheese-adjsrc}
\end{figure}

In Figure~\ref{fig:cheese-adjsrc}, we compare the adjoint source of
different methods at the first iteration. To better visualize the
differences, we focus on one of the traces and zoom into the wave-type
features. The $L^2$ adjoint source is simply the difference between
the observed and the synthetic signals, while the one by the linear
scaling is closer to its envelope. The enhancement in the
low-frequency contents matches our analysis in
Section~\ref{sec:Data_Normalization}. It also partially explains why
the linear scaling alone is often observed to mitigate the
cycle-skipping issues effectively. The adjoint source by the square
scaling has the oscillatory features of the seismic data. The
exponential and the softplus scaling methods share similar adjoint
sources at the first iteration of the inversion.

\begin{remark}
  Although inversion with the linear scaling fails in the
  Camembert example, it often works well in realistic cases where
  the recorded data contain various types of seismic waveforms; for
  example, see Figure~\ref{fig:test1_W2} and
  Figure~\ref{fig:Marm-W2}. For cases where the linear scaling
  struggles, one can turn to the softplus scaling~\cite{qiu2017full},
  as shown in Figure~\ref{fig:cheese}. With properly chosen
  hyperparameters, the softplus scaling keeps the convexity of the
  $W_2$ metric, which we proved in Corollary~\ref{thm:
    softplus}. Hence, the inversion process does not suffer from
  cycle-skipping issues.
\end{remark}

\subsection{The salt model}
In this section, we invert a more realistic model problem, which
represents the challenging deep-layer reconstruction with
reflections. It is an application of other improvements of the
$W_2$-based inversion discussed in Section~\ref{sec:three-layer} that
are beyond tackling local minima.

\begin{figure}[t]
  \centering
  \subfloat[The original, low-pass filtered and high-pass filtered true velocity]{\includegraphics[width=0.95\textwidth]{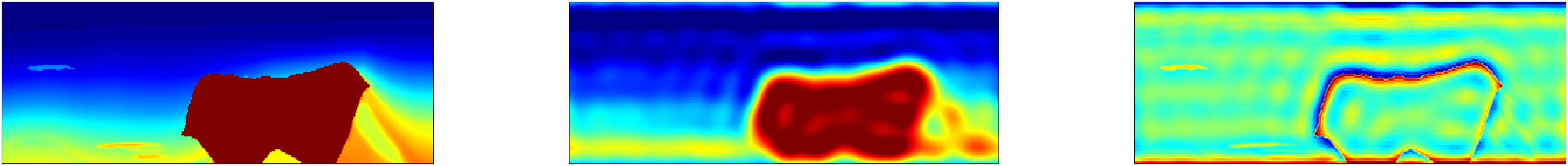}\label{fig:BP_vtrue}}\\
  \subfloat[The original, low-pass filtered and high-pass filtered initial velocity]{\includegraphics[width=0.95\textwidth]{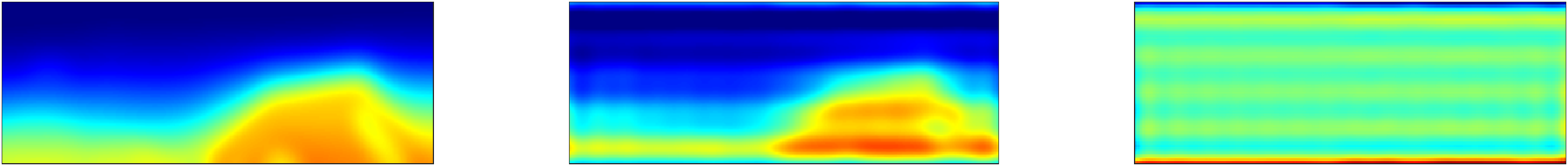}\label{fig:BP_v0}}  \\
  \caption{Salt reconstruction:~(a)~the true~and~(b)~the initial
    velocity.}
\end{figure}

We consider Figure~\ref{fig:BP_vtrue} as the true velocity model,
which is also part of the 2004 BP benchmark (2.8 km in depth and 9.35
km in width).  The model is representative of the complex geology in
the deepwater Gulf of Mexico. The main challenges in this area are
obtaining a precise delineation of the salt and recovering information
on the sub-salt velocity variations~\cite{billette20052004}. All
figures displayed here contain three parts: the original, the
low-pass, and the high-pass filtered velocity models.

The well-known velocity model with strongly reflecting salt inclusion
can be seen as a further investigation of
Section~\ref{sec:three-layer} in a more realistic setting. The
inversion results from these sharp discontinuities in velocity are the
same as the layered model in Section~\ref{sec:three-layer}. With many
reflections and refraction waves contributing to the inversion in the
more realistic models, it is harder to determine the most relevant
mechanisms. Different from the layered example in
Section~\ref{sec:three-layer}, the observed data here contains diving
waves, which are wavefronts continuously refracted upwards through the
Earth due to the presence of a vertical velocity gradient. However,
reflections still carry the essential information of the deep region
in the subsurface and are the driving force in the salt inclusion
recovery.

\begin{figure}[t]
  \centering
  \subfloat[The original, low-pass filtered and high-pass filtered $L^2$-FWI after 100 iterations]{\includegraphics[width=0.95\textwidth]{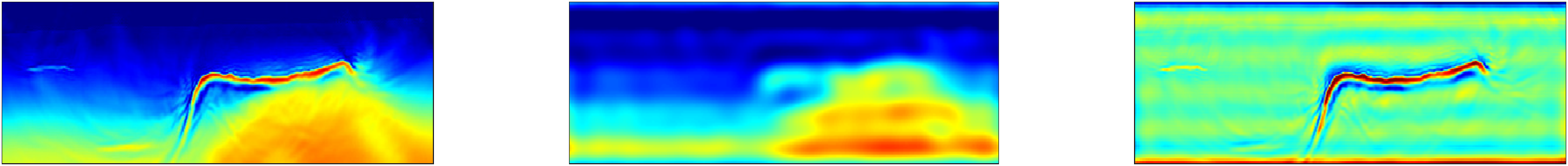}\label{fig:BP_L2_100}}   \\
  \subfloat[The original, low-pass filtered and high-pass filtered $W_2$-FWI after 100 iterations]{\includegraphics[width=0.95\textwidth]{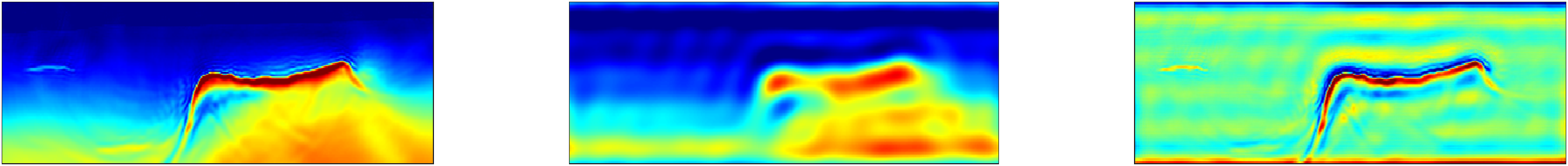}\label{fig:BP_W2_100}}   \\
  \caption{Salt reconstruction:~(a) the $L^2$-based~and~(b)~$W_2$-based
    inversion after 100 L-BFGS iterations.}
\end{figure}
\subsubsection{The synthetic setting}
The inversion starts from an initial model with the smoothed
background without the salt (Figure~\ref{fig:BP_v0}). We place 11
sources of 15 Hz Ricker wavelet and 375 receivers equally on the
top. The total recording time is 4 seconds. The observed data is
dominated by the reflection
\note{I changed  ``refection'' to \\``reflection''.}
from the top of the salt inclusion. After
100 iterations, FWI using both objective functions can detect the salt
upper boundary. Figure~\ref{fig:BP_L2_100} and
Figure~\ref{fig:BP_W2_100} also show that the high-wavenumber
components of the velocity models are partially recovered in both
cases, while the smooth components barely change. This is different
from the Marmousi example, where high-wavenumber components cannot be
correctly recovered ahead of the smooth modes.

With more iterations, the $W_2$-based inversion gradually recovers
most parts of the salt body (Figure~\ref{fig:BP_W2_600}), which is
much less the case for $L^2$-based inversion, as shown in
Figure~\ref{fig:BP_L2_600}. In particular, one can observe that a
wrong sublayer is created after 600 L-BFGS iterations in
Figure~\ref{fig:BP_L2_600}, from which one may have a misleading
interpretation about the Earth. Figure~\ref{fig:BP_W2_600} matches the
original salt body. In general, features related to the salt inclusion
can be better determined by using optimal-transport-related metrics as
the misfit function with the help of both refraction and reflection;
also see~\cite[fig.~4]{yangletter} for example.


\begin{figure}[t]
  \centering
  \subfloat[The original, low-pass filtered and high-pass filtered $L^2$-FWI after 600 iterations]{\includegraphics[width=0.95\textwidth]{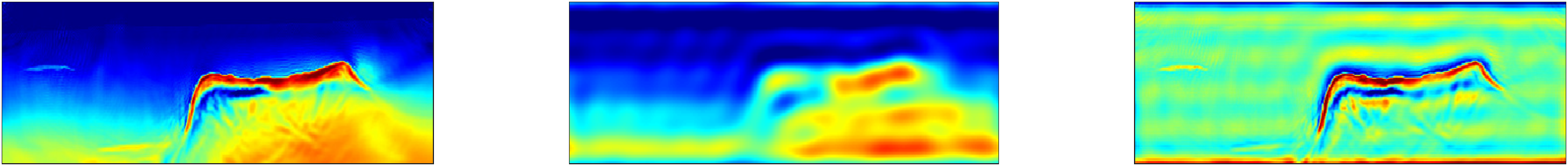}\label{fig:BP_L2_600}}   \\
  \subfloat[The original, low-pass filtered and high-pass filtered
  $W_2$-FWI after 600
  iterations]{\includegraphics[width=0.95\textwidth]{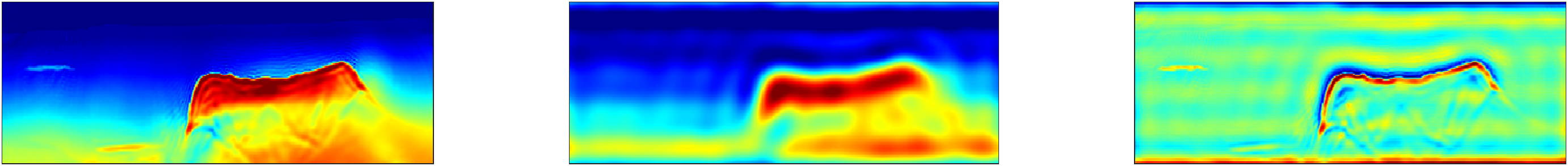}\label{fig:BP_W2_600}}
  \caption{Salt reconstruction:~(a) the
    $L^2$-based~and~(b)~$W_2$-based inversion after 600 L-BFGS
    iterations.}
  \vspace{-4ex}
\end{figure}

\begin{figure}[t]
  \centering
  \subfloat[Comparison between the clean data and the noisy data]{\includegraphics[scale=0.45]{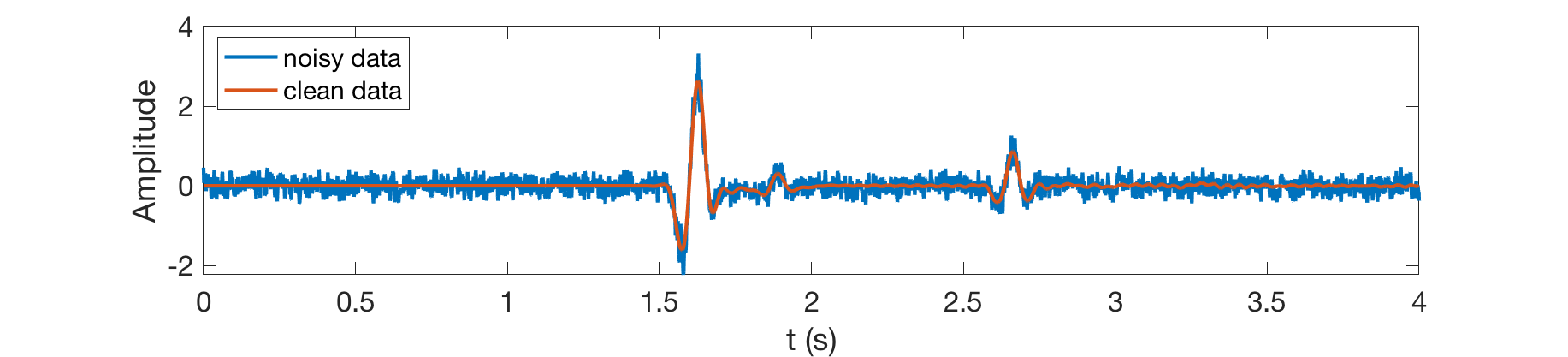}\label{fig:BP2-CRIME-data}}   \\
  \subfloat[$L^2$-based
  inversion]{\includegraphics[scale=0.4]{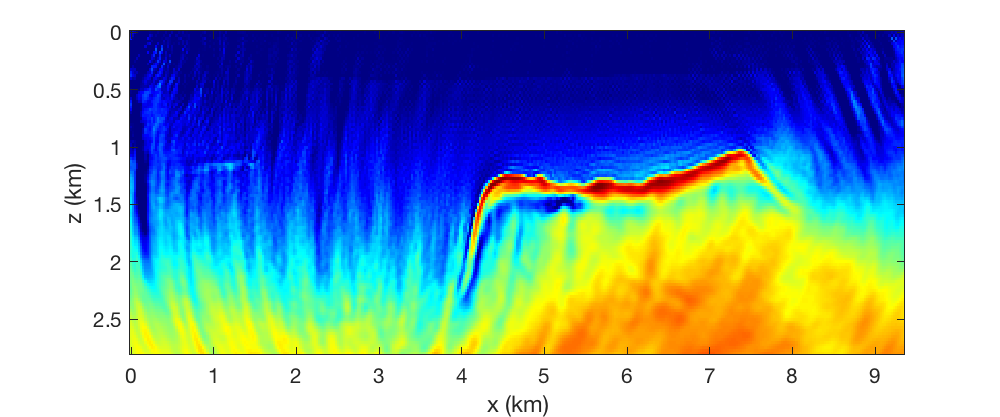}\label{fig:BP2-CRIME-L2}}
  \subfloat[$W_2$-based
  inversion]{\includegraphics[scale=0.4]{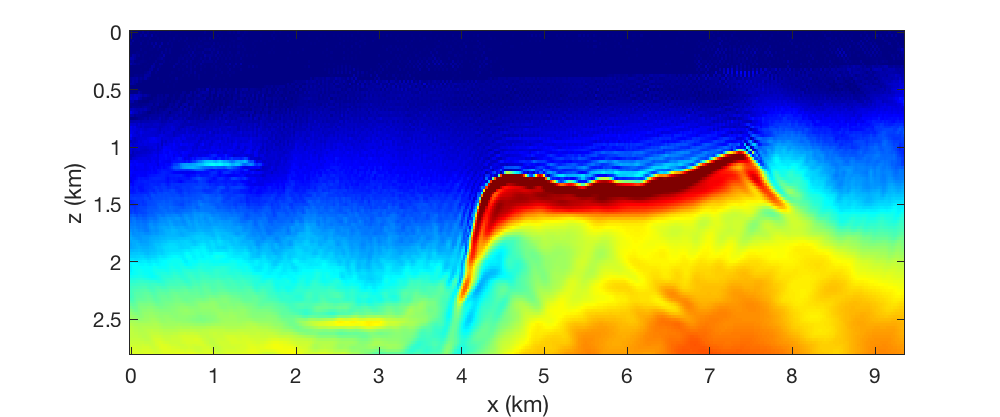}\label{fig:BP2-CRIME-W2}}
  \caption{Salt reconstruction under the realistic setting:~(a)~the
    comparison between the clean data and the noisy data, which is
    also generated by a refined mesh (b)~$L^2$-based inversion
    and~(c)~$W_2$-based inversion after convergence.  }
  \vspace{-3ex}
\end{figure}

\subsubsection{A more realistic setting}
Similar to~Section~\ref{sec:MarmReal}, we present another test where a
refined finite-difference mesh simulates the observed data, and thus a
different wave propagator is used in this inversion
test. Additionally, the reference data contains correlated mean-zero
noise. Other setups remain the same. Figure~\ref{fig:BP2-CRIME-data}
compares samples of the (noisy) observed data used in this test and
the one in the previous test. Figure~\ref{fig:BP2-CRIME-L2} is the
inversion using the $L^2$ norm as the objective function, while
Figure~\ref{fig:BP2-CRIME-W2} presents the inversion results based on
the $W_2$ metric where the linear scaling~\eqref{eq:linear} is used to
normalize the signals. Both tests are stopped after the L-BFGS
algorithm can no longer find a feasible descent direction. With the
presence of noise, both methods correctly reconstruct the upper
boundary of the salt body based on the phase information of the
reflections. However, amplitudes of the wave signals are corrupted by
the noise, which affects the reconstruction of the entire salt body.

The $W_2$-based inversion with the noisy data can still recover a
significant amount of the salt body. However, the thickness of the
layer is smaller than the one from the synthetic setting
(Figure~\ref{fig:BP_W2_600}). It is expected based on our discussion
on the small-amplitude sensitivity in
Section~\ref{sec:SmallAmplitude}. The energy reflected by the upper
boundary of the salt is minimal, but it is critical for the $W_2$
metric to reconstruct the model features below the reflecting
interface in the previous synthetic setting. The small reflection can
be obscured by the presence of noise and modeling error under this
realistic setting.

Nevertheless, the $W_2$-based reconstruction has almost no footprint
from the noise except the salt body, while the $L^2$-based inversion
has a distinct noise pattern; see Figure~\ref{fig:BP2-CRIME-L2}. It is
an interesting phenomenon observed from all numerical tests of
$W_2$-based inversion. Instead of overfitting the noise or the wrong
information in the data, $W_2$-based inversion stops with no feasible
descent direction, while $L^2$-based inversion continues updating the
model parameter, but mainly fits the noise and numerical errors, which
in turn gives noisy and incorrect reconstructions. It is another
demonstration of the good stability of the $W_2$ metric for
data-fitting problems, as discussed in~\cite{ERY2019}.

\section{Conclusion}
In this paper, we have analyzed several favorable new properties of
the quadratic Wasserstein distance connected to seismic inversion,
compared to the standard $L^2$ techniques. We have presented a sharper
convexity theorem regarding both translation and dilation changes in
the signal. The improved theorem offers a more solid theoretical
foundation for the wide range of successful field data inversions
conducted in the exploration industry. It shows why trapping in local
minima, the so-called cycle skipping, is avoided, and the analysis
gives guidance to algorithmic developments.

Data normalization is a central component of the paper. It has always
been a limitation for optimal transport applications, including
seismic imaging, which needs to compare signed signals where the
requirement of nonnegativity and the notion of equal mass are not
natural. We study different normalization methods and define a class
with attractive properties. Adding a buffer constant turns out to be
essential. In a sequence of theorems, we show that the resulting
relevant functional for optimization is a metric, and the normalized
signals are more regular than the original ones. We also prove a
Huber-norm type of property, which reduces the influence of seismic
events that are far apart and should not affect the optimization. The
analysis here explains the earlier contradictory
observations~\cite{yang2017application} that linear normalization
often works better in applications than many other scaling methods,
even if it lacks convexity with respect to shifts~\cite{Survey2}.

The final contribution of the paper is to present and analyze the
remarkable capacity of the $W_2$-based inversion of sublayer recovery
with only the reflection data even when there is no seismic wave
returning to the surface from this domain. The conventional $L^2$ norm
does not perform well, and here it is not the issue of cycle
skipping. Both amplitude and frequency play a role. Compared to the
$W_2$ metric, the $L^2$ norm lacks sensitivity to small-amplitude
signals. We saw this in numerical tests and from classical refraction
analysis. The inherent insensitivity of the $L^2$ norm to
low-frequency contents of the residual is a primary reason 
that $L^2$-FWI often fails to recover the model
kinematics. $W_2$-based inversion captures the essential low-frequency
modes of the data residual, directly linked to the low-wavenumber
structures of the velocity model. This property is important for
applications of this type, where the initial model has poor background
velocity.

With this paper, the mathematical reasons for the favorable properties
of $W_2$-based inversion become quite clear. There are still several
other issues worth studying that could have important practical
implications. Examples are further analysis of other forms of optimal
transport techniques as, for example, unbalanced optimal transport and
$W_1$. The scalar wave equation is the dominating model in practice,
but elastic wave equations and other more realistic models are gaining
ground, and extending the analysis and best practices to these models
will be essential.

        





\ack We appreciate the comments and suggestions of the anonymous
referee(s). This work is supported in part by the National Science
Foundation through grants DMS-1620396 and DMS-1913129. We thank
Dr.~Sergey Fomel, Dr.~Lingyun Qiu, and Dr.~Kui Ren for constructive
discussions and the sponsors of the Texas Consortium for Computational
Seismology (TCCS) for financial support.



\begin{coninfo}
\vspace{-.63\baselineskip}\textsc{Yunan Yang}\\
Courant Institute\\
251 Mercer St., Office 1105A\\
New York, NY 10012\\
USA\\
  \eaddr{yunan.yang@nyu.edu}
  \pagebreak\\
  \textsc{Bj\"orn Engquist}\\
  Department of Mathematics\\
  \null\quad  and Oden Institute\\
  The University of Texas at Austin\\
  1 University Station C1200\\
  Austin, TX 78712\\
  USA\\
  \eaddr{engquist@\\oden.utexas.edu}
\end{coninfo}
\end{document}